\documentclass[aap]{imsart}

\RequirePackage{amsthm,amsmath,amsfonts,amssymb}
\RequirePackage[authoryear]{natbib}
\RequirePackage[colorlinks,citecolor=blue,urlcolor=blue]{hyperref}
\RequirePackage{graphicx}

\startlocaldefs

\usepackage{url}
\usepackage{mathtools}
\usepackage{empheq}
\usepackage{latexsym}
\usepackage{enumitem}
\usepackage{eurosym}
\usepackage{dsfont}
\usepackage{appendix}
\usepackage{tcolorbox}
\usepackage{xcolor} 
\usepackage[utf8]{inputenc}
\usepackage[T1]{fontenc}
\usepackage{multirow}
\usepackage{todonotes}
\usepackage{lmodern}
\usepackage{anyfontsize}
\usepackage{stmaryrd}
\usepackage{cleveref}
\usepackage{scalerel}
\usepackage{comment}
\usepackage[english]{babel}
\usepackage[english=british]{csquotes}
\usepackage{tikz}
\usepackage{mathrsfs}
\usepackage{upgreek}
\usepackage{booktabs}
\usepackage{subcaption}
\usepackage[justification=centering]{caption}
\usepackage{array}
\usepackage{scalefnt}
\usepackage[nottoc]{tocbibind}

\newcommand{\smallfont}[1]{\text{\fontsize{4}{4}\selectfont$#1$}}
\newcommand{\tinyfont}[1]{\text{\fontsize{3}{3}\selectfont$#1$}}

\definecolor{red}{rgb}{0.7,0.15,0.15}
\colorlet{RED}{red}
\definecolor{green}{rgb}{0,0.5,0}
\definecolor{blue}{rgb}{0,0,0.7}

\makeatletter \@addtoreset{equation}{section}

\theoremstyle{plain}

\newtheorem{theorem}{Theorem}[section]

\newtheorem{lemma}[theorem]{Lemma}
\newtheorem{proposition}[theorem]{Proposition}

\theoremstyle{definition}
\newtheorem{definition}[theorem]{Definition}

\newtheorem{remark}[theorem]{Remark}
\newtheorem{assumption}[theorem]{Assumption}

\def \E{\mathbb{E}}
\def \F{\mathbb{F}}
\def \G{\mathbb{G}}
\def \H{\mathbb{H}}
\def \I{\mathbb{I}}

\def \L{\mathbb{L}}

\def \N{\mathbb{N}}
\def \P{\mathbb{P}}

\def \R{\mathbb{R}}
\def \S{\mathbb{S}}

\def\Ac{{\cal A}}
\def\Bc{{\cal B}}
\def\Cc{{\cal C}}

\def\Fc{{\cal F}}

\def\Ic{{\cal I}}

\def\Kc{{\cal K}}
\def\Lc{{\cal L}}

\def\Nc{{\cal N}}
\def\Oc{{\cal O}}
\def\Pc{{\cal P}}

\def\Rc{{\cal R}}

\def\Tc{{\cal T}}

\def\Zc{{\cal Z}}

\def\eps{\varepsilon}
\def\erm{\mathrm{e}}
\def\drm{\mathrm{d}}
\def\t{\top\!}
\DeclareMathOperator*{\argmax}{argmax}

\DeclareMathOperator*{\essinf}{{essinf}^\P}
\DeclareMathOperator*{\esssup}{esssup}

\def\Cf{{\mathfrak C}}
\newcommand{\xdim}{d}
\newcommand{\bmdim}{n}
\def\d{\mathrm{d}}
\def\as{\text{\rm--a.s.}}
\def\ae{\text{\rm--a.e.}}
\def\qs{\text{\rm--q.s.}}

\def\e{{\mathrm{e}}}
\def\sigmah{\widehat{\sigma}}

\newcommand{\newinf}{\mathop{\mathrm{inf}\vphantom{\mathrm{sup}}}}
\endlocaldefs

\begin{document}

\begin{frontmatter}
\title{Closed-loop equilibria for Stackelberg games: a story about stochastic targets}
\runtitle{Stochastic target approach for Stackelberg games}

\begin{aug}
\author[A]{\fnms{Camilo}~\snm{Hern\'andez}\ead[label=e1]{camilohe@usc.edu}\orcid{0000-0003-0123-5963}},
\author[B]{\fnms{Nicol\'as}~\snm{Hern\'andez-Santib\'a\~{n}ez}\ead[label=e2]{nicolas.hernandezs@usm.cl}\orcid{0000-0002-5202-7155}},
\author[C]{\fnms{Emma}~\snm{Hubert}\ead[label=e3]{emma.hubert@dauphine.psl.eu}\orcid{0000-0001-6058-6953}},
\and
\author[D]{\fnms{Dylan}~\snm{Possama\"{i}}\ead[label=e4]{dylan.possamai@math.ethz.ch}\orcid{0000-0002-9364-0124}}

\address[A]{ISE Department, University of Southern California\printead[presep={,\ }]{e1}}

\address[B]{Departamento de Matem\'atica, Universidad T\'ecnica Federico Santa Mar\'ia\printead[presep={,\ }]{e2}}

\address[C]{CEREMADE, Université Paris Dauphine \& PSL\printead[presep={,\ }]{e3}}

\address[D]{Department of Mathematics, ETH Z\"{u}rich\printead[presep={,\ }]{e4}}
\end{aug}

\begin{abstract}
We provide a general approach to reformulating any continuous-time stochastic Stackelberg differential game under \emph{closed-loop strategies} as a single-level optimisation problem with target constraints. More precisely, we consider a Stackelberg game in which the leader and the follower can both control the drift and the volatility of a stochastic output process, in order to maximise their respective expected utility. The aim is to characterise the Stackelberg equilibrium when the players adopt `closed-loop strategies', \textit{i.e.} their decisions are based \emph{solely} on the historical information of the output process, excluding especially any direct dependence on the underlying driving noise, often unobservable in real-world applications.
We first show that, by considering the---second-order---backward stochastic differential equation associated with the continuation utility of the follower as a controlled state variable for the leader, the latter's unconventional optimisation problem can be reformulated as a more standard stochastic control problem with target constraints. 
Thereafter, adapting the methodology developed by \cite{soner2002dynamic} or \cite{bouchard2010optimal}, the optimal strategies, as well as the corresponding value of the Stackelberg equilibrium, can be characterised through the solution of a well-specified system of Hamilton--Jacobi--Bellman equations. 
For a more comprehensive insight, we illustrate our approach through a simple example, facilitating both theoretical and numerical detailed comparisons with the solutions under different information structures studied in the literature.	
\end{abstract}



\end{frontmatter}
\tableofcontents

\section{Introduction}

The concept of hierarchical or bi-level solutions for games was first introduced by \cite{stackelberg1934marktform} to describe market situations where certain firms wield dominance over others. In the straightforward context of a two-player non--zero-sum game, this solution concept, now widely recognised as the \emph{Stackelberg equilibrium}, illustrates a scenario where one player, designated as the leader (she), declares her strategy first. Subsequently, the second player, known as the follower (he), optimally adjusts his strategy in response to the leader’s initial move. To ascertain her optimal strategy, the leader must anticipate the follower’s reaction to any given strategy, thereby selecting the one that maximises her reward function in light of the follower’s best response. Thus, a Stackelberg equilibrium is defined by the combination of the leader’s optimal action and the follower’s rational response to that action. This solution concept is particularly pertinent in situations where players possess asymmetric power, akin to the original market conditions described by \citeauthor*{stackelberg1934marktform}, or when one player holds more information than the other. For instance, Stackelberg equilibria naturally emerge in games where only one player has knowledge of both players’ cost or reward functions, or when one player is more time-efficient in determining her optimal strategy.

\paragraph*{Dynamic Stackelberg games} After its introduction, this equilibrium concept has been thoroughly studied in static competitive economics, but the mathematical treatment of its dynamic version was not developed until the 70s, first in discrete-time models by \cite{cruz1975survey,cruz1976stackelberg},
\cite{gardner1977feedback},
\cite{basar1978new, basar1979closedloop},
and then, more interestingly for us, in continuous-time ones by 
\cite{chen1972stackelberg},  \cite{simaan1973additional,simaan1973stackelberg,simaan1976stackelberg}, \cite{papavassilopoulos1979nonclassical,papavassilopoulos1980sufficient},
\cite{papavassilopoulos1979leader},
\cite{basar1980teamoptimal},
\cite{basar1981new},
\cite{bagchi1984stackelberg}.

\smallskip

For instance, \cite{chen1972stackelberg} investigate Stackelberg solutions for a two-player non--zero-sum dynamic game with finite horizon $T>0$, in which both players can observe the state $X$ and its dynamics, but only the leader knows both reward functions. In their model, the leader first chooses her control $\alpha\in \Ac$ to maximise her reward function $J_{\rm L}$, and then the follower wishes to maximise his reward function $J_{\rm F}$ by choosing his own control $\beta\in \Bc$, given admissibility sets $\Ac$ and $\Bc$.
In this dynamic setting, the reward functions take the form
\begin{align*}
	J_{\rm L} (\alpha, \beta) \coloneqq  g_{\rm L} (X_T) + \int_{0}^{T} f_{\rm L} (t,X_t,\alpha_t,\beta_t )\d t, \;
	J_{\rm F} (\alpha, \beta) \coloneqq  g_{\rm F} (X_T) + \int_{0}^{T} f_{\rm F} (t,X_t,\alpha_t,\beta_t )\d t,
\end{align*}
subject to the following dynamics for the state process
\begin{align*}
	\drm X_t = \lambda (t,X_t,\alpha_t, \beta_t) \drm t, \; t \in [0,T], \; X_{0} =  x_0.
\end{align*}
A strategy $(\alpha^\star, \beta^\star)$ is called a \textit{Stackelberg equilibrium} if, for any $\alpha \in \Ac$
\begin{align*}
	J_{\rm L} (\alpha^\star,\beta^\star) \geq J_{\rm L}(\alpha,b^\star(\alpha)), \; \text{where} \; b^\star (\alpha) \coloneqq  \argmax_{\beta \in \Bc} \big\{J_{\rm F}(\alpha,\beta)\big\},\; \text{and} \; \beta^\star \coloneqq  b^\star (\alpha^\star).
\end{align*}
More importantly, \cite{chen1972stackelberg} introduce two nuanced refinements of the Stackelberg solution concept, contingent on the information accessible to the players: the \emph{open-loop} strategy, where decisions are made at time zero based on the initial state, and the \emph{feedback} strategy, where decisions at time $t$ are influenced solely by the current state. These distinctions lead to fundamentally different solutions, as the strategies diverge in their formulation and execution. This categorisation of Stackelberg equilibria has been pivotal in subsequent literature, particularly in the realm of stochastic dynamic Stackelberg games. It is unsurprising, then, that this framework will form the cornerstone of our analysis in this paper.

\paragraph*{Stochastic Stackelberg games} The pioneering exploration of stochastic versions of Stackelberg games dates back to the late 1970s, marked by the discrete-time models of \cite{castanon1976equilibria}, \cite{basar1979stochastic}, and \cite{basar1984feedback}. \cite[Chapter 7]{basar1999dynamic} provide a comprehensive overview of Stackelberg game theory at that time, encompassing static, deterministic discrete- and continuous-time, and stochastic discrete-time frameworks. However, it was not until the seminal work of \cite{yong2002leader}  that the literature began to incorporate continuous-time stochastic models. In this advanced framework, the output process is elegantly described as the solution to a stochastic differential equation of the following form
\begin{align}\label{eq:general_dynamic}
	\d X_t = \sigma(t,X_t,\alpha_t,\beta_t) \big( \lambda(t,X_t,\alpha_t,\beta_t) \d t + \d W_t \big), \; t \in [0,T],\; X_{0} =  x_0,
\end{align}
where $W$ is a Brownian motion, and the controls $\alpha$ and $\beta$ are chosen by the leader and the follower, respectively.
As previously mentioned, the information available to the players plays a pivotal role in determining the solution concept. \cite{yong2002leader} utilises the stochastic maximum principle to derive the open-loop solution for a linear--quadratic Stackelberg game, wherein both players can manipulate the drift and volatility of the state variable. Open-loop solutions have also been examined by researchers such as \cite{oksendal2013stochastic} and \cite{moon2021linear} in jump-diffusion models, and by \cite{shi2016leader} in a linear--quadratic framework characterised by asymmetric information. Concurrently, feedback solutions have been explored through the dynamic programming approach. Notable examples include \cite{he2008cooperative} who study a cooperative advertising and pricing problem, and \cite{bensoussan2014feedback} who investigate an infinite-horizon model. This methodology was further refined by \cite{huang2021verification} to address a finite-horizon problem with volatility control.
\smallskip

Similar to Nash equilibrium concepts, one can also consider so-called \textit{closed-loop} Stackelberg solutions, where the strategies of both players can depend on \textit{the trajectory of the state variable}. However, as noted by \cite{basar1999dynamic} and \cite{simaan1973additional}, closed-loop equilibria are notoriously hard to study, even in simple dynamic games. 
One work in this direction is \cite{bensoussan2015maximum}, which extends the stochastic maximum principle approach to characterise \textit{adapted} closed-loop \textit{memoryless} Stackelberg solutions and, in a linear--quadratic framework, provides a comparison with the open-loop equilibrium. \cite{li2023closed,li2023linear} also discuss `closed-loop solvability' within a linear--quadratic framework, but also restrict to memoryless strategies, and their approach is thus similar to the one developed previously by \cite{bensoussan2015maximum}. Additionally, \cite{li2022closed} investigate closed-loop strategies with one-step memory within a deterministic and discrete-time setting, adding another layer of depth to the understanding of closed-loop equilibria.

\smallskip 
While we defer to \Cref{ss:different_equilibrium} the precise definitions of open-loop, feedback, and closed-loop Stackelberg solutions in a stochastic continuous-time framework, as well as a comparison of these concepts through a simple example, we emphasise that, to the best of our knowledge, there is no literature on stochastic Stackelberg games in which the players' strategies are allowed to depend on the whole trajectory of the output process. 
One goal of this paper is precisely to fill this gap in the literature: we develop an approach that allows us to characterise Stackelberg equilibria with \textit{general} closed-loop strategies, in the sense that both the leader's and follower's strategies can depend on the trajectory of the state variable up to the current time, as opposed to the memoryless strategies considered in \cite{bensoussan2015maximum,li2023closed,li2023linear}.

\paragraph*{Extensions and applications} Before describing our approach and results in more details, one should mention that there are now many extensions and generalisations of the traditional leader--follower game, such as zero-sum solutions, mixed leadership, control of backward SDEs, learning problems, large-scale games, and the mean-field setting, among others.\footnote{See \cite{sun2023zero} for zero-sum games, \cite{bensoussan2019feedback} for mixed leadership,  \cite{zheng2020stackelberg,zheng2021stackelberg} and \cite{feng2022backward} for the case where the controlled state dynamics is given by a backward SDE, \cite{li2023solving} and \cite{zheng2022linear,zheng2022stackelberg} for learning games and \cite{ni2023deterministic} for the study of the time-inconsistency of open-loop solutions. As for larger-scale games, we mention \cite{li2018forward} for the study of repeated Stackelberg games, in which a follower is also the leader of another game, and \cite{kang2022three} for a three-level game. 
	The case of one leader and many followers, originally introduced in a static game by \cite{leitmann1978generalized} and in a stochastic framework by \cite{wang2020asymmetric}, \cite{vasal2022sequential}, has been extended to the mean-field setting 
	in \cite{fu2020mean}, \cite{aid2020mckean}, \cite{si2021backward}, \cite{vasal2022master}, \cite{lv2023linear}, \cite{li2023closed}, \cite{gou2023linear}, \cite{dayanikli2023machine}, \cite{cong2024direct}.}
Lastly, we remark that Stackelberg games cover a wide range of applications, from original economic models, as highlighted by \cite{bagchi1984stackelberg} and \cite{van2010survey}, to operation research and management science, as reviewed by \cite{li2017review} and \cite{dockner2000differential}. Specific applications in these areas include, but are not limited to, marketing channels as in \cite{he2007survey}, cooperative advertising as in \cite{chutani2014feedback}, \cite{he2008cooperative}, insurance as in \cite{havrylenko2022risk}, \cite{han2024optimal}, \cite{guan2024stackelberg}, and energy generation as in \cite{aid2020mckean}.

\paragraph*{A `new' Stackelberg solution concept} In this paper, we consider a stochastic Stackelberg game in continuous time with two players, a leader and a follower, both of whom can control the drift and volatility of the output process $X$, whose dynamics take the general form \eqref{eq:general_dynamic}. Our main theoretical result characterises the Stackelberg equilibrium when the strategies of both players are \emph{closed-loop}, in the sense that their strategies can depend on time and on the path of the output process $X$. More precisely, we allow both players to build strategies whose value at time $t \in [0,T]$ can be a function of time $t$ but more importantly of the trajectory of the process $X$ up to time $t$, denoted $X_{\cdot \wedge t}$. In particular, under this information concept, the players' decisions cannot directly depend on the underlying driving noise. As already emphasised, to our knowledge only the four aforementioned papers \cite{bensoussan2015maximum,li2023closed,li2023linear,li2022closed} study Stackelberg equilibria for strategies falling into the class of `closed-loop'. However, the first three papers focus on the \emph{memoryless case}, in the sense that the admissible strategies at time $t$ do not actually depend on the trajectory of the process up to time $t$, but only on the value of the process at that time, namely $X_t$. The last paper \cite{li2022closed} introduces a notion of memory by allowing the strategy at time $t$ to depend on $X_t$ and $X_{t-1}$, even though in a deterministic and discrete-time framework. The authors nevertheless show that strategies with one-step memory may lead, even in simple frameworks, to different equilibria compared to their memoryless counterparts, which thus provides a first motivation to study a form of `pathwise' (as opposed to memoryless) closed-loop strategies.
\smallskip

Beyond the distinction between `memoryless' and `pathwise' closed-loop strategies, another significant difference of our solution concept comparing to \cite{bensoussan2015maximum,li2023closed,li2023linear} is the \emph{adaptedness} of the admissible strategies. In these three papers, the strategies are assumed to be adapted to the filtration generated by the underlying noise. Informally, it implies that they may also depend on the paths of the Brownian motion driving the output process $X$. While this assumption is necessary to develop a resolution approach based on the stochastic maximum principle, one may question its feasibility in practice. Indeed, in real-world applications, it is debatable whether one actually observes the paths of the underlying noise, which is usually a modelling artefact without any physical reality.\footnote{For a more thorough discussion of this point, which is intimately linked to the question of whether one should adopt the `weak' or `strong' point of view in stochastic optimal control problems, we refer to the illuminating discussion in \cite[Section 9.1.1]{zhang2017backward}.} We thus consider in our framework that admissible closed-loop strategies should instead be adapted with respect to the filtration generated by the output process $X$. This different, albeit natural, concept of information for continuous-time stochastic Stackelberg games actually echoes the definition of closed-loop equilibria in the literature on `classical' stochastic differential games (see, for example \cite[Definition 5.5]{carmona2016lectures} for the case of closed-loop Nash equilibrium, or \cite{possamai2020zero} for zero-sum games). 

\smallskip

It should also be emphasised that the concept of information studied here, simply labelled \emph{closed-loop} for convenience, is therefore different from the so-called `adapted closed-loop' concept introduced (but not studied) by \cite{bensoussan2015maximum}, in which the players' strategies may depend on the whole trajectory of the output process $X$, but are nevertheless adapted with respect to the filtration generated by the underlying Brownian motion. Although it is outside the scope of this paper to study the characterisation of adapted closed-loop solutions for Stackelberg games, our illustrative example suggests that this concept of information may be `too broad'. More precisely, we will see in a simple example that if the leader can design a strategy depending on the trajectories of both the output and the underlying driving noise, then she can actually impose the maximum effort on the follower. 
This observation first echoes similar reductions of information asymmetry under specific settings in so-called \textit{team problems}, as discussed in deterministic continuous-time frameworks of \cite{basar1980teamoptimal} and \cite{papavassilopoulos1980sufficient}, and earlier in the discrete-time setting by \cite{basar1979closedloop}.
This observation also suggests that the difference between `adapted closed-loop' (in the sense of \cite{bensoussan2015maximum}) and what we coined `closed-loop' is akin to the difference between first-best and second-best equilibria defined in the literature on principal--agent problems, which are themselves specific Stackelberg games. 
This parallel is further reinforced by the fact that our solution concept, as well as the solution approach we propose, are in fact strongly inspired by the theory on continuous-time principal-agent problems. 

\paragraph*{Solution approach via stochastic target} The main contribution of our paper is therefore to provide a characterisation of the \emph{closed-loop} equilibrium (in the sense previously discussed) to a general continuous-time stochastic Stackelberg game, in which both players can control the drift and volatility of the output process. Allowing for path-dependent strategies leads to a more sophisticated form of equilibrium which, consequently, is more challenging to solve.
Indeed, in this case, the classical approaches used in the literature to characterise open-loop, or closed-loop memoryless equilibria, such as the maximum principle, can no longer be used. 
The approach we developed in this paper is based on the dynamic programming principle and stochastic target problems: the main idea is to use the follower's value function as a state variable for the leader's problem. 
More precisely, by writing forward the dynamics of the value function of the follower, which by the dynamic programming principle solves a backward SDE, we are able to reformulate the leader's problem as a stochastic control problem of a (forward) SDE system with a stochastic target constraint. 
We remark that the idea of considering the forward dynamics of the value function of the follower in a Stackelberg game, but with a continuum of followers, was used independently in \cite{dayanikli2023machine} to develop a numerical algorithm by means of Lagrange multipliers, \textit{i.e.} when the target constraint is added to leader's objective function as a penalisation term. 
Our approach is different in that we employ the methodology developed in \cite{bouchard2010optimal} and \cite{bouchard2009stochastic}, which leverages the dynamic programming principle for problems with stochastic target constraints established in \cite{soner2002dynamic,soner2002stochastic}, to provide a theoretical characterisation of the closed-loop solution through a system of Hamilton--Jacobi--Bellman (HJB) equations. 

\paragraph*{Overview of the paper} We first introduce in \Cref{sec:example} a simple illustrative example, in order to highlight the various concepts of Stackelberg equilibrium and the different approaches available to solve them. More importantly, we informally explain our approach in \Cref{ss:approach-example} through its application to the example under consideration. The rigorous formulation of the general problem is introduced in \Cref{sec:general_formulation}. In \Cref{sec:stochastic_target}, we reformulate the leader's problem in this general Stackelberg equilibrium as a stochastic control problem with stochastic target constraint, which is then solved in \Cref{sec:solve_leader}. 

{\small\smallskip
	\textit{\small Notations.} We let $\N^\star$ be the set of positive integers, $\R_+ \coloneqq [0,\infty)$ and $\R_+^\star\coloneqq (0,\infty)$. 
	For $(d,n)\in\N^\star\times\N^\star$, $\R^{d\times n}$ denotes the set of $d\times n$ matrices with real entries, while $\S^d$ (resp.  $\S^d_+$) the set of $d\times d$ (resp. positive semi-definite) symmetric matrices with real entries. 
	For any closed convex subset $S \subseteq \R$, we will denote by $\Pi_S(x)$ the Euclidean projection of $x \in \R$ on $S$. 
	For $T>0$ and a finite-dimensional Euclidean space $E$, $\Cc([0,T],\R)$ denotes the space of continuous functions from $[0,T]$ to $\R$,
	and $\Cc^{1,2}([0,T]\times E,\R)$ denotes the subset of $\Cc([0,T]\times E,\R)$ of all continuous functions from $[0,T]\times E$ to $\R$, which are continuously differentiable in time and twice continuously differentiable in space. For every $\varphi\in \Cc^{1,2}([0,T]\times E,\R)$, we denote by $\partial_t \varphi$ its partial derivative with respect to time and by $\partial_{x} \varphi$ and $\partial_{xx}^2 \varphi$ its gradient and Hessian with respect to the space variable, respectively. {We denote by $\Lc^0([0,T],E)$ the set of Borel-measurable maps from $[0,T]$ to $E$}. We agree that the supremum over an empty set is $-\infty$. For a stochastic process $X$, we denote by $\F^X\coloneqq (\Fc^X_t)_{t\geq 0}$ the filtration generated by $X$. {For any filtration $\F=(\Fc_t)_{t\geq 0}$, we denote by $\F^+\coloneqq (\Fc^+_t)_{t\geq 0}$ its right limit.}   }

\section{Illustrative example}\label{sec:example}

As already outlined in the introduction, there exist various concepts of Stackelberg equilibrium. In order to highlight their differences and describe the appropriate methods to compute each of them, we choose to develop in this section a simple illustrative example. 

\smallskip

Let $T > 0$ be a finite time horizon. For the sake of simplicity in this section, we focus on the strong formulation by fixing a probability space $(\Omega, \Fc, \P)$ supporting a one-dimensional Brownian motion $W$. We slightly abuse notations here and denote by $\F^W \coloneqq  (\Fc_t^W)_{t \in [0,T]}$ the natural filtration generated by $W$, $\P$-augmented in order to satisfy the usual hypotheses. We assume that the controlled one-dimensional state process $X$ satisfies the following dynamics
\begin{align}\label{eq:dynamic_x}
	\d X_t = ( \alpha_t + \beta_t) \d t + \sigma \d W_t, \; t \in [0,T],  \; X_0 = x_0 \in \R,
\end{align}
where the pair $(\alpha,\beta)$ represents the players' decisions and $ \sigma \in \R$ is a given constant. More precisely, the leader first announces her strategy $\alpha \in \Ac$ at the beginning of the game, where $\Ac$ is an appropriate family of $A$-valued processes for $A \subseteq \R$. With the knowledge of the leader’s action, the follower chooses an optimal response, \textit{i.e.} a control $\beta \in \Bc$ optimising his objective function, for a given set $\Bc$ of $B$-valued processes for $B\subseteq \R$. The sets $\Ac$ and $\Bc$ will be defined subsequently, as they crucially depend on the solution concept considered.

\smallskip

We assume that, given $\alpha \in \Ac$ chosen by the leader, the follower solves the following optimal stochastic control problem
\begin{align}\label{eq:pb-follower-example}
	V_{\rm F} (\alpha) \coloneqq  \sup_{\beta \in \Bc}\big\{ J_{\rm F} (\alpha, \beta)\big\}, \; \text{with} \; J_{\rm F} (\alpha, \beta) \coloneqq  \E^\P \bigg[ X_T - \dfrac{c_{\rm F}}2  \int_0^T \beta_t^2 \d t \bigg],
\end{align}
for some cost coefficient $c_{\rm F} >0$. Assuming uniqueness here to simplify, the best response of the follower to a control $\alpha \in \Ac$ chosen by the leader is naturally defined by
\begin{align}\label{eq:best-response-follower-example}
	\beta^\star (\alpha) \coloneqq  \argmax_{\beta \in \Bc} \big\{J_{\rm F} (\alpha, \beta)\big\}.
\end{align}

The leader, anticipating the follower’s optimal response $\beta^\star(\alpha)$, chooses $\alpha \in \Ac$ that optimises her own performance criterion. More precisely, we assume here that the leader's optimisation is given by
\begin{align}\label{eq:pb-leader-example}
	V_{\rm L} \coloneqq  \sup_{\alpha \in \Ac} \big\{J_{\rm L} \big(\alpha, \beta^\star (\alpha) \big)\big\}, \; \text{with} \; J_{\rm L} \big(\alpha, \beta^\star (\alpha) \big) \coloneqq  \E^\P \bigg[ X_T - \dfrac{c_{\rm L}}2  \int_0^T \alpha_t^2 \d t \bigg],
\end{align}
for some $c_{\rm L} > 0$, and where the dynamics of $X$ are now driven by the optimal response of the follower, \textit{i.e.}
\begin{align*}
	\d X_t = \big( \alpha_t + \beta^\star_t (\alpha) \big) \d t + \sigma \d W_t, \; t \in [0,T] , \; X_0 = x_0 \in \R.
\end{align*}
The leader’s optimal action and the follower’s rational response, namely the couple $(\alpha^\star, \beta^\star (\alpha^\star))$ for $\alpha^\star$ a maximiser in \eqref{eq:pb-leader-example}, constitute a \textit{global} Stackelberg solution or equilibrium. To ensure that the value of the Stackelberg game is finite for all the various equilibrium concepts, one should require the sets $A$ and $B$ to be bounded. For the sake of simplicity, we assume here that $A \coloneqq [-a_\circ,a_\circ]$ and $B\coloneqq [0,b_\circ]$ for 
$a_\circ > c_{\rm L}^{-1}$ and $b_\circ > c_{\rm F}^{-1}$.\footnote{The latter assumption is only intended to ensure that the `natural' open-loop equilibrium can be reached, see \Cref{sss:AOL}.}

\smallskip

The following section introduces the various notions of equilibrium in continuous-time stochastic Stackelberg games, and compares their solution. More importantly for our purpose, \Cref{ss:approach-example} illustrates our approach, based on dynamic programming and stochastic target problems, allowing to characterise a new notion of Stackelberg equilibrium, which we coin \textit{closed-loop}. Before proceeding, it may be useful to have in mind the optimal---or reference---equilibrium for the leader, \textit{i.e.} when she chooses both strategy directly. This optimal scenario for the leader, which can be labelled \textit{first-best} in reference to its counterpart in principal--agent problems\footnote{Our choice to coin said reformulation as `first-best' is not fortuitous, it is a terminology well-studied in the contract theory literature, see for instance \cite{cvitanic2012contract}, which is one particular instance of a Stackelberg game. Nevertheless, as mentioned in the introduction, this notion of equilibrium also echoes the concept of team solutions in Stackelberg games, as discussed for example in \cite{basar1980teamoptimal}.}, should naturally arise when the leader can deduce the follower's strategy from her observation, and is able to \textit{strongly} penalise him whenever he deviates from the optimal strategy she recommended. The value of the leader in this \textit{first-best} problem is naturally defined by 
\begin{align}\label{eq:pb_first-best}
	V_{\rm L}^{\rm FB} \coloneqq \sup_{(\alpha,\beta)\in\Ac\times\Bc} \big\{J_{\rm L}(\alpha,\beta)\big\},
\end{align}
where here, $\Ac$ and $\Bc$ are the sets of {$\F^W$}-adapted processes taking values in $A$ and $B$, respectively.
This corresponds to a simple stochastic control problem, whose solution is provided in the following lemma. Its proof is very straightforward using standard HJB techniques or even by pointwise optimisation, and therefore omitted here.
\begin{lemma}[First-best solution]\label{lem:first-best}
	The optimal efforts in the first-best scenario are given by $\alpha^{\rm FB}_t = c_{\rm L}^{-1}$ and $\beta_t^{\rm FB} = b_\circ$ for all $t \in [0,T]$, which induce the following values for the leader and the follower, respectively
	\begin{align*}
		V_{\rm L}^{\rm FB} &= x_0 + \bigg(\dfrac{1}{2 c_{\rm L}} + b_\circ \bigg)T, \quad
		V_{\rm F}^{\rm FB} \coloneqq J_{\rm F} \big(\alpha^{\rm FB},\beta^{\rm FB} \big)
		= x_0 + \bigg( \dfrac{1}{c_{\rm L}} + b_\circ - \dfrac12 c_{\rm F} b^2_\circ \bigg)T.
	\end{align*}
\end{lemma}

\subsection{Various Stackelberg equilibria}\label{ss:different_equilibrium}

There exist various notions of equilibrium in a continuous-time stochastic Stackelberg game. These concepts are related to the information available to both players, the leader and the follower, at the beginning and during the game. Following the nomenclature in \cite{basar1999dynamic} for dynamic Stackelberg games, and extended to the stochastic version in \cite{bensoussan2015maximum}, we {\emph{informally}\footnote{The definition of the information available to both players is rather informal here, to adhere to the concepts introduced in \cite{bensoussan2015maximum}. More rigorously, it could be defined as the filtration generated by the processes observable by both players. Nonetheless, we will define in a rigorous way the sets $\Ac$ and $\Bc$ of admissible efforts depending on the solution concept considered.}} define by $\Ic_t$ the information available to both players at time $t \in [0,T]$ and distinguish four cases
\begin{enumerate}[label=$(\roman*)$]
	\item \emph{adapted open-loop} (AOL) when $\Ic_t = \{x_0, W_{\cdot \wedge t}\}$;
	\item \emph{adapted feedback} (AF) when $\Ic_t = \{X_t, W_{\cdot \wedge t} \}$;
	\item \emph{adapted closed-loop memoryless} (ACLM) when $\Ic_t = \{x_0, X_t, W_{\cdot \wedge t}\}$;
	\item \emph{adapted closed-loop} (ACL) when $\Ic_t = \{X_{\cdot \wedge t}, W_{\cdot \wedge t}\}$.
\end{enumerate}
As explained in \cite{bensoussan2015maximum}, the information structures $(i), (iii)$ and $(iv)$ lead to the concept of global Stackelberg solutions, where the leader actually dominates the follower over the entire duration of the game. In these situations, a Stackelberg equilibrium $(\alpha^\star, \beta^\star(\alpha^\star))$ is characterised as in the illustrative example above
\begin{align*}
	J_{\rm F}  (\alpha, \beta^\star (\alpha) ) 
	\geq  J_{\rm F} (\alpha, \beta), \; \text{and} \; J_{\rm L}  (\alpha^\star, \beta^\star (\alpha^\star)  ) \geq  J_{\rm L}   (\alpha, \beta^\star (\alpha)  ), \;   \forall  (\alpha,\beta) \in \Ac \times \Bc.
\end{align*}
The information structure $(ii)$ leads to a different concept of solution in which the leader has only an instantaneous advantage over the follower. More precisely, a \textit{feedback} Stackelberg equilibrium $(\alpha^\star, \beta^\star(\alpha^\star))$ should satisfy
\begin{align*}
	J_{\rm F}  (\alpha^\star, \beta^\star (\alpha^\star) ) \geq  J_{\rm F} (\alpha^\star, \beta), \; \text{and} \;
	J_{\rm L}  (\alpha^\star, \beta^\star (\alpha^\star)  ) \geq  J_{\rm L}   (\alpha, \beta^\star (\alpha) ), \; \forall  (\alpha,\beta) \in \Ac \times \Bc.
\end{align*}
In the following, we illustrate the existing approaches to computing the equilibrium in the first three information structures in the context of the above example. Even though the last information structure, corresponding to the \textit{adapted} closed-loop (with memory) case, has not been studied in the literature, we are able to characterise it in this example. Indeed, our analysis established a connection between this Stackelberg solution concept and the first-best scenario, already discussed in \Cref{lem:first-best}.

\smallskip

However, the real aim of this paper is not to study existing solution concepts, but to introduce a new, albeit natural, concept of information, corresponding to the definition of closed-loop equilibrium in the literature on stochastic differential games (see, for example, \cite[Definition 5.5]{carmona2016lectures}), in which the information available to both players at time $t \in [0,T]$ is---again informally---defined as
\begin{enumerate}[resume,label=$(\roman*)$]
	\item \emph{closed-loop} (CL) when $\Ic_t = \{X_{\cdot \wedge t}\}$.
\end{enumerate}
In particular, this information concept is different from the \emph{adapted closed-loop} case introduced in \cite{bensoussan2015maximum} and mentioned above, as we do not assume here that the players have access to the paths of the Brownian motion. As already highlighted in the introduction, considering such an information structure makes sense, especially in real-world applications, as it usually seems unrealistic to believe that players can actually observe the underlying noise driving the output process, the latter being in most cases a modelling artefact. Admissible strategies constructed using this information structure are therefore not assumed to be adapted to the natural filtration generated by the Brownian motion, in contrast to \emph{adapted closed-loop} strategies, hence we simply refer to them as \emph{closed-loop}. 

\smallskip
More precise specifications of this solution concept, along with an informal description of the methodology we develop to characterise the corresponding Stackelberg equilibrium, are presented separately in \Cref{ss:approach-example}. We present below the main results obtained in the context of the example, especially the comparison of the values obtained for both players, depending on the equilibrium considered.

\paragraph*{Comparison of the equilibria} The results we obtain for the different solution concepts are summarised in \Cref{tab:comparison} below. Before commenting on our results, we should point out that these findings were obtained for the example introduced at the beginning of this section, and by no means do we claim or expect that they would all be true in a more general context. Nevertheless, given the significance of some of these findings, especially the fact that, from the leader's value point of view\footnote{From the follower's point of view, all the inequalities are naturally reversed.}, {$V_{\rm L}^{\rm AOL}= V_{\rm L}^{\rm AF}< V_{\rm L}^{\rm ACLM-\bar{K}} < V_{\rm L}^{\rm CL} < V_{\rm L}^{\rm ACL} = V_{\rm L}^{\rm FB}$}, investigating the extent to which they hold in greater generality could be the subject of future research. Note here that the value $ V_{\rm L}^{\rm ACLM-\bar{K}}$ corresponds to a restricted version of the ACLM problem introduced above, whose precise formulation can be found in \Cref{sss:ACLM}. Its introduction is necessary, as we will point out that we were not successful in getting an explicit formula in this case for the general problem.

\begin{table}[!ht] 
	\centering
	\renewcommand{\arraystretch}{1.4}
	\begin{tabular}{|c|c|c|c|c|} 
		\hline
		& \textbf{AOL} and {\bf AF} & {\bf ACLM-$\bar{K}$}  & {\bf ACL} and {\bf FB}   \\   \hline
		\rule{0pt}{1.4\normalbaselineskip} $V_{\rm L}$ & $ x_0 + \bigg( \dfrac{1}{2c_{\rm L}} + \dfrac{1}{c_{\rm F}} \bigg) T$   &   \( x_0 + \bigg(\dfrac{1}{2c_{\rm L}} + b_\circ \bigg) T - \Delta_{\rm L} \) &  $x_0 + \bigg( \dfrac{1}{2c_{\rm L}} + b_\circ \bigg) T$  \\[0.7em] \hline    
		\rule{0pt}{1.4\normalbaselineskip} $V_{\rm F}$ & $ x_0 + \bigg( \dfrac{1}{c_{\rm L}} + \dfrac{1}{2c_{\rm F}} \bigg) T $  & $ x_0  + \bigg( \dfrac{1}{c_{\rm L}} + b_\circ - \dfrac12 c_{\rm F} b^2_\circ \bigg) T - \Delta_{\rm F} $ & $x_0 + \bigg( \dfrac{1}{c_{\rm L}} + b_\circ - \dfrac12 c_{\rm F} b^2_\circ \bigg)T$  \\[0.7em] \hline
	\end{tabular}
	\caption{\small Comparison of the various Stackelberg equilibria}\label{tab:comparison}
	\footnotesize{{\color{black}In the ACLM-$\bar K$ case, $\Delta_{\rm L} \coloneqq  \frac{1}{\bar K} \big(b_\circ \log(b_\circ c_{\rm F}) + b_\circ - \frac{1}{c_{\smallfont{\rm F}}} \big)$, and $\Delta_{\rm F} \coloneqq \frac{1}{\bar K} \big( \log(b_\circ  c_{\smallfont{\rm F}}) - \frac{(b_\circ c_{\smallfont{\rm F}} - 1)}{c_{\smallfont{\rm F}} } + \frac{(b_\circ^2 c_{\smallfont{\rm F}}^2 - 1)}{4 c_{\smallfont{\rm F}}} \big) $, for $\bar K=\frac{1}{T} \log\big( \frac{1}{b_\circ} \big( a_\circ + \frac{1}{c_{\smallfont{\rm L}}} +\frac{1}{2c_{\smallfont{\rm F}}} \big) + \frac12 b_\circ c_{\smallfont{\rm F}} \big)$.}}
\end{table}

First of all, it is obviously expected that, for any concept of Stackelberg equilibrium, the value of the leader will be lower than her value in the first-best case, introduced as a reference in \Cref{lem:first-best}, since in this scenario the leader can directly choose the optimal effort of the follower. It is also expected that the more available information the leader can use to implement her strategy, the higher the value she will obtain, which translates mathematically into the following inequalities
\begin{align}\label{eq:value_info}
	V_{\rm L}^{\rm AOL} \leq V_{\rm L}^{\rm ACLM} \leq V_{\rm L}^{\rm ACL}, \; 
	V_{\rm L}^{\rm AF} \leq V_{\rm L}^{\rm ACLM}, \; \text{and} \;
	V_{\rm L}^{\rm CL} \leq V_{\rm L}^{\rm ACL}.
\end{align}
In the context of our simple example, our first finding is that the Stackelberg equilibrium, and hence the associated values for the leader and the follower, coincide for both the adapted open-loop (\Cref{sss:AOL}) and the adapted feedback (\Cref{sss:AF}) information structures. This might reflect how the additional information under the feedback structure is counterbalanced by the global dominance of the open-loop strategies. Regarding the value of the leader in the ACLM-$\bar{K}$ information structure (\Cref{sss:ACLM}), strict inequalities with respect to the values in the AOL and ACL cases can be obtained for specific choices of the model parameters. Namely, we assume in \Cref{lemma-ACLM} that 
	\begin{align}\label{cdtn:a_0_ACLM}
		a_\circ \geq \frac{1}{c_{\rm L}} + \frac{(b_\circ c_{\rm F} -1)^2}{2c_{\rm F}},
	\end{align}
	in order to compute explicitly the value of the leader.
On the other hand, our analysis of the Stackelberg game under adapted closed-loop strategies in \Cref{sss:ACL} shows that as long as the leader can effectively punish the follower at no additional cost, see \Cref{eq.punishment}, then the problem degenerates to the first-best case. More precisely, by observing the trajectory of $X$ as well as that of $W$, the leader can actually deduce the follower's effort at each time, and thus force him to perform the maximum effort $b_\circ$, threatening to significantly penalise him otherwise. This is the case, for instance, if 
\begin{align}\label{cdtn:a_0_ACL}
	a_\circ \geq \dfrac{1}{2c_{\rm F}} - b_\circ + \dfrac12 c_{\rm F} b_\circ^2 - \dfrac{1}{c_{\rm L}}.
\end{align}
Finally, regarding our closed-loop equilibrium, while it is already clear that the value for the leader should be lower than in the ACL and FB cases, we can also argue that it is higher than in the AOL (and thus AF) case (see \Cref{sss:numerics}).
However, the comparison with the ACLM case is less straightforward. Unfortunately, we are not able to obtain explicit results in this framework, even in the context of this simple example, and we thus rely on numerical results, also presented in \Cref{sss:numerics}. These numerical results seem to illustrate that the CL equilibrium gives a higher value for the leader compare to the ACLM-$\bar{K}$ case, at least when $a_\circ$ is chosen sufficiently large so that \Cref{cdtn:a_0_ACLM} and \Cref{cdtn:a_0_ACL} are satisfied. Although we cannot rule out the possibility that these conclusions could be reversed for different sets of parameters, the numerical results nevertheless highlight that these two equilibria are essentially different.

\subsubsection{Adapted open-loop strategies}\label{sss:AOL}

In a Stackelberg game under the adapted open-loop (AOL) information structure, both players have access to the initial value of $X$, namely $x_0$, and the trajectory of the Brownian motion $W$. Since the leader first announces her strategy $\alpha$, its value $\alpha_t$ at a any time $t \in [0,T]$ should only depend on the realisation of the Brownian motion on $[0,t]$, and on the initial value $x_0$ of the state. 
The leader's strategy space $\Ac$ in this case is thus naturally defined by
\[
\Ac \coloneqq   \big\{   \alpha: [0,T]\times\Omega\ni(t,\omega)\longmapsto a(t,W_{\cdot \wedge t}(\omega), x_0), a:[0,T]\times \Cc([0,T],\R)\times\R \longrightarrow A \, \text{Borel}\big\}.
\]
As the follower makes his decision after the leader announces her whole strategy $\alpha$ on $[0,T]$, his strategy may also depend on the leader's announced strategy. More precisely, the value $\beta_t$ of the follower’s response strategy at time $t \in [0,T]$ is naturally measurable with respect to $\Fc_t^W$, but can also depend on the leader’s strategy $\alpha$. His response strategy space is thus defined by\footnote{Here and in the following, we consider in $\Lc^0([0,T],A)$ the topology induced by the Dunford--Schwartz pseudo-metric $\rho(f,g)=\inf_{c>0} \arctan \{c+\mu(\{ |f-g|>c\}) \}$, with $\mu$ the Lebesgue measure.}
\begin{align*}
	\Bc \coloneqq  \big\{ \beta &:[0,T]\times\Omega\ni(t,\omega) \longmapsto  b(t,W_{\cdot\wedge t}(\omega),x_0,\alpha_{\cdot\wedge t}(\omega)): \\
	b &:[0,T]\times \Cc([0,T],\R)\times\R\times\Lc^0([0,T],A) \longrightarrow B,\; \text{Borel} \big\}.
\end{align*}
Note that, at any time $t \in  [0,T]$, since the information available to the leader is also available to the follower, the follower can naturally compute the value of the leader's strategy at that instant $t$, \textit{i.e.} $\alpha_t(\omega)$. However, he cannot anticipate the future values of the leader's strategy $\alpha$.

\smallskip

As described in \cite[Section 3]{bensoussan2015maximum}, one way to characterise a global Stackelberg equilibrium under the AOL information structure is to rely on the maximum principle. 
A general result is given, for example, in \cite[Proposition 3.1]{bensoussan2015maximum}, but we briefly describe this approach through its application to our example. 
Recall that, given the leader's strategy $\alpha \in \Ac$, the follower's problem is defined by \Cref{eq:pb-follower-example}, where the dynamics of the state variable $X$ satisfies \eqref{eq:dynamic_x}. To solve this stochastic optimal control problem through the maximum principle, we first define the appropriate Hamiltonian
\begin{align*}
	h^{\rm F} (t,a,y, z, b) 
	\coloneqq  (a+b) y + \sigma z - \dfrac{c_{\rm F}}2 b^2, \; (t,a,y, z, b)  \in [0,T] \times A \times \R^2\times B.
\end{align*}
Suppose now that there exists a solution $\beta^\star(\alpha)$ to the follower's problem \eqref{eq:pb-follower-example} for any $\alpha \in \Ac$. Then, the maximum principle states that there exists a pair of real-valued, $\F^W$-adapted processes $(Y^{\rm F},Z^{\rm F})$ such that
\begin{align}\label{eq:MP-open-loop-example}
	\begin{cases}
		\displaystyle \d X_t = \big( \alpha_t + \beta^\star_t(\alpha) \big) \d t + \sigma \d W_t, \; t \in [0,T], \; X_0 = x_0; \\[0.5em]
		\displaystyle \d Y_t^{\rm F} = Z_t^{\rm F} \d W_t, \; t \in [0,T], \; Y^{\rm F}_T = 1; \\[0.5em]
		\displaystyle \beta^\star_t(\alpha) \coloneqq  \argmax_{b \in B} \big\{h^{\rm F} \big( t,\alpha_t,Y_t^{\rm F}, Z_t^{\rm F}, b \big)\big\}, \; \mathrm{d}t\otimes\P\text{\rm--a.e.}
	\end{cases}
\end{align}
Note that the drift in the backward SDE (BSDE for short) in \eqref{eq:MP-open-loop-example}, commonly called adjoint process, is equal to $0$ because the Hamiltonian $h^{\rm F}$ does not depend on the state variable. Clearly, in this simple example, the pair $(Y^{\rm F},Z^{\rm F})$ satisfying the BSDE is the pair of constant processes $(1,0)$. This leads to the optimal constant control $\beta^\star_t(\alpha) = 1/c_{\rm F} \in B$ for all $t \in [0,T]$. In particular, this control is independent of the leader's choice of $\alpha$. The leader's problem defined by \eqref{eq:pb-leader-example} thus becomes
\begin{align*}
	V_{\rm L}=\sup_{\alpha \in \Ac} \bigg\{\E^\P \bigg[ X_T - \dfrac{c_{\rm L}}2  \int_0^T \alpha_t^2 \d t \bigg]\bigg\}, \;
	\text{subject}\; \text{to} \; \d X_t = \bigg( \alpha_t + \dfrac{1}{c_{\rm F}} \bigg) \d t + \sigma \d W_t, \; t \in [0,T].
\end{align*}
This optimal control problem is trivial to solve, and also leads to an optimal constant control for the leader, namely $\alpha^\star_t = 1/c_{\rm L} \in A$ for all $t \in [0,T]$. The open-loop equilibrium is thus given by $(1/c_{\rm L}, 1/c_{\rm F})$, which is admissible thanks to the assumptions $a_\circ \geq 1/c_{\rm L}$ and $b_\circ \geq 1/c_{\rm F}$, and one can easily compute the corresponding values for the leader and the follower, given in \Cref{tab:comparison}.

\subsubsection{Adapted feedback strategies}\label{sss:AF}

A Stackelberg game under the adapted feedback (AF) information structure differs from the other Stackelberg equilibrium, not only in the information structure itself, but also in the way the game is played. In this scenario, both players only have access to the current value of $X$ and the trajectory of the Brownian motion $W$. In other words, the leader's strategy at time $t \in [0,T]$ can only depend on the value $X_t$ and the realisation of the Brownian motion on $[0,t]$. Therefore, the leader’s and follower’s strategy spaces are respectively defined by
\begin{align*}
	\begin{split}
		\Ac \coloneqq   \big\{  \alpha &: [0,T]\times\Omega\ni(t,\omega)\longmapsto  a(t,W_{\cdot \wedge t}(\omega), X_t(\omega)): \\ 
		a &:[0,T]\times \Cc([0,T],\R)\times\R \longrightarrow A,\; \text{Borel}\big\}, \\
		\Bc \coloneqq  \big\{ \beta &: [0,T]\times\Omega\ni(t,\omega)\longmapsto b(t,W_{\cdot\wedge t}(\omega),X_t(\omega),\alpha_t(\omega)): \\ 
		b &:[0,T]\times \Cc([0,T],\R)\times\R\times A \longrightarrow B,\; \text{Borel} \big\}.
	\end{split}
\end{align*}
Under this information structure, the equilibrium is not \textit{global}, in the sense that at each time $t \in [0,T]$, the leader first decides her action $\alpha_t$, and then the follower makes his decision, immediately after observing the leader’s instant action at time $t$, rather than her whole strategy over $[0,T]$.
Mathematically speaking, an AF Stackelberg solution is a pair $(\alpha^\star, \beta^\star (\alpha^\star) ) \in \Ac \times \Bc$ satisfying $J_{\rm F} (\alpha^\star, \beta^\star (\alpha^\star)  )\geq  J_{\rm F} (\alpha^\star, \beta)$, $ \forall  \beta \in \Bc$, $J_{\rm L} (\alpha^\star, \beta^\star (\alpha^\star) ) \geq  J_{\rm L}  (\alpha, \beta^\star (\alpha) )$, $\forall \alpha \in \Ac$.
As mentioned in \cite[p.~1960]{bensoussan2015maximum}, in a Markovian setting like the example under consideration, one can find so-called `deterministic menus', in the sense that the functions \( a \) and \( b \) defining the strategy spaces \( \Ac \) and \( \Bc \) above can be taken independent of the Brownian motion. 
	
\smallskip

The above remark justifies, in the particular setting of the example, the use of a dynamic programming approach to compute the equilibrium. We therefore follow the method developed in \cite{bensoussan2014feedback}, which relies on the characterisation of the Stackelberg equilibrium via a system of coupled HJB equations. More precisely, we first introduce the players' Hamiltonians
\begin{align*}
	h^{\rm F} \big(t,z^{\rm F},a, b \big) 
	\coloneqq  (a+b) z^{\rm F} - \dfrac{c_{\rm F}}2 b^2, \;\text{and}\; 
	h^{\rm L} \big(t,z^{\rm L},a, b \big)
	\coloneqq  (a+b) z^{\rm L} - \dfrac{c_{\rm L}}2 a^2,
\end{align*}
for $(t,z^{\rm F},z^{\rm L}, a,b)  \in [0,T] \times \R^2 \times A \times B$.
For a fixed action of the leader, the follower's optimal response is given by the maximiser of his Hamiltonian, \textit{i.e.}
\begin{align*}
	b^\star \big(t,z^{\rm F},a \big) \coloneqq  \argmax_{b \in B} \big\{h^{\rm F} \big(t,z^{\rm F},a, b\big)\big\} = \Pi_B \bigg( \dfrac{z^{\rm F}}{c_{\rm F}} \bigg), \; (t,z^{\rm F}, a )  \in [0,T] \times \R \times A,
\end{align*}
recalling that $\Pi_{B}(x)$ denotes the projection of $x \in \R$ on $B$.
One should then replace this optimal response into the leader's Hamiltonian. Nevertheless, in this example, it does not change the functional maximising the leader's Hamiltonian, given by
\begin{align*}
	a^\star \big(t,z^{\rm F},z^{\rm L} \big) \coloneqq  \argmax_{a \in A} \big\{h^{\rm L} \big(t,z^{\rm L},a, b^\star (t,z^{\rm F},a ) \big)\big\} = \Pi_A \bigg( \dfrac{z^{\rm L}}{c_{\rm L}} \bigg), \; (t,z^{\rm F}, z^{\rm L} )  \in [0,T] \times \R^2.
\end{align*}
To compute the equilibrium, one must solve the following system of coupled Hamilton--Jacobi--Bellman equations
\begin{align*}
	\begin{cases}
		- \partial_t v_{\rm F} - \bigg(\Pi_A \bigg( \dfrac{\partial_{x} v_{\rm L}}{c_{\rm L}} \bigg) + \Pi_B \bigg( \dfrac{\partial_{x} v_{\rm F}}{c_{\rm F}} \bigg) \bigg) \partial_{x} v_{\rm F} + \dfrac{c_{\rm F}}2 \Pi_B^2 \bigg( \dfrac{\partial_{x} v_{\rm F}}{c_{\rm F}} \bigg) - \dfrac12 \sigma^2 \partial_{xx} v_{\rm F} = 0, \\[0.8em]
		- \partial_t v_{\rm L} - \bigg(\Pi_A \bigg( \dfrac{\partial_{x} v_{\rm L}}{c_{\rm L}} \bigg) + \Pi_B \bigg( \dfrac{\partial_{x} v_{\rm F}}{c_{\rm F}} \bigg) \bigg) \partial_{x} v_{\rm L} + \dfrac{c_{\rm L}}2 \Pi_A^2 \bigg( \dfrac{\partial_{x} v_{\rm L}}{c_{\rm L}} \bigg) - \dfrac12 \sigma^2 \partial_{xx} v_{\rm L} = 0,
	\end{cases}
\end{align*}
on $[0,T) \times \R$, with boundary conditions $v_{\rm F} (T,x) = v_{\rm L} (T,x) = x$, $x \in \R$. One can check using a standard verification theorem that the appropriate solutions to the previous system for $(t,x) \in [0,T] \times \R$ are 
\begin{align*}
	v_{\rm F} (t,x) = x + \bigg( \dfrac{1}{c_{\rm L}} 
	+ \dfrac1{2 c_{\rm F}} \bigg) (T-t), \; \text{and} \;
	v_{\rm L} (t,x) = x + \bigg( \dfrac1{c_{\rm F}} + \dfrac{1}{2 c_{\rm L}} \bigg) (T-t),
\end{align*}
which correspond to the constant strategies $(1/c_{\rm L}, 1/c_{\rm F}) \in A \times B$. In particular, the feedback Stackelberg equilibrium coincides with the open-loop solution computed before, both in terms of strategy and corresponding value.

\subsubsection{Adapted closed-loop memoryless strategies} \label{sss:ACLM}

If the information structure is assumed to be adapted closed-loop memoryless (ACLM), then both players have access to the initial and current value of $X$, as well as the trajectory of the Brownian motion $W$. This means that both players can make the values of their decisions at time $t$ contingent on additionally the current state information $X_t$, when compared to the AOL information structure case. Then, the leader’s strategy space and the follower’s response strategy space are naturally defined by 
\begin{align*}
	\Ac \coloneqq   \big \{  \alpha &: [0,T]\times\Omega\ni(t,\omega)\longmapsto a(t,W_{\cdot \wedge t}(\omega), X_t(\omega),x_0): \\  a &:[0,T]\times \Cc([0,T],\R)\times\R^2 \longrightarrow A,\;\text{Borel}\big\}, \\
	\Bc \coloneqq  \big\{  \beta &: [0,T]\times\Omega\ni(t,\omega)\longmapsto b(t,W_{\cdot\wedge t}(\omega),X_t(\omega),x_0,\alpha_{\cdot\wedge t}(\omega)):\\
	b &:[0,T]\times \Cc([0,T],\R)\times\R^2\times \Lc^0([0,T],A)) \longrightarrow B,\; \text{Borel} \big\}.
\end{align*}
As mentioned above, the main difference between the ACLM and the AOL information structures is that the leader's control at time $t$ can now depend on the value of the state at that time. However, by choosing his strategy $\beta$, the follower will naturally impact the dynamic of the state $X$ and thus its value, which in turn impacts the value of the leader's control $\alpha$. Therefore, in order to compute his optimal response to a strategy $\alpha$ of the leader, the follower needs to take into account the retroaction of his control on the value of the leader's control, which thus leads to a more sophisticated form of equilibrium. In particular, contrary to the AOL case where the leader is relatively myopic, in the sense that she cannot possibly take into account the choice of the follower, she can now design a strategy indexed on the state that will therefore take into account the follower's actions.

\smallskip
In order to characterise the global Stackelberg equilibrium under the ACLM information structure, we can again rely on the maximum principle (see \cite[Section 4]{bensoussan2015maximum}). 
First, to highlight the dependency of the value $\alpha_t$ on the current value of the state $X_t$, we write $\alpha_t \eqqcolon  a_t(X_t)$ for $a:[0,T]\times \Omega\times \R\times \{x_0\}\longrightarrow A$, whose values at a fixed $(t,\omega)\in [0,T]\times \Omega$ induces the family ${\rm A}$ of mappings ${\rm a}: \R\times \{x_0\}\longrightarrow A$. We can then follow the maximum principle approach as before, but taking into account this dependency. More precisely, as before, we fix the leader's strategy $\alpha \in \Ac$ and thus its value $a_t(X_t)$ at time $t$, and consider the follower's problem given by \eqref{eq:pb-follower-example}, but now subject to the following dynamics
\begin{align*}
	\d X_t = ( a_t(X_t) + \beta_t) \d t + \sigma \d W_t, \; t \in [0,T], \; X_0 = x_0,
\end{align*}
where the dependency of the leader's control on the state appears explicitly. This dependency will thus also appear similarly in the follower's Hamiltonian
\begin{align*}
	h^{\rm F} (t,{\rm a},x,y, z, b) 
	\coloneqq  ({\rm a}(x) +b) y + \sigma z - \dfrac{c_{\rm F}}2 b^2, \; (t,{\rm a},x,y, z, b)  \in [0,T] \times {\rm A} \times \R^3\times B.
\end{align*}
Suppose there exists a solution $\beta^\star(\alpha)$ to the follower's problem \eqref{eq:pb-follower-example} for any $\alpha \in \Ac$, then the maximum principle states that there exists a pair of $\F^W$-adapted processes $(Y^{\rm F},Z^{\rm F})$ satisfying the forward--backward SDE (FBSDE for short)
\begin{align*}
	\begin{cases}
		\displaystyle \d X_t = \big( a_t(X_t) + \beta^\star_t(\alpha) \big) \d t + \sigma \d W_t, \; t \in [0,T], \; X_0 = x_0, \\[0.5em]
		\displaystyle \d Y_t^{\rm F} = - \partial_x h^{\rm F} \big( t,\alpha_t,X_t,Y_t^{\rm F}, Z_t^{\rm F}, \beta^\star_t(\alpha) \big) \d t + Z_t^{\rm F} \d W_t, \; t \in [0,T], \; Y^{\rm F}_T = 1, \\[0.5em]
		\displaystyle \beta^\star_t(\alpha) \coloneqq  \argmax_{b \in B} \big\{h^{\rm F} \big( t,\alpha_t,X_t,Y_t^{\rm F}, Z_t^{\rm F}, b \big)\big\}, \; t \in[0,T].
	\end{cases}
\end{align*}
Notice that $h^{\rm F}$ now depends explicitly on the state variable, and thus the associated partial derivative is not equal to zero, contrary to the AOL case. By computing the maximiser of $h^{\rm F}$ over $b \in B$, the previous FBSDE system becomes
\begin{align}\label{eq:MP-ACLM-example}
	\begin{cases}
		\displaystyle	\d X_t = \bigg( a_t(X_t) + \Pi_B \bigg( \dfrac{Y_t^{\rm F}}{c_{\rm F}} \bigg) \bigg) \d t + \sigma \d W_t, \; t \in [0,T], \; X_0 = x_0, \\[0.8em]
		\displaystyle	\d Y_t^{\rm F} = - \partial_x a_t(X_t) Y_t^{\rm F} \d t + Z_t^{\rm F} \d W_t, \; t \in [0,T], \; Y^{\rm F}_T = 1.
	\end{cases}
\end{align}
One can then reformulate the leader's problem defined by \eqref{eq:pb-leader-example} as a stochastic control problem of an FBSDE system
\begin{align}\label{eq:leader-reformulated-ACLM-example}
	V_{\rm L} = \sup_{\alpha \in \Ac} \bigg\{\E^\P \bigg[ X_T - \dfrac{c_{\rm L}}2  \int_0^T \alpha_t^2 \d t \bigg]\bigg\}, \; \text{subject to the dynamics in \eqref{eq:MP-ACLM-example}}.
\end{align}

The presence of the derivative $\partial_x a$ of the leader's strategy in \eqref{eq:MP-ACLM-example} results in a non-standard optimal control problem for the leader, but this problem can nevertheless also be solved via the maximum principle, as described in \cite[Section 4]{bensoussan2015maximum}. More precisely, the idea to solve the leader's problem is to look at efforts of the form $a_t(X_t)=a^{2}_t X_t + a^{1}_t$, where $a^{1}$ and $a^{2}$ are $\F^W$-adapted, $\R$-valued processes such that $a^{2}_t X_t + a^{1}_t\in A$ for every $t\in[0,T]$, $\P$--a.s. We define $\Ac^2$ as the space of processes $(a^1,a^2)$ satisfying these properties. It then follows from \cite[Theorem 4.1]{bensoussan2015maximum} that $V_{\rm L}=\widetilde V_{\rm L}$, where
\begin{align}\label{eq:MP-ACLM-affine}
	\widetilde V_{\rm L} \coloneqq \sup_{( a^{\smallfont1}, a^{\smallfont2})\in\Ac^2} \bigg\{\E^\P \bigg[ X_T - \dfrac{c_{\rm L}}2  \int_0^T \big( a^{2}_t X_t + a^{1}_t \big)^2 \d t \bigg]\bigg\}, 
\end{align}
subject to
\begin{align} \label{eq:ACLM-states-before-linearizing}
	\begin{cases}
		\displaystyle	\d X_t = \bigg( a^{2}_t X_t +a^{1}_t + \Pi_B \bigg( \dfrac{Y_t^{\rm F}}{c_{\rm F}} \bigg) \bigg) \d t + \sigma \d W_t, \; t \in [0,T], \; X_0 = x_0 ,\\[0.8em]
		\displaystyle	\d Y_t^{\rm F} = - a^{2}_t Y_t^{\rm F} \d t + Z_t^{\rm F} \d W_t, \; t \in [0,T], \; Y^{\rm F}_T = 1.
	\end{cases}
\end{align}
To solve $\widetilde V_{\rm L}$, we define, for $(t,x,x^\prime, y,y^\prime, z, z^\prime, {\rm a}^{1}, {\rm a}^{2})  \in [0,T] \times \R^8$, the Hamiltonian
\begin{align*}
	h^{\rm L} (x,x^\prime, y,y^\prime, z, z^\prime, {\rm a}^{1}, {\rm a}^{2}) 
	\coloneqq  \bigg( {\rm a}^{2} x + {\rm a}^{1} + \Pi_B \bigg( \dfrac{y^\prime}{c_{\rm F}} \bigg) \bigg) y 
	+ \sigma z 
	- {\rm a}^{2} y^\prime x^\prime
	- \dfrac{c_{\rm L}}2 ({\rm a}^{2} x + {\rm a}^{1})^2.
\end{align*}
Notice that another issue arises here when using the maximum principle for problem $\widetilde V_{\rm L}$, as the maximiser of $h^{\rm L}$ is not well-defined without further restriction on the strategy $\alpha \in \Ac$. A way to tackle this issue is to impose \emph{a priori} bounds on $\partial_x a$, as done in \cite[Section 5.2]{bensoussan2015maximum}. We thus define the following ACLM-$k$ problem, for some $k>0$, in which we assume that $\|a^{2}\|_{\infty} \leq k$
	\begin{align}\label{eq:MP-ACLM-k-affine}
		\widetilde V_{\rm L}^k \coloneqq \sup_{( a^{\smallfont1}, a^{\smallfont2})\in\Ac^2_k} \bigg\{\E^\P \bigg[ X_T - \dfrac{c_{\rm L}}2  \int_0^T \big( a^{2}_t X_t + a^{1}_t \big)^2 \d t \bigg]\bigg\}, 
	\end{align}
	subject to \eqref{eq:ACLM-states-before-linearizing} and where $\Ac^2_k$ is the restriction of $\Ac^2$ to the pairs $( a^{1}, a^{2})$ such that $\|a^{2}\|_{\infty} \leq k$. 
By \cite[Theorem 4.1]{bensoussan2015maximum}, if $\hat \alpha$ is a solution {to the leader's ACLM-$k$ problem} with the corresponding state trajectory $(\hat X, \hat Y^{\rm F},\hat  Z^{\rm F})$, then there exists a triple of $\F^W$-adapted processes $(X^{\rm L}, Y^{\rm L}, Z^{\rm L})$ such that
\begin{align*}
	\begin{cases}
		\displaystyle	\d X^{\rm L}_t = - \partial_{y^\smallfont{\prime}} h^{\rm L} \d t - \partial_{z^\smallfont{\prime}} h^{\rm L} \d W_t,\;  t \in [0,T], \; X^{\rm L}_0 = 0, \\[0.5em]
		\displaystyle	\d Y_t^{\rm L} = - \partial_{x} h^{\rm L} \d t + Z_t^{\rm L} \d W_t, \; t \in [0,T], \; Y^{\rm L}_T = 1,
	\end{cases}
\end{align*}
where the derivatives of $h^{\rm L}$ are evaluated at
\begin{align*}
	\big(\hat X_t,X^{\rm L}_t, Y_t^{\rm L}, \hat Y_t^{\rm F}, Z_t^{\rm L}, \hat Z_t^{\rm F}, \hat a_t(\hat X_t ) - \partial_x \hat a _t(\hat X_t) \hat X_t, \partial_x \hat a _t(\hat X_t ) \big),
\end{align*}
and
\begin{align*}
	\big( \hat a _t(\hat X_t ) - \partial_x \hat a_t(\hat X_t ) \hat X_t, \partial_x \hat a_t(\hat X_t ) \big)
	\in  \argmax_{({\rm a}^{\smallfont 1}, {\rm a}^{\smallfont 2}) \in A_k^\smallfont{2}(\hat X_\smallfont{t})} 
	\big\{h^{\rm L} \big(\hat X_t,X^{\rm L}_t, Y_t^{\rm L}, \hat Y_t^{\rm F}, Z_t^{\rm L}, \hat Z_t^{\rm F}, {\rm a}^{1}, {\rm a}^{2}\big)\big\},
\end{align*}
for $t \in [0,T]$,
{ where $A^2_k(x)$ is the set of $(a^1,a^2)\in\R^2$ such that $a^1+a^2 x\in A$ and $|a^2| \leq k$.} Optimising $h^{\rm L}$ with respect to $a^{1}$ gives 
\begin{align*}
	\hat a^{1}(y,x) \coloneqq \dfrac{y}{c_{\rm L}} -a^{2} x,\;\text{and}\; h^{\rm L} (x,x^\prime, y,y^\prime, z, z^\prime, \hat a^{1}, a^{2}) 
	=  \dfrac{1}2 \dfrac{y^2}{c_{\rm L}} + \dfrac{y y^\prime}{c_{\rm F}} + \sigma z 
	- a^{2} y^\prime x^\prime.
\end{align*}

Then, the maximisation with respect to $a^{2}$ gives $\hat a^{2}  \coloneqq -k{\rm sign}(y^\prime x^\prime)$. Therefore, by the maximum principle, if $(\hat a_1, \hat a_2)$ is a solution to Problem \eqref{eq:MP-ACLM-affine}, then there exists a tuple of $\F^W$-adapted processes $(\hat X,X^{\rm L}, Y^{\rm F}, Z^{\rm F},Y^{\rm L}, Z^{\rm L})$ such that
\begin{align}\label{eq.smp.clm}
	\begin{cases}
		\displaystyle\d \hat X_t = \bigg( \hat a^{2}_t \hat X_t + \hat a^{1}_t + \Pi_B \bigg( \dfrac{Y_t^{\rm F}}{c_{\rm F}} \bigg) \bigg) \d t + \sigma \d W_t, \;  t \in [0,T], \; X_0 = x_0,\\[0.8em]
		\displaystyle \d X^{\rm L}_t = - \bigg( \dfrac{Y^{\rm L}_t}{c_{\rm F}} -\hat a^{2}_t X^{\rm L}_t \bigg) \d t, \;  t \in [0,T], \; X^{\rm L}_0 = 0, \\[1em]
		\displaystyle	\d Y_t^{\rm F} = - \hat a^{2}_t Y_t^{\rm F} \d t + Z_t^{\rm F} \d W_t, \;  t \in [0,T], \; Y^{\rm F}_T = 1, \\[0.8em]
		\displaystyle\d Y_t^{\rm L} = {\color{black} 0 \times } \d t + Z_t^{\rm L} \d W_t, \;  t \in [0,T], \; Y^{\rm L}_T = 1.
	\end{cases}
\end{align}
We can solve the system explicitly for $(Y^{\rm L}, Z^{\rm L},Y_\cdot^{\rm F},Z^{\rm F}) = (1,0,\mathrm{e}^{k(T-\cdot)},0)$ which implies that $X^{\rm L}$ is a negative process. Therefore, the rest of the solution to the system is given, for $t\in[0,T]$, by $\hat a^1_t = c^{-1}_{{\rm L}} - k \hat X_t$, $\hat a^2 = k$,
\begin{align*} 
	\hat X_t = x_0 + \dfrac{t}{c_{\rm L}} +  \int_0^t   \Pi_{B} \bigg( \frac{\mathrm{e}^{k(T-s)}}{c_{\rm F}} \bigg)  \d s + \sigma W_t, \; \text{and} \;
	X_t^{\rm L} = -\frac{(\mathrm{e}^{kt}-1)}{kc_{\rm F}}.
\end{align*}
We deduce from the solution above the following candidate equilibrium for ACLM-$k$,
\[
\alpha^\star_k(t,X_t) = \frac{1}{c_{\rm L}} + k(X_t - \hat X_t),  \; \beta_k^\star(t) = \Pi_{B} \bigg( \frac{\mathrm{e}^{k(T-t)}}{c_{\rm F}} \bigg), \; t \in [0,T].
\]
It is proved in  \Cref{lemma-ACLM} that this pair of strategies is a solution to the ACLM-$k$ problem for $k\in[0,\bar K]$, with $\bar K\coloneqq\frac{1}{T} \log\big( \frac{1}{b_\circ} \big( a_\circ + \frac{1}{c_{\rm L}} +\frac{1}{2c_{\rm F}} \big) + \frac12 b_\circ c_{\rm F} \big) $. Moreover, the value of the leader and the follower are given by 
\begin{align*}
	\widetilde V_{\rm L}^k &= x_0 + \frac{T}{2 c_{\rm L}} + b_\circ T - \frac{1}{k} \bigg(b_\circ \log(b_\circ c_{\rm F}) + b_\circ - \frac{1}{c_{\rm F}} \bigg) , \\
	V_{\rm F}(\alpha_k^\star) &= x_0 + \frac{T}{c_{\rm L}} + b_\circ t_\circ^k - \frac{1}{2}c_{\rm F} b_\circ^2 t_\circ^k +  \frac{(b_\circ c_{\rm F} - 1)}{k c_{\rm F} } - \frac{(b_\circ^2 c_{\rm F}^2 - 1)}{4 k c_{\rm {F}}}, \; t_\circ^k \coloneqq T - \frac{\log(b_\circ  c_{\smallfont{\rm F}})}{k}.
\end{align*}
Notice that the above values are clearly non-decreasing with $k$ for $b_\circ$ large enough, and that if we could let $k$ go to $\infty$ above, the values would converge to that of the FB and ACL scenarii. This is however not possible since $k$ has to remain lower than $\bar K$, and this is why we cannot here fully characterise the general solution to the ACLM case. Notwithstanding, and though this is a rather informal statement, we expect $\widetilde V_{\rm L}^{\bar K}$ to be a relevant approximation for the value $\widetilde V_{\rm L}=V_{\rm L}$ of the ACLM scenario. It can also be checked that the limit as $k$ goes to $0$ is the value of the follower in the AOL case, thus showing that the latter is dominated by the ACLM case.

\subsubsection{Adapted closed-loop strategies}\label{sss:ACL}

Recall that when the information structure is assumed to be adapted closed-loop (with memory), both the leader and the follower observe the paths of the state $X$ and the underlying Brownian motion, and can use these observations to construct their strategies. Then, the leader’s strategy space and the follower’s response strategy space are
\begin{align*}
	\Ac \coloneqq   \big  \{ \alpha &: [0,T]\times\Omega\ni(t,\omega)\longmapsto a(t,W_{\cdot \wedge t}(\omega), X_{\cdot \wedge t}(\omega)):  \\
	a &:[0,T]\times \Cc([0,T],\R)^2 \longrightarrow A,\; \text{Borel}  \big \}, \\
	\Bc \coloneqq  \big\{ \beta &: [0,T]\times\Omega\ni(t,\omega)\longmapsto b(t,W_{\cdot\wedge t}(\omega),X_{\cdot\wedge t}(\omega),\alpha_{\cdot\wedge t}(\omega)):  \\
	b &:[0,T]\times \Cc([0,T],\R)^2\! \times \Lc^0([0,T],A)) \longrightarrow B,\; \text{Borel} \big\}.
\end{align*}

In our example, and under this particular information structure, the leader has actually enough information to deduce the effort of the follower. {Therefore, if the leader has enough \textit{bargaining power}, she may actually force the follower to undertake a recommended effort. More precisely, for $a_\circ$ sufficiently large, the leader would be able to punish the follower if he deviates from the desired action.} Indeed, suppose the leader wants to force the follower to perform the action $\hat\beta \in\Bc$ while doing herself an action $\hat\alpha \in \Ac$. One possible way to induce these strategies is for the leader to play $\alpha_t \coloneqq  \hat\alpha_t - p \mathbf{1}_{\{\beta_\smallfont{t}^{\smallfont{\circ}} \neq \hat \beta_\smallfont{t} \}}$ for some penalty coefficient $p \geq 0$, and where $\beta^{\circ}$ represents the `reference' effort
\begin{align*}
	\beta_t^{\circ} & \coloneqq \underset{\varepsilon \searrow 0}{\rm{lim sup}}\; \bigg\{ \frac{\upbeta_t^{\circ}-\upbeta_{t-\varepsilon}^{\circ}}{\varepsilon}\bigg\}, \; \text{with} \; \upbeta_t^{\circ} \coloneqq  X_t - \sigma W_t- \int_0^t \hat{\alpha}_s \d s, \; t\in[0,T].
\end{align*}
In words, by implementing the strategy $\alpha$ defined above, the leader threatens to punish the follower whenever the observed effort $\beta^{\circ}$ deviates from the recommended effort $\hat{\beta}$. Note that the definition of $\beta^{\circ}$ makes use of the fact that the leader observes the trajectories of both the state and the Brownian motion. In particular, such strategy $\alpha$ could not be implemented under the previous ACLM information structure. In general, we can say that the leader can `effectively punish' the follower {for not playing $\hat\beta$} if
\begin{align}\label{eq.punishment}
	\exists \alpha\in \Ac, J_{\rm F}(\alpha,\hat \beta)\geq J_{\rm F}(\alpha, \beta), \; \forall \beta\in \Bc, \; \text{and}\; J_{\rm L}(\alpha,\hat\beta)\geq J_{\rm L}( \hat \alpha,\hat \beta).
\end{align}
In words, {there exists an admissible strategy $\alpha\in\Ac$ such that} the optimal response of the follower to $\alpha$ is to play $\hat \beta$, and there is no detriment to the leader's utility when implementing the strategy $\alpha$ instead of $\hat \alpha$. We mention that in this example, we actually have the equality $J_{\rm L}(\alpha,\hat\beta)= J_{\rm L}( \hat \alpha,\hat \beta)$. More precisely, the leader can replicate the first-best solution by choosing $\hat{\alpha} = c_{\rm L}^{-1}$ and forcing the follower's action $\hat{\beta} = b_\circ $. Indeed, given the leader's strategy $\alpha_t \coloneqq  c_{\rm L}^{-1} - p \mathbf{1}_{\{\beta_\smallfont{t}^{\smallfont{\circ}} \neq b_\smallfont{\circ}\}}$, we have for all $\beta \in \Bc$
\begin{align*}
	J_{\rm F} (\alpha, b_\circ) - J_{\rm F} (\alpha, \beta) = \E^\P \bigg[ \int_0^T \bigg( b_\circ - \dfrac{c_{\rm F}}2 b_\circ^2 + p \mathbf{1}_{\{\beta_\smallfont{t}^{\smallfont{\circ}} \neq b_\smallfont{\circ}\}} - \beta_t + \dfrac{c_{\rm F}}2 \beta_t^2  \bigg) \drm t \bigg],
\end{align*}
and therefore the effectiveness of the punishment amounts to \( p \geq  (2c_{\rm F})^{-1}+  c_{\rm F} b_\circ^2/2-b_\circ\). This strategy can be implemented if the process $\alpha$ defined above is admissible, in the sense that it takes values in $A$. {Therefore, if $a_\circ$ is sufficiently large, for instance if Condition \eqref{cdtn:a_0_ACL} holds,
	then the solution to the ACL Stackelberg equilibrium in this example coincides with the first-best problem, whose solution is given in \Cref{lem:first-best}.}

\begin{remark}
	The previous argument shows that for any Stackelberg game under adapted closed-loop {\rm(ACL)} strategies for which \eqref{eq.punishment} holds with $(\hat \alpha,\hat \beta)$ being the solution to the first-best scenario, then the equality $V_{\rm L} = V_{\rm L}^{\rm FB}$ holds.
\end{remark}

\subsection{Closed-loop strategies}\label{ss:approach-example}

The approach we developed in this paper provides a way of studying and characterising a new, albeit natural, type of Stackelberg equilibrium in which the both players only have access to the trajectory of the state variable $X$. Consistent with the literature on stochastic differential games (see, for example, \cite{carmona2016lectures}), we simply refer to this concept of information as \textit{closed-loop} (CL). Under this information structure, both players can take into account \emph{only} the whole past trajectory of the state $X$ when making their decisions. Then, the leader’s strategy space and the follower’s response strategy space are respectively given by 
\begin{align*}
	\Ac \coloneqq  \big \{ \alpha &: [0,T]\times\Omega\ni(t,\omega)\longmapsto a(t, X_{\cdot \wedge t}(\omega)): \,
	a :[0,T]\times \Cc([0,T],\R) \longrightarrow A,\;\text{Borel}\big\}, \\
	\Bc \coloneqq  \big\{ \beta &: [0,T]\times\Omega\ni(t,\omega)\longmapsto  b(t,X_{\cdot\wedge t}(\omega),\alpha_{\cdot\wedge t}(\omega)): \\
	b &:[0,T]\times \Cc([0,T],\R)\times \Lc^0([0,T],A)) \longrightarrow B,\;\text{Borel} \big\}.
\end{align*}

As already mentioned in the introduction, allowing for path-dependency leads to a more realistic and sophisticated form of equilibrium and, consequently, more challenging to solve. In this case, the difficulty arises as the approaches developed above for solving Stackelberg open-loop or closed-loop memoryless equilibrium, mostly relying on the maximum principle, can no longer be used. To the best of our knowledge, there is currently no method developed in the literature for solving Stackelberg games within the framework of this very general, yet quite natural, information structure. The aim of this paper is, therefore, precisely to propose an approach, based on the dynamic programming principle and stochastic target problems, for characterising the solution for this type of equilibrium. 

\smallskip 

Our methodology, which consists of two main steps, is informally illustrated through the example presented at the top of this section. 
The first step is to use the follower's value function as a state variable for the leader's problem. 
More precisely, this value function solves a backward SDE, and by writing it in a forward way, we are able to reformulate the leader's problem as a stochastic control problem of an SDE system with stochastic target constraints. 
The second step consists in applying the methodology developed by \cite{bouchard2010optimal} to characterise such a stochastic control problem with target constraints through a system of Hamilton--Jacobi--Bellman equations.
Note that the reasoning developed in this section is quite informal, the aim being simply to illustrate our method; the reader is referred to \Cref{sec:general_formulation} onwards for the rigorous description of our approach.

\begin{remark}
	The stochastic target approach developed here is specifically designed for settings in which the leader cannot observe the underlying noise and must base her strategy solely on the observed trajectory of the state process. This structure, reminiscent of moral hazard settings in principal--agent problems, makes the approach naturally suited to path-dependent closed-loop strategies. By contrast, for `adapted' equilibrium concepts $(i)$--$(iv)$, where strategies may depend explicitly on the Brownian motion, our approach may no longer be applicable. Nevertheless, in such cases, classical tools such as the stochastic maximum principle or HJB methods can instead be used, and often lead to simpler formulations, as illustrated in this section through the example.
\end{remark}

\subsubsection{Reformulation as a stochastic target problem}\label{ss.leader.reformulation}

Recall that, given the leader's strategy $\alpha \in \Ac$, the follower's problem is given by \eqref{eq:pb-follower-example}. The idea of our approach to compute the Stackelberg equilibrium for closed-loop strategies is to consider the BSDE satisfied by the value function of the follower.\footnote{Actually, one should switch to the weak formulation of the problem in order to consider the BSDE representation of the follower's value. Nevertheless, once again our goal here is simply to illustrate our method, and we refer to \Cref{sec:general_formulation} for the rigorous formulation.} With this in mind, we introduce the dynamic value function of the follower given by
\begin{align*}
	Y_t \coloneqq  \underset{\beta \in \Bc}{{\esssup}}^\P \bigg\{\E^\P \bigg[ X_T - \dfrac{c_{\rm F}}2  \int_t^T \beta_s^2 \d s \bigg| \Fc_t \bigg]\bigg\}, \; t \in [0,T],
\end{align*}
where the state variable $X$ follows the dynamics given by \eqref{eq:dynamic_x}. It is easy to show that, by introducing the appropriate Hamiltonian, \textit{i.e.}
\begin{align*}
	H^{\rm F} (t, z, a) 
	\coloneqq  \sup_{b \in B} \bigg\{ (a+b) z - \dfrac{c_{\rm F}}2 b^2 \bigg\}, \; (t,z,a)  \in [0,T] \times \R \times A,
\end{align*}
the value function of the follower for a given $\alpha \in \Ac$ is a solution to the following BSDE
\begin{align*}
	\d Y_t = - H^{\rm F} (t,Z_t,\alpha_t) \mathrm{d}t + Z_t \mathrm{d}X_t, \; t \in [0,T], \; Y^\alpha_T = X_T,
\end{align*}
for some $Z \in \Zc$, where $\Zc$ is a set of $\F$-adapted processes taking value in $\R$ and satisfying appropriate integrability conditions. The maximiser of the Hamiltonian is naturally given by the functional $b^\star(z) = \Pi_{\tilde B} (z)/c_{\rm F}$, $z \in \R$, where $\tilde B \coloneqq [0, b_\circ c_{\rm F}]$. For a given $\alpha \in \Ac$ chosen by the leader, we are thus led to consider the FBSDE system
\begin{align}\label{eq:ACL-follower-example}
	\begin{cases}
		\displaystyle	\d X_t = \bigg( \alpha_t + \dfrac{1}{c_{\rm F}} \Pi_{\tilde B} \big( Z_t \big) \bigg) \d t + \sigma \d W_t, \; t \in [0,T], \; X_0 = x_0, \\[0.8em]
		\displaystyle	\d Y_t = \dfrac{1}{2 c_{\rm F}} \Pi^2_{\tilde B} \big( Z_t \big) \d t 
		+ \sigma Z_t \d W_t, \; t \in [0,T], \; Y_T = X_T.
	\end{cases}
\end{align}
Consequently, the leader's problem defined by \eqref{eq:pb-leader-example} becomes
\begin{align}\label{eq:leader_FBSDE}
	V_{\rm L} (x_0) = \sup_{\alpha \in \Ac} \bigg\{\E^\P \bigg[ X_T - \dfrac{c_{\rm L}}2  \int_0^T \alpha_t^2 \d t \bigg]\bigg\}, \text{ subject to the FBSDE system \eqref{eq:ACL-follower-example}}.
\end{align}

Unfortunately, the literature on the optimal control problem of FBSDEs is quite scarce and, to the best of our knowledge, is not able to accommodate the closed-loop scenario described above, see for instance \cite{yong2010optimality} or \cite{wu2013general}. 
Nevertheless, to continue the reformulation of the leader's problem, one can write the BSDE in \eqref{eq:ACL-follower-example} as a forward SDE for a given initial condition $y_0 \in \R$, and thus consider the following SDE system
\begin{align}\label{eq:ACL-follower-example-FFSDE}
	\begin{cases}
		\displaystyle	\d X_t = \bigg( \alpha_t + \dfrac{1}{c_{\rm F}} \Pi_{\tilde B} ( Z_t) \bigg) \d t + \sigma \d W_t, \; t \in [0,T], \; X_0 = x_0,\\[0.8em]
		\displaystyle	\d Y_t = \dfrac{1}{2 c_{\rm F}} \Pi^2_{\tilde B} ( Z_t)  \d t 
		+ \sigma Z_t \d W_t, \; t \in [0,T], \; Y_0 = y_0,
	\end{cases}
\end{align}
for some $(\alpha,Z) \in \Ac \times \Zc$.
However, by doing so, one needs to take into account an additional constraint, in order to ensure that the equality $Y_T = X_T$ holds with probability one at the end of the game. 
More precisely, one of the main results of our paper, stated for the general framework in \Cref{prop.equivformulationleader}, is that the leader's problem originally defined here by \eqref{eq:pb-leader-example} is equivalent to the following stochastic \textit{control and target} problem
\begin{align}\label{eq.ex.leader.reform}
	\widehat V_{\rm L} (x_0) &\coloneqq  \sup_{y_\smallfont{0} \in \R}\big\{ \widetilde V_{\rm L} (0,x_0,y_0)\big\}, \\ \text{where} \; 
	\widetilde V_{\rm L} (0,x_0,y_0)
	&\coloneqq \sup_{(Z,\alpha) \in \Cf(x_\smallfont{0},y_\smallfont{0})} \bigg\{\E^{\P}  \bigg[ X_T - \dfrac{c_{\rm L}}2  \int_0^T \alpha_t^2 \d t \bigg]\bigg\}, \; (x_0,y_0) \in \R^2, \nonumber
\end{align}
subject to the system \eqref{eq:ACL-follower-example-FFSDE}, and where $\mathfrak{C}(x_0,y_0)\coloneqq \{ (Z,\alpha)\in\Zc\times\Ac :Y_T = X_T, \; \P\as \}$. 

\begin{remark}\label{rk:markovian}
	In the above reformulation of the leader's problem, $\widetilde V_{\rm L}$ corresponds to the value function of an optimal control problem with stochastic target constraint. More importantly, while the original leader's problem was non-Markovian due to her closed-loop $($path-dependent$)$ strategy $\alpha$, the consideration of $Y$ as an additional state variable now makes her control problem Markovian. In particular, the strategy $\alpha_t$ at time $t \in [0,T]$, originally defined as a measurable function of the path of $X$ up to time $t$, can be transformed into a function of $X_t$ and $Y_t$. This property of the leader's reformulated problem is standard in continuous-time principal-agent problems, and was therefore expected here since we are using a similar approach. This also highlights the fact that in our formulation, the leader will generically use the whole path of $X$ in order to design the equilibrium, since $Y$ is generally \emph{not} a Markovian function of $X$.
\end{remark} 

\subsubsection{Interpretation of the reformulated problem}\label{sec.ex.cl.interp}

For fixed $y_0 \in \R$, the leader's problem is to choose a couple $(Z,\alpha)$ of admissible controls. With this in mind, given the state $X$ observable in continuous time, she can construct an additional process $Y$, starting from $Y_0 = y_0$, with the following dynamics
\begin{align*}
	\d Y_t = - H^{\rm F} (t,Z_t,\alpha_t) \mathrm{d}t + Z_t \mathrm{d}X_t, \; t \in [0,T].
\end{align*}
Note $Y$ can be constructed based solely on the observation through time of the path of $X$, and in particular does not require any knowledge of the follower's control $\beta$ nor of the underlying Brownian motion $W$. Now, the couple $(Z,\alpha)$ of admissible processes chosen by the leader should be such that the terminal condition $Y_T = X_T$ is satisfied $\P$--a.s. Indeed, under this important condition, the follower's problem originally defined by \eqref{eq:pb-follower-example} can be rewritten as
\begin{align*}
	V_{\rm F} (\alpha) \coloneqq  \sup_{\beta \in \Bc} \bigg\{\E^\P \bigg[ X_T - \dfrac{c_{\rm F}}2  \int_0^T \beta_t^2 \d t \bigg] \bigg\}
	= \sup_{\beta \in \Bc} \bigg\{\E^\P \bigg[ Y_T - \dfrac{c_{\rm F}}2  \int_0^T \beta_t^2 \d t \bigg]\bigg\}.
\end{align*}
With the knowledge of the dynamic of $Y$, as well as the leader's controls $(Z,\alpha)$, the follower sees that
\begin{align*}
	V_{\rm F} (\alpha) 
	&= y_0 + \sup_{\beta \in \Bc} \bigg\{\E^\P \bigg[ - \int_0^T H^{\rm F} (t,Z_t,\alpha_t) \mathrm{d}t
	+ \int_0^T Z_t \d X_t - \dfrac{c_{\rm F}}2  \int_0^T \beta_t^2 \d t \bigg] \bigg\}\\
	&= y_0 + \sup_{\beta \in \Bc}\bigg\{ \E^\P \bigg[ \int_0^T \bigg( Z_t \beta_t - \dfrac{c_{\rm F}}2 \beta_t^2 \bigg) \mathrm{d}t \bigg] \bigg\}- \E^\P \bigg[ \int_0^T \sup_{b \in B} \Big\{ b Z_t - \dfrac{c_{\rm F}}2 b^2 \Big\} \mathrm{d}t \bigg],
\end{align*}
making it clear, at least heuristically here, that his best response strategy coincides with the maximiser of the Hamiltonian, namely $\beta_t \coloneqq \Pi_{\tilde B} (Z_t) / c_{\rm F}$, $t \in [0,T]$. This optimal choice provides him with the maximum value, for all $(\alpha, Z) \in \Ac \times \Zc$, \textit{i.e.} $V_{\rm F} (\alpha) = y_0$. Overall, for a given $y_0 \in \R$, which actually coincides with the follower's value, the leader designs her strategy through the couple $(Z,\alpha)$ such that $Y_T = X_T$ is satisfied $\P\as$ for a well-chosen process $Y$, inducing the follower's optimal response $\beta_\cdot \coloneqq  \Pi_{\tilde B} (Z_\cdot) / c_{\rm F}$. Note that the leader should not only communicate to the follower the couple $(Z,\alpha)$ of controls, but she should also indicate how these controls are designed, namely the construction of the underlying process $Y$: all these ingredients are part of the strategy implemented by the leader.

\subsubsection{Characterisation of the equilibrium}\label{sss:CL_sto_target_eq}

Given the reformulation of the leader's problem as a stochastic control problem with stochastic target constraint, the second step consists now in applying the methodology in \cite{bouchard2010optimal} to solve the latter problem and thus obtain a characterisation of the corresponding Stackelberg equilibrium. Recall first that in our illustrative example, the leader's reformulated problem is given by \eqref{eq.ex.leader.reform}, namely
\begin{align*}
	\widehat V_{\rm L} (x_0) &\coloneqq  \sup_{y_\smallfont{0} \in \R}\big\{ \widetilde V_{\rm L} (0,x_0,y_0)\big\}, \;
	\widetilde V_{\rm L} (t,x,y) \coloneqq  \sup_{(Z,\alpha) \in \Cf(t,x,y)} \bigg\{\E^{\P}  \bigg[ X_T^{t,x,Z,\alpha} - \dfrac{c_{\rm L}}2  \int_t^T \alpha_s^2 \d s \bigg]\bigg\}, 
\end{align*}
where for $(t,x,y) \in [0,T] \times \R^2$, the set $\Cf(t,x,y)$ is defined by
\begin{align*}
	\Cf(t,x,y) \coloneqq  \big\{ (Z,\alpha)\in\Zc\times\Ac: Y_T^{t,y,Z,\alpha} = X_T^{t,x,Z,\alpha}, \; \P\as  \big\},
\end{align*}
with the controlled state variables $X$ and $Y$ satisfying the following dynamics
\begin{align}
	\begin{cases}
		\displaystyle	\drm X_s^{t,x,Z,\alpha} = \bigg( \alpha_s + \dfrac{1}{c_{\rm F}} \Pi_{\tilde B} ( Z_s) \bigg) \d s + \sigma \d W_s, \; s \in [t,T], \; X_t^{t,x,Z,\alpha} = x, \\[0.8em]
		\displaystyle	\drm Y_s^{t,y,Z,\alpha} = \dfrac{1}{2 c_{\rm F}} \Pi^2_{\tilde B} ( Z_s) \d s	+ \sigma Z_s \d W_s, \; s \in [t,T], \; Y_t^{t,y,Z,\alpha} = y.
	\end{cases}
\end{align}
In particular, for fixed $(t,x,y) \in [0,T] \times \R^2$, $\widetilde V_{\rm L} (t,x,y)$ corresponds to the dynamic value function of an optimal control problem with stochastic target constraints. 
Thus, we define for any $t \in [0,T]$ the target reachability set
\begin{align*}
	V_G(t)\coloneqq  \big\{ (x,y)\in\R^2:  \exists (Z,\alpha)\in\Zc\times\Ac, \; Y_T^{t,y,Z,\alpha} = X_T^{t,x,Z,\alpha},\;  \P\as \big\}.
\end{align*}
An intermediary but important result for our approach, see \Cref{lemma.boundaries}, is to show that the closure of the reachability set $V_G(t)$ coincides with the following set
\begin{align*}
	\hat{V}_G(t)\coloneqq  \{ (x,y) \in \R^2 : w^-(t,x) \leq y \leq w^+(t,x) \},
\end{align*}
for appropriate auxiliary functions $w^-$ and $w^+$. 
It is then almost straightforward to extend the approach in \cite{bouchard2010optimal} to characterise the leader's value function $\widetilde V_{\rm L}$ as the solution to a specific system of Hamilton--Jacobi--Bellman (HJB) equations and therefore determine the corresponding optimal strategy. 
More precisely, this can be achieved in three main steps. 
First, the auxiliary functions $w^-$ and $w^+$ can be characterised as solutions (in an appropriate sense) to specific HJB equations. 
Then, the leader's value function $\widetilde V_{\rm L}$ satisfies another specific HJB equation on each of these boundaries. 
Finally, in the interior of the domain, $\widetilde V_{\rm L}$ is a solution to the classical HJB equation, but with the non-standard boundary conditions obtained in the previous step, see \Cref{thm.verification}. 
These three steps are described below in the framework of our illustrative example.

\paragraph*{The auxiliary functions} The lower and upper boundaries $w^-$ and $w^+$ can be characterised as the solutions to the following specific HJB equations on $[0,T) \times \R$,
\begin{align*}
	- \partial_t w^-  - \inf_{ (z,a) \in N(p)}   \big\{h^b( \partial_x w^- , \partial_{xx} w^- ,z,a)\big\}  &= 0, \quad w^-(T,x) = x, \\
	- \partial_t w^+  - \sup_{ (z,a) \in N(p)} \big\{h^b( \partial_x w^+ , \partial_{xx} w^+ ,z,a)\big\}  &= 0, \quad w^+(T,x) = x, \\
	\text{ with } h^b( p,q,z,a)  \coloneqq - \dfrac{1}{2 c_{\rm F}} \Pi^2_{\tilde B} (z) +  \bigg( a + \dfrac{1}{c_{\rm F}} \Pi_{\tilde B} (z) \bigg)p &+ \dfrac12 \sigma^2 q, 
\end{align*}
for all $(p,q) \in \R^2$ and $(z,a) \in  N(p) \coloneqq \{(z,a) \in \R \times A : \sigma z = \sigma  p\}$.
Since $\sigma \neq 0$, the constraint set $N$ boils down to $N(p) = \{(p,a):a \in A\}$, for all $p \in \R$. Using in addition the ansatz $\partial_x w^{\pm} \in \tilde B$, one obtains the following HJB equations on $(t,x) \in [0,T) \times \R$
\begin{align*}
	- \partial_t w^-(t,x) - \dfrac12 \sigma^2 \partial_{xx} w^-(t,x) - \dfrac{1}{2 c_{\rm F}} \big(\partial_x w^-(t,x) \big)^2 - \inf_{ a \in A} \big\{ \partial_x w^-(t,x) a \big\} &= 0, \\
	- \partial_t w^+(t,x) - \dfrac12 \sigma^2 \partial_{xx} w^+(t,x) - \dfrac{1}{2 c_{\rm F}} \big(\partial_x w^+(t,x) \big)^2 - \sup_{ a \in A} \big\{ \partial_x w^+(t,x) a \big\} &= 0,
\end{align*}
with terminal condition $w^-(T,x) = w^+(T,x) = x$, $x \in \R$. Recalling that $A=[-a_\circ,a_\circ]$, one can explicitly compute the auxiliary functions, solution to the previous HJB system, as 
$w^{\pm} (t,x) = x + \ell^\pm (t)$ for $(t,x) \in [0,T] \times \R$ where
\begin{align}\label{eq:lpm}
	\ell^-(t)\coloneqq \bigg(\frac{1}{2c_{\rm F}}-a_\circ\bigg)(T-t),\; \ell^+(t)\coloneqq \bigg(\frac{1}{2c_{\rm F}}+a_\circ\bigg)(T-t), \; t \in [0,T].
\end{align}

\begin{remark}
	In the context of this example, to have meaningful, \emph{i.e.} finite, solutions, the boundedness assumption on $A$ is necessary. 
	Though the methodology developed in {\rm \cite{bouchard2010optimal}} can cover the case of unbounded action sets, this will require imposing growth conditions that, in turn, will rule out the framework of the current example. 
	Moreover, the possibility of discontinuous or exploding solutions requires working with viscosity solutions to the above {\rm PDEs}.
\end{remark}

\paragraph*{The value function at the boundaries} The second step is to determine the HJB equations satisfied by the value function $\widetilde V_{\rm L}(t,x,y)$ on the boundaries, \textit{i.e.} on $\{y = w^-(t,x) \}$ and $\{y = w^+(t,x) \}$, for all $(t,x) \in [0,T]\times \R$. With this in mind, we define for all $p \coloneqq (p_1,p_2)^\top \in \R^2$, $q \in \R^{2\times2}$ and $(z,a) \in \R \times A$,
\begin{align*}
	{\rm h} (p,q,z,a)  \coloneqq   - \dfrac{c_{\rm L}}2  a^2 + \bigg( a + \dfrac{\Pi_{\tilde B} (z)}{c_{\rm F}}  \bigg) p_1 
	+ \dfrac{\Pi^2_{\tilde B} (z) }{2 c_{\rm F}} p_2
	+ \dfrac12 \sigma^2 q_{11}
	+ \dfrac12 \sigma^2 z^2 q_{22}
	+ \sigma^2 z q_{12}.
\end{align*}
We then introduce the following Hamiltonians, for all $(t,x,p,q) \in [0,T] \times \R \times \R^2 \times \R^3$,
\begin{align*}
	{\rm H}^-(t,x,p,q)\coloneqq  \sup_{ (z,a) \in \Zc^\smallfont{-}(t,x) }\big\{{\rm h} (p,q,z,a)\big\}, \; \text{and} \;
	{\rm H}^+(t,x,p,q)\coloneqq  \sup_{ (z,a) \in \Zc^\smallfont{+}(t,x) }\big\{{\rm h} (p,q,z,a)\big\},
\end{align*}
in which the sets $\Zc^\pm(t,x)$ are defined by
\begin{align*}
	\Zc^\pm(t,x) &\coloneqq  \big\{(z,a) \in \R \times A  :  \sigma z = \sigma \partial_x w^\pm, \; \text{and} \; \pm \partial_t w^\pm \pm h^b ( \partial_x w^\pm, \partial_{xx} w^\pm,z,a) \geq 0 \big\}.
\end{align*}
The value function $\widetilde V_{\rm L}$ should satisfy on $\{y = w^-(t,x)\}$ the following equation
\begin{align*}
	- \partial_t v(t,x,y) - {\rm H}^-(t,x, \partial_{\rm x} v(t,x,y),\partial^2_{\rm x} v(t,x,y)) = 0, \; (t,x,y) \in [0,T) \times \R^2, 
\end{align*}
with terminal condition $v(T,x,w^-(T,x)) = x$, $x \in \R$.\footnote{Here, $\partial_{\rm x} v(t,x,y)$ and $\partial^2_{\rm x} v(t,x,y)$ denote respectively the gradient and Hessian of the function $v$ in \emph{both} space variables ${\rm x}\coloneqq (x,y)$.} Given the previous HJB equation satisfied by $w^-$, it is clear that $\Zc^-(t,x) = \{(1,-a_\circ)\}$, for all $(t,x) \in [0,T) \times \R$. We thus obtain a standard PDE for $\widetilde V_{\rm L}$ on $\{y = w^-(t,x) \}$, for $(t,x) \in [0,T] \times \R$,
\begin{align*}
	- \partial_t v 
	+ \dfrac{1}2 c_{\rm L} a_\circ^2 
	- \bigg( \dfrac{1}{c_{\rm F}} - a_\circ \bigg) \partial_x v
	- \dfrac{1}{2 c_{\rm F}} \partial_y v 
	- \dfrac12 \sigma^2 \partial_{xx} v
	- \dfrac12 \sigma^2 \partial_{yy} v
	- \sigma^2 \partial_{xy} v  = 0, 
\end{align*}
which leads to the following solution
\begin{align*}
	\widetilde V_{\rm L} (t,x,w^-(t,x)) = x + \bigg( - a_\circ - \dfrac12 c_{\rm L} a_\circ^2 + \dfrac{1}{c_{\rm F}} \bigg) (T-t), \; (t,x) \in [0,T] \times \R.
\end{align*} 

On the other hand, on $\{y = w^+(t,x)\}$, the value function should be solution to
\begin{align*}
	- \partial_t v(t,x,y) - {\rm H}^+(t,x, \partial_{\rm x} v(t,x,y),\partial^2_{\rm x} v(t,x,y)) = 0, \; (t,x,y) \in [0,T) \times \R^2, 
\end{align*}
with $v(T,x,w^+(T,x)) = x$, $x \in \R$. Through similar computations, one obtains
\begin{align*}
	\widetilde V_{\rm L} (t,x,w^+(t,x)) = x + \bigg( a_\circ - \dfrac12 c_{\rm L} a_\circ^2 + \dfrac{1}{c_{\rm F}} \bigg) (T-t), \; (t,x) \in [0,T] \times \R.
\end{align*}

\paragraph*{Value function inside the domain} Finally,  for $(t,x) \in [0,T] \times \R$ but inside the domain, the value function $\widetilde V_{\rm L}$ is solution to the classical HJB equation for stochastic control
\begin{align*}
	&- \partial_t v(t,x,y) - H^{\rm L} (\partial_{\rm x} v(t,x,y),\partial^2_{\rm x} v(t,x,y)) = 0, \; y \in (w^-(t,x), w^+(t,x)) \\
	\text{where}\; &\ H^{\rm L}(p,q) \coloneqq  \sup_{ (z,a) \in  \R \times A } \big\{{\rm h} (p,q,z,a)\big\}, \; (p,q) \in \R^2 \times \R^{2\times2},
\end{align*}
but instead of the usual terminal condition, we need to enforce the specific boundary conditions obtained in the previous step.
The previous HJB equation can be written as
\begin{align}\label{eq:HJB_example}
	-\partial_t v 
	- \sup_{ z \in \R } \bigg\{ \dfrac{1}{c_{\rm F}} \Pi_{\tilde B} (z) \partial_x v 
	+ \dfrac{1}{2 c_{\rm F}} \Pi^2_{\tilde B} (z) \partial_y v
	+ \dfrac12 \sigma^2 z^2 \partial_{yy} v
	+ \sigma^2 z \partial_{xy} v \bigg\} & \\
	- \sup_{ a \in A } \bigg\{ a \partial_x v 
	- \dfrac{c_{\rm L}}2  a^2 \bigg\}
	- \dfrac12 \sigma^2 \partial_{xx} v &= 0. \nonumber
\end{align}
It is relatively straightforward to show that one can look for a solution of the form
\[
v(t,x,y)=x+u(t,y-x),\; (t,x)\in[0,T]\times\R,\; y\in[w^-(t,x),w^+(t,x)],
\]
where the map $u$ now solves the PDE on the domain $t\in[0,T),\; \xi\in(\ell^-(t),\ell^+(t))$
\begin{align*}
	\begin{cases}
		\displaystyle \partial_tu +\sup_{a\in A}\bigg\{a(1-\partial_\xi u)-\frac{c_{\rm L}}2a^2\bigg\} \\
		\displaystyle \quad \; \; \, +\sup_{z\in\R}\bigg\{ \dfrac{\Pi_{\tilde B} (z)}{c_{\rm F}}  (1-\partial_\xi u) 
		+ \dfrac{\Pi^2_{\tilde B} (z)}{2 c_{\rm F}}  \partial_\xi u + \dfrac{\sigma^2}2  (z-1)^2 \partial_{\xi\xi} u \bigg\}=0,\\
		\displaystyle u(t,\ell^-(t))=\bigg(\frac1{c_{\rm F}}-a_\circ-\frac{c_{\rm L}}{2}a_\circ^2\bigg)(T-t),\;
		\displaystyle u(t,\ell^+(t))=\bigg(\frac1{c_{\rm F}}+a_\circ-\frac{c_{\rm L}}{2}a_\circ^2\bigg)(T-t),
	\end{cases}
\end{align*}
where the functions $\ell^-$ and $\ell^+$ were defined in \eqref{eq:lpm}.
As far as we know, the previous PDE does not admit explicitly solutions, but can be solved numerically. Once this is achieved, it remains to maximise $u(0,\xi)$ over $\xi \in (\ell^-(0), \ell^+(0))$. Given such a maximiser $\xi^\star \in (\ell^-(0), \ell^+(0))$, the corresponding $y_0 \coloneqq x + \xi^\star \in [w^-(0,x), w^+(0,x)]$ and the associated value $v(0,x,y_0) = u(0,\ell)$ will respectively give the follower's and leader's value functions for the initial condition $X_0 = x$. The numerical results are presented next.


\begin{remark}
Though not the focus of this work, it will be lackadaisical on our end not to digress on Markovian solutions to the Stackelberg game.\footnote{We are grateful to an anonymous referee for encouraging us to include this discussion.} As highlighted in \Cref{rk:markovian}, a closed-loop equilibrium generally fails to be Markovian. Nonetheless, in specific settings, it may be possible to rewrite the corresponding strategies as Markovian functions of the state, evaluated along its optimal trajectory. We refer to \Cref{rmk.verification} $(iii)$ for a discussion of when such a Markovian representation may hold in the general case. However, even if such a Markovian representation of the strategies is possible, this does not necessarily constitute an equilibrium within the class of Markovian strategies. 

\smallskip

In order to assess whether a closed-loop equilibrium can also qualify as a Markovian equilibrium, it would be necessary to first specify what constitutes a Markovian solution in the Stackelberg framework. A first natural choice for the strategy spaces is
\begin{align*}
\begin{split}
	\Ac &\coloneqq   \big\{  \alpha : [0,T]\times\Omega\ni(t,\omega)\longmapsto  a(t,X_t(\omega)): 
	a :[0,T]\times \R \longrightarrow A,\; \text{Borel}\big\}, \\
	\Bc &\coloneqq  \big\{ \beta : [0,T]\times\Omega\ni(t,\omega)\longmapsto b(t,X_t(\omega),\alpha_t(\omega)):
	b :[0,T]\times \R\times A \longrightarrow B,\; \text{Borel} \big\}.
\end{split}
\end{align*}
However, with this choice, the notion of Markovian equilibrium would align---at least in Markovian settings---with the adapted feedback solution concept. Indeed, as already mentioned in \Cref{sss:AF}, adapted feedback strategies in Markovian frameworks turn out to be independent of the Brownian motion and would therefore coincide with Markovian strategies as defined above. Moreover, the above definition of $\Bc$ prevents the follower's admissible strategies to depend on the whole control process of the leader, which would result in a \textit{local} notion of Stackelberg equilibrium, as also observed in the adapted feedback case. This would hinder a meaningful comparison with the closed-loop equilibrium studied in this section, as consistency in the definition of Stackelberg equilibrium seems crucial, even if the admissible strategy spaces differ.

\smallskip

Alternatively, to maintain a global notion of Stackelberg equilibrium, one may slightly change the above strategy spaces, to allow the follower's response to depend on the entire past trajectory of the leader's strategy, namely $\alpha_{\cdot\wedge t}(\omega)$ at time $t \in [0,T]$.
To characterise such a Markovian equilibrium, one possibility would be to apply the first part of our approach described in \Cref{ss.leader.reformulation}, namely rewriting the leader's problem as a control problem over an FBSDE system, \textit{i.e.} \eqref{eq:leader_FBSDE}. 
The second step of our methodology, relying on the stochastic target approach, inherently leads to path-dependent strategies and is therefore unlikely to be suitable in this Markovian context. One may need to rely on existing techniques developed for the control of FBSDE systems. However, while these are well-understood in Markovian settings, the difficulty here lies in the fact that the follower’s strategy may depend on the entire trajectory of the leader’s control, which complicates the application of such methods and raises non-trivial challenges. This is an interesting direction that we believe would deserve further investigation in future research.
\end{remark}

\subsubsection{Comparison with other solution concepts and numerical results}\label{sss:numerics}

For the numerical results\footnote{In \Cref{fig1:numerics,fig2:numerics,fig3:numerics,fig4:numerics}, the `CL curves' in the right-hand plots labelled (b) appear piecewise constant. This is not a property of the true solution, but rather an artefact of the numerical method. Specifically, it results from the space discretisation: the grid resolution cannot be refined arbitrarily due to convergence and stability constraints, leading to the observed behaviour.}, we first consider a benchmark scenario in \Cref{fig1:numerics}, with parameters $T=1$, $x=1$, $c_{\rm F} = c_{\rm L} = 1$, $\sigma =1$, $a_\circ = 10$, and $b_\circ = 3$.
We then study in \Cref{fig3:numerics} a scenario in which the leader's cost of effort increases to $c_{\rm L} = 1.25$, and conversely in \Cref{fig2:numerics} when now the follower's cost of effort increases to $c_{\rm F} = 1.25$. Finally, we represent in \Cref{fig4:numerics} the impact of an increase of $a_\circ$ from $10$ to $15$. Note that in these four scenarii, $a_\circ$ is chosen sufficiently large so that \Cref{cdtn:a_0_ACLM} and \Cref{cdtn:a_0_ACL} are satisfied. {Moreover, in all the simulations below, the curve `ACLMk' represents the value of the problem ACLM-$\bar{K}$, described in \Cref{sss:ACLM}.}

\begin{figure}[!ht]
	\centering
	\begin{subfigure}{.49\textwidth}
		\centering
		\includegraphics[scale=0.47]{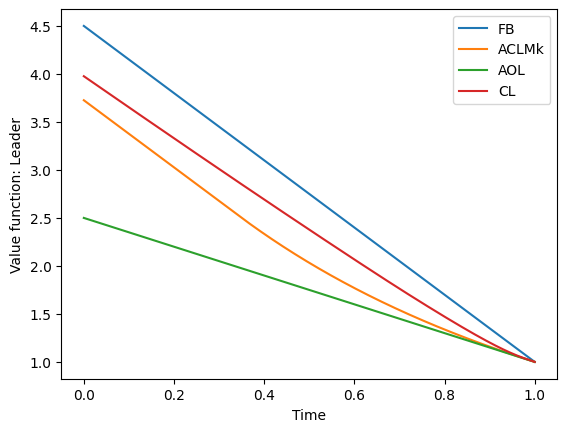}
		\centering
		\caption{\small Leader's value function.}
		\label{fig1:value_leader}
	\end{subfigure}
	\begin{subfigure}{.49\textwidth}
		\includegraphics[scale=0.47]{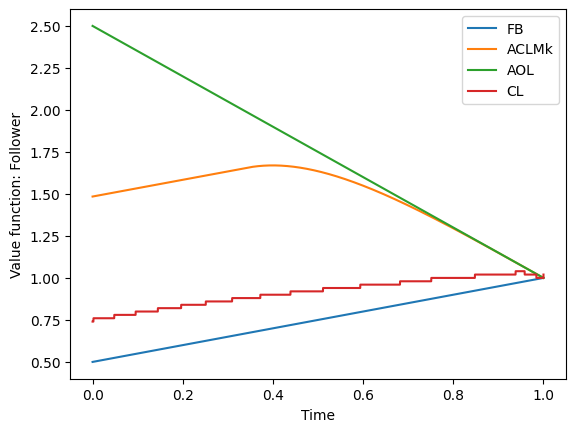}
		\centering
		\caption{\small Follower's value function.}
		\label{fig1:value_follower}
	\end{subfigure}
	\caption{\small Comparison of the value functions, for $c_{\rm F} = c_{\rm L} = 1$, and $a_\circ = 10$.}
	\label{fig1:numerics}
\end{figure}

\smallskip

We first remark that for the four sets of parameters, we have the following inequalities for the leader's value function,
\begin{align*}
	V_{\rm L}^{\rm AOL} = V_{\rm L}^{\rm AF}< V_{\rm L}^{\rm ACLM-\bar{K}} < V_{\rm L}^{\rm CL} < V_{\rm L}^{\rm ACL} = V_{\rm L}^{\rm FB},
\end{align*}
and the converse inequalities for the follower's value. Most of these inequalities were to be expected, as already mentioned in \eqref{eq:value_info}. In addition, the inequality $V_{\rm L}^{\rm AOL} < V_{\rm L}^{\rm CL}$ is straightforward using the explicit solution for the AOL case. Indeed, such solution is a pair of constant effort, which is obviously an admissible strategy for the CL information structure. Moreover, it is easy to show that, if the leader decides to commit to the strategy $1/c_{\rm L}$ in the CL case, the follower's best response will still be $1/c_{\rm F}$, and therefore $V_{\rm L}^{\rm AOL} \leq V_{\rm L}^{\rm CL}$. To obtain the strict inequality, one can notice that such equilibrium would actually correspond to forcing $z=1$ (instead of $z \in \R$) in the HJB equation \eqref{eq:HJB_example}. Using an appropriate comparison principle for PDE, we thus obtain the desired inequality.
Furthermore, for these chosen sets of parameters, the leader's value in the closed-loop equilibrium is higher than her value in the ACLM-$\bar K$ scenario. Recall that this scenario provides the optimal ACLM strategy when the derivative of the leader's effort is assumed to be bounded by $\bar K$, and where $\bar K$ is chosen so that the leader's strategy remains admissible, \textit{i.e.} takes value in $A$. The value obtained in this scenario should therefore be comparable with the value in the ACLM case, and the numerical results therefore suggest that the value obtained by the leader in the CL case is greater than the value she would obtain in an ACLM situation. More importantly for our study, the numerical results highlight that the behaviour of the value functions over time is significantly different, confirming that our proposed closed-loop equilibrium leads to fundamentally different strategies with respect to the ACLM information structure, even in this very simple example.

\smallskip

Comparing \Cref{fig1:numerics} with \Cref{fig3:numerics} in more detail, one can observe that the increase in the leader's cost of effort negatively impacts both her and the follower's value in any equilibrium concepts. This is a logical outcome, since if the leader's effort cost is higher, she will exert less effort, which negatively impacts the terminal value of the output process for both players, in every scenario.

\smallskip

\begin{figure}[!h]
	\centering
	\begin{subfigure}{.49\textwidth}
		\centering
		\includegraphics[scale=0.47]{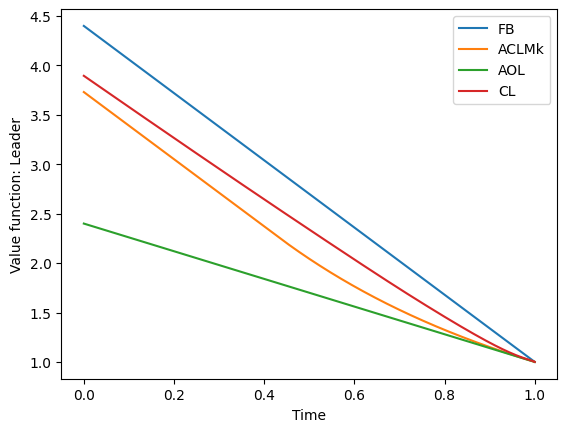}
		\centering
		\caption{\small Leader's value function.}
		\label{fig3:value_leader}
	\end{subfigure}
	\begin{subfigure}{.49\textwidth}
		\includegraphics[scale=0.47]{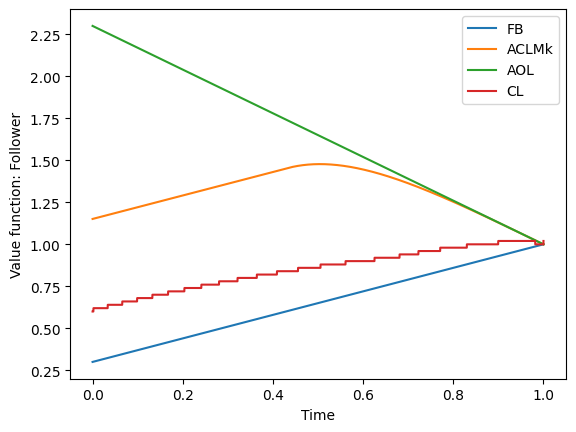}
		\centering
		\caption{\small Follower's value function.}
		\label{fig3:value_follower}
	\end{subfigure}
	\caption{\small Comparison of the value functions, for $c_{\rm F} = 1$, $c_{\rm L} = 1.25$, and $a_\circ = 10$.}
	\label{fig3:numerics}
\end{figure}

Comparing now \Cref{fig1:numerics} with \Cref{fig2:numerics}, we can observe that when the follower's cost of effort slightly increases, it also negatively impacts both his and the leader's value for almost all concepts of equilibrium, for the same reason as above, except in the ACL/first-best case. Indeed, in this scenario, the leader's value function remains unchanged, as the follower will always exert the maximal effort $b_\circ$. Therefore, only the follower's value is impacted by the increase in his cost.

\begin{figure}[!h]
	\centering
	\begin{subfigure}{.49\textwidth}
		\centering
		\includegraphics[scale=0.47]{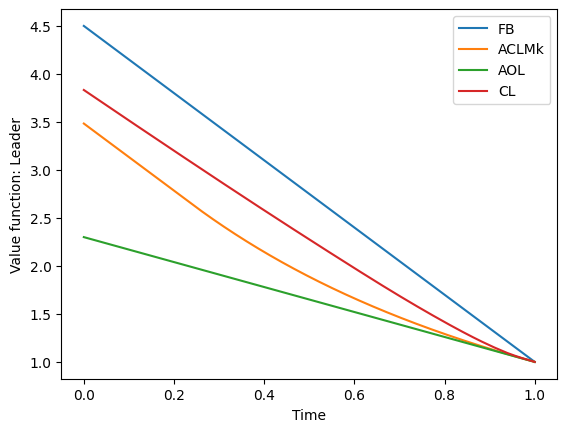}
		\centering
		\caption{\small Leader's value function.}
		\label{fig2:value_leader}
	\end{subfigure}
	\begin{subfigure}{.49\textwidth}
		\includegraphics[scale=0.47]{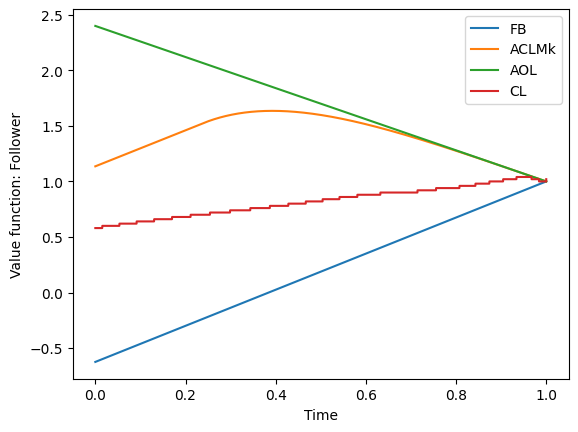}
		\centering
		\caption{\small Follower's value function.}
		\label{fig2:value_follower}
	\end{subfigure}
	\caption{\small Comparison of the value functions, for $c_{\rm F} = 1.25$, $c_{\rm L} = 1$, and $a_\circ = 10$.}
	\label{fig2:numerics}
\end{figure} 

\begin{figure}[!h]
	\centering
	\begin{subfigure}{.49\textwidth}
		\centering
		\includegraphics[scale=0.47]{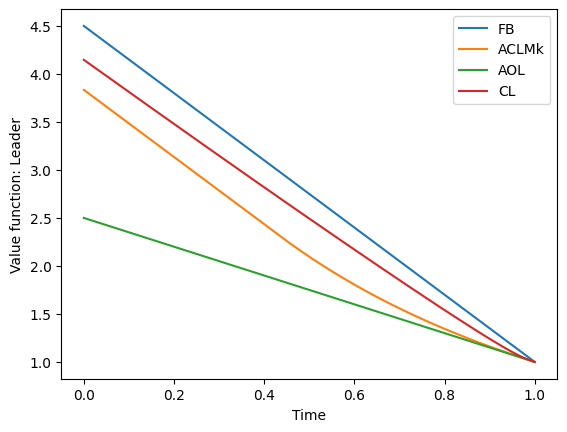}
		\centering
		\caption{\small Leader's value function.}
		\label{fig4:value_leader}
	\end{subfigure}
	\begin{subfigure}{.49\textwidth}
		\includegraphics[scale=0.47]{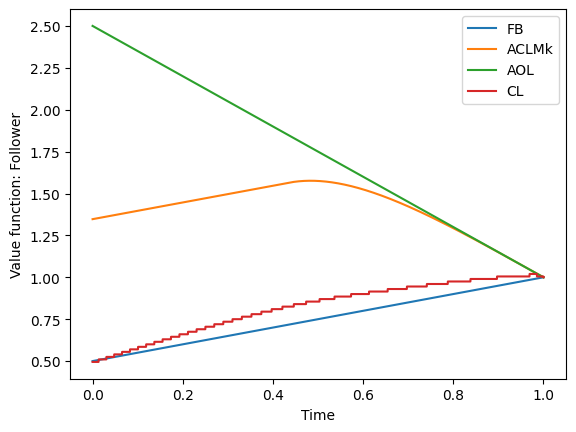}
		\centering
		\caption{\small Follower's value function.}
		\label{fig4:value_follower}
	\end{subfigure}
	\caption{\small Comparison of the value functions, for  $c_{\rm F} = 1$, $c_{\rm L} = 1$, and $a_\circ = 15$.}
	\label{fig4:numerics}
\end{figure}

Finally, comparing \Cref{fig4:numerics} with the benchmark in \Cref{fig1:numerics}, one can notice that increasing the parameter $a_\circ$, representing the maximum absolute value of the leader's effort, will only impact the values in the ACLM-$\bar K$ and CL cases. Indeed, in the AOL and AF cases, the leader will always exert the constant effort $1/c_{\rm L}$, independently of $a_\circ$. Similarly, in the ACL and FB scenarios, the leader will still be able to force the follower to exert the maximal effort $b_\circ$.
However, in the closed-loop equilibrium, when $a_\circ$ increases, the leader has more bargaining power to incentivise the follower to exert a higher effort. More precisely, when studying the partial differential equations satisfied by the boundaries $w^\pm$, one can notice that if $a_\circ$ increases, the cone formed by the boundaries becomes larger. The leader should still ensure that the target constraint is satisfied, and therefore set the control $Z$ to $1$ when one of the barriers is hit, but as the cone is wider this constraint becomes less restrictive. Intuitively, if the set $A$ was not bounded, the boundaries $w^-$ and $w^+$ would be at $-\infty$ and $+\infty$ respectively, leading to an unconstrained problem for the leader. With this in mind, the limit of the leader's value when $a_\circ$ goes to infinity should coincide with her value in the first-best case. In other words, the higher $a_\circ$, the longer the leader can force the follower to exert the maximal effort $b_\circ$ instead of his optimal effort $1/c_{\rm F}$. The same reasoning holds for the ACLM-$\bar K$ scenario: if $a_\circ$ is larger, the parameter $\bar K$ is also larger, and the leader can therefore force the follower to exert the maximal effort $b_\circ$ during a longer period. Similarly, if $a_\circ$ goes to infinity, then $\bar K$ would also go to infinity, meaning that the follower would be forced to apply the maximal effort $b_\circ$ during the entire game.

\section{General problem formulation}\label{sec:general_formulation}

Let $T>0$, $\Omega\coloneqq \Cc([0,T];\R^d)$, topologised by uniform convergence, and $X$ be the canonical process on $\Omega$, that is
\[
X_t(x) \coloneqq x(t), \; x\in\Omega, \; t\in[0,T].
\] 
We denote by $\F=(\Fc_t)_{t\geq 0}$ the filtration generated by $X$, \emph{i.e.} $\Fc_t \coloneqq  \Fc_t^X$, $t\in[0,T]$. The process $X$ represents the output of the game, which will be controlled in weak formulation by both the leader and the follower.

\smallskip

Let $\mathbf{M}(\Omega)$ be the set of all probability measures on $(\Omega,\Fc_T)$. 
$\mathbb P\in \mathbf{M}(\Omega)$ is said to be a semi-martingale measure if $X$ is an $(\F,\P)$--semi-martingale. 
We denote by $\mathcal{P}_S$ the set of all semi-martingale measures. 
By \citeauthor*{karandikar1995pathwise} \cite{karandikar1995pathwise}, there exists an $\mathbb F$--progressively measurable process denoted by $[ X]\coloneqq ([ X]_t)_{t\in [0,T]}$ coinciding with the quadratic variation of $X$, $\mathbb P\as$, for any $\P\in\Pc_S$. 
The density with respect to the Lebesgue measure is denoted by a non-negative symmetric matrix $\widehat{\sigma}^2_t\in\S^{d}$ defined by
\begin{equation}\label{eq:sigmahat}
\widehat{\sigma}^2_t\coloneqq  \underset{\varepsilon \searrow 0}{\rm{lim sup}} \bigg\{ \frac{[ X]_t- [X]_{t-\varepsilon}}{\varepsilon}\bigg\}, \; t\in[0,T].
\end{equation}
The so-called universal filtration $\mathbb{F}^U \coloneqq ( \mathcal{F}^U _t)_{0\leq t \leq T} $ is given by \(\mathcal{F}^U _t \coloneqq  \bigcap_{\mathbb{P}\in\mathbf{M}(\Omega)}\mathcal{F}_t^{\mathbb{P}}\), where $\mathcal{F}_t^{\mathbb{P}}$ is the usual $\P$-completion of $\Fc_t$. For any subset $\mathcal{P}\subseteq\mathbf{M}(\Omega)$, letting $\Nc^{\Pc}$ denote the collection of $\Pc$-polar sets, \emph{i.e.} the sets which are $\mathbb{P}$-negligible for all $\mathbb{P}\in \Pc$, we define the filtration $\mathbb{F}^{\mathcal{P}}\coloneqq ( \mathcal{F}^\mathcal{P}_t )_{t\in [0,T]}$, defined by \(\mathcal{F}^{\mathcal{P}}_t\coloneqq \mathcal{F}_{t}^\Pc \vee \Nc^\mathcal{P},\ t\in[0, T]\).

\subsection{Controlled state dynamics}\label{sec.stateprocessform}

Given finite-dimensional Euclidian spaces $A$ and $B$, we describe the state process by means of the coefficients   
\[
\sigma: [0,T]\times \Omega \times  A \times B \longrightarrow \R^{d\times n},\; \text{and}\; \lambda:[0,T]\times \Omega \times A\times B \longrightarrow\R^n,\]
assumed to be Borel-measurable and non-anticipative, in the sense that, for $\varphi\in\{\sigma,\lambda\}$, $\varphi_t(x,a,b)=\varphi_t(x_{\cdot\wedge t},a,b)$, $(t,x,a,b)\in [0,T]\times\Omega\times A\times B$. Since the product $\sigma\lambda$ will appear often, we abuse notations and write $\sigma\lambda_t(x,a,b)\coloneqq \sigma_t(x,a,b)\lambda_t(x,a,b)$, for all $(t,x,a,b)\in [0,T]\times\Omega\times A\times B$. These functions satisfy the following conditions, which we comment upon in \Cref{rmk.dynamics}.

\begin{assumption}\label{Assumption.data}
	\begin{enumerate}[label=$(\roman*)$,itemsep=5pt]
		\item The map $\Omega \ni x\longmapsto \sigma_t(\cdot,a,b)$ is continuous for every $(t,a,b)\in [0,T]\times A \times B$, and, in addition for every $x \in \Omega$,  $\sigma\sigma^\top_t(x,a,b)\coloneqq\sigma_t(x,a,b)\sigma^\top_t(x,a,b)$ is invertible. Moreover, there is $\ell_\sigma>0$ such that $|\sigma_t(x,a,b)|+|(\sigma\sigma^\top)^{-1}(t,x,a,b)|\leq \ell_\sigma$ for every $(t,x,a,b)\in  [0,T]\times \Omega \times  A \times B$.
		\item There is $\ell_\lambda>0$ such that $|\lambda_t(x,a,b)|\leq \ell_\lambda$, for every $(t,x,a,b)\in  [0,T]\times \Omega \times  A \times B$. 
	\end{enumerate}
\end{assumption}

The actions of the leader are valued in $A$, and the actions of the follower are valued in $B$. 
We define the sets of controls $\Ac$ and $\Bc$ as the ones containing the $\F$-predictable processes with values in $A$ and $B$, respectively. Let $x_0\in\R^d$, then for $(\alpha,\beta)\in\Ac\times\Bc$, the controlled state equation is given by the SDE
\begin{equation}\label{eq:X-dynamics-with-drift}
	X_t = x_0 + \int_0^t \sigma \lambda_s(X_{\cdot\wedge s},\alpha_s,\beta_s) \mathrm{d}s + \int_0^t \sigma_s(X_{\cdot\wedge s},\alpha_s,\beta_s) \mathrm{d}W_s, \; t\in[0,T],
\end{equation}
where $W$ denotes an $n$-dimensional Brownian motion. We characterise \eqref{eq:X-dynamics-with-drift} in terms of weak solutions. These are elegantly represented in terms of so-called martingale problems and Girsanov's theorem, see \citeauthor*{stroock1997multidimensional} \cite{stroock1997multidimensional} for details. Indeed, let us consider the SDE
\begin{align}\label{eq:X-dynamics-without-drift}
	X_t = x_0  + \int_0^t \sigma_s(X_{\cdot\wedge s},\alpha_s,\beta_s) \mathrm{d}W_s, \; t\in[0,T],
\end{align}
and denote by $\Pc$ the set of weak solutions to \eqref{eq:X-dynamics-without-drift}. This is
\begin{align*}
	\Pc\coloneqq \{ \P\in {\bf M}(\Omega): &\ \exists W^\P, \;\bmdim\text{-dimensional } \P\text{--Brownian motion},\\ \text{and}\; &\ (\alpha,\beta)\in \Ac\times\Bc\text{ for which }\eqref{eq:X-dynamics-without-drift}\text{ holds }\P\as\}.
\end{align*}

By Girsanov's theorem, any $\P\in \Pc$ induces $\bar\P \in {\bf M}(\Omega)$ weak solution to \eqref{eq:X-dynamics-with-drift}, where $\bar \P$ is defined by
\begin{align}\label{eq.girsanov}
	\frac{ \mathrm{d}\bar \P}{\mathrm{d}\P}\coloneqq  \exp\bigg( \int_0^T \lambda_s(X_{\cdot\wedge s},\alpha_s,\beta_s)\cdot \mathrm{d}W_s^\P - \frac{1}{2}\int_0^T \|\lambda_s(X_{\cdot\wedge s},\alpha_s,\beta_s)\|^2 \mathrm{d}s  \bigg).
\end{align}
For any action $\alpha\in\Ac$ of the leader, we define the set $\Rc(\alpha)$ of admissible responses of the follower by
\[
\Rc(\alpha)  \coloneqq  \{ (\P,\beta) \in  \Pc \times\Bc:  \P \text{ unique measure in } \Pc,\;  \eqref{eq:X-dynamics-without-drift}\text{ holds }\P\as\text{ with }(\alpha,\beta) \},
\] 
as well as the set of weak solutions 
\begin{align*}
	\Pc^\alpha  \coloneqq  \{ \P\in \Pc:\eqref{eq:X-dynamics-without-drift}\text{ holds }\P\as\text{ with }(\alpha,\beta), \text{ for some }\beta\in \Rc(\alpha)\}.
\end{align*}
\begin{remark}\label{rmk.dynamics}
	\begin{enumerate}[label=$(\roman*)$, ref=.$(\roman*)$,wide,  labelindent=0pt]
		\item We note that $\Pc$ is nonempty due to the continuity assumption on $\sigma$, ensuring that solutions do exist for instance for constant controls $\alpha$ and $\beta$, see {\rm \cite[Theorem 6.1.6]{stroock1997multidimensional}}. 
		{Concerning the uniqueness of weak solutions, we impose it as a condition for the admissible controls of the follower. 
			That is, for a pair $(\alpha,\beta)$ of controls played by the leader and the follower, the law of $X$ is uniquely determined.}
		
		\item {We also stress that in the above formulation, there is no need to enlarge the canonical space.
			This subtlety is significant in the context of Stackelberg games, as doing so would mean changing the information structure of the game.
			Indeed, we note that in the definition of $\Pc$, $W^\P$ is a Brownian motion in the original canonical space $\Omega$. 
			Given our assumptions on the volatility $\sigma\sigma^\top$, namely its invertibility and boundedness, we do not need to enlarge $\Omega$ in this setting. 
			In general, if the volatility is allowed to degenerate, one may need to introduce external sources of randomness and define a Brownian motion on an enlarged probability space. 
			We refer the reader to {\rm \cite[Section 4.5]{stroock1997multidimensional}} and {\rm \cite[Section 2.1.2]{possamai2018stochastic}} for a discussion on these results. 
		} 
	\end{enumerate}
\end{remark}

\subsection{The closed-loop Stackelberg game between the leader and the follower}
\textcolor{black}{The rewards of the players are specified through the mappings
	\[
	c : [0,T]\times \Omega \times  A \times B \longrightarrow \R,\; g:\Omega \longrightarrow\R,\;  C:[0,T]\times \Omega \times A\times B \longrightarrow\R,\;\text{and}\; G:\Omega \longrightarrow\R,
	\]
	assumed to be Borel-measurable and non-anticipative. They satisfy the next assumption which we comment on in \Cref{rem:nonuniq}.}
\begin{assumption}\label{Assumption.data.2}
	{There is $\ell_\textup{r}>0$ such that $|c_t(x,a,b)|+|C_t(x,a,b)|+|g(x)|+|G(x)|\leq \ell_\textup{r}$ for all $(t,x,a,b)\in  [0,T]\times \Omega \times  A \times B$.}
\end{assumption}

The timing of the game is as follows. The leader chooses first a control $\alpha\in\Ac$ to which the follower responds with $\beta\in\Bc$. 
The response is, of course, dependent on the control chosen by the leader. 
Given an action $\alpha\in\Ac$, the problem of the follower is given by 
\begin{equation}\label{eq:problem-follower}
	V_{\rm F}(\alpha) \coloneqq  \sup_{(\P,\beta)\in\Rc(\alpha)} \bigg\{ \E^{\bar\P} \bigg[ \int_0^T c_s(X_{\cdot\wedge s},\alpha_s,\beta_s) \mathrm{d}s + g(X_{\cdot\wedge T}) \bigg]\bigg\}.
\end{equation}

We say that $(\P,\beta)\in\Rc(\alpha)$ is an \emph{optimal response} to $\alpha\in\Ac$, and write $(\P,\beta)\in\Rc^\star(\alpha)$, if $(\P,\beta)$ is a solution to Problem \eqref{eq:problem-follower}. We will assume that there exists $\alpha^o\in\Ac$ such that $\Rc^\star(\alpha^o)\neq \emptyset$. Then, the leader chooses a control from the set $\Ac$ and anticipates the optimal response of the follower. Therefore, the leader faces the following problem
\begin{equation}\label{eq:problem-leader}
	V_{\rm L} \coloneqq  \sup_{\alpha\in\Ac} \sup_{(\P,\beta)\in\Rc^{\smallfont{\star}}(\alpha)}  \bigg\{ \E^{{\bar \P}} \bigg[   \int_0^T C_s(X_{\cdot\wedge s},\alpha_s,\beta_s) \d s + G(X_{\cdot \wedge T}) \bigg]\bigg\}, 
\end{equation}

\begin{remark}\label{rem:nonuniq}
	\begin{enumerate}[label=$(\roman*)$, ref=.$(\roman*)$,wide,  labelindent=0pt]
		\item We assume that the functions in our model are bounded just to simplify the expositions of the results. 
		These assumptions can be weakened by imposing the usual integrability conditions in the set of admissible controls of the players. 
		The results in this section and in {\rm \Cref{sec:stochastic_target}} will still hold. 
		The analysis becomes more delicate when studying the so-called target reachability set, defined in {\rm \Cref{sec:solve_leader}}, through its upper and lower boundaries, and to characterise them by our methods.
		\item Notice that under our convention that the supremum over an empty set is equal to $-\infty$, the leader will never choose $\alpha$ such that $\Rc^\star(\alpha)=\emptyset$. Thus, the assumption on the existence of $\alpha^o\in\Ac$ such that $\Rc^\star(\alpha^o)\neq \emptyset$ guarantees that the problem of the leader is not degenerate. This is a relatively harmless assumption, in the sense that the whole problem is ill-posed otherwise. Nonetheless, it is straightforward to check, for instance when the follower only controls the drift of $X$, that under standard integrability assumptions, any constant choice for $\alpha^o$ satisfies our requirements.
		\item \label{rem:nonuniq.3} Let us mention that the existence of optimal responses is fundamental for Stackelberg games and cannot be dropped. 
		Indeed, the main motivation in this game is that the leader plays first by anticipating the response of the follower.
		On the other hand, {we assume that the leader has enough bargaining power to make the follower choose a maximiser that suits her best, or equivalently, we consider the problem of an \emph{optimistic} leader for whom}, if the follower has multiple optimal responses---and thus he is indifferent among all of them---he will choose one that benefits the leader the most.
		This is consistent with, for instance, {\rm\cite[Section 2.1]{bressan2011noncooperative}}, {\rm\cite{zemkoho2016solving}}, or {\rm \cite{havrylenko2022risk}}. 	Alternatively, one could take an adversarial perspective in which the leader faces the problem
		\begin{equation*}
			V_{\rm L}^{\text{\scalefont{0.8}\textup{Pes}}} \coloneqq  \sup_{\alpha\in\Ac}  \inf_{(\P,\beta)\in\Rc^\smallfont{\star}(\alpha)}  \bigg\{\E^{{\bar \P}} \bigg[   \int_0^T C_s(X_{\cdot\wedge s},\alpha_s,\beta_s) \d s + G(X_{\cdot \wedge T}) \bigg]\bigg\}.
		\end{equation*}	
		
		This is the \emph{pessimistic} point of view, which has also been coined \emph{generalised} or \emph{weak} Stackelberg equilibrium, see {\rm \cite{leitmann1978generalized}}, {\rm \cite{basar1999dynamic}}, {\rm\cite{wiesemann2013pessimistic}}, or {\rm\cite{liu2018pessimistic}}. Notice that in this case, existence of equilibria may become problematic, which led part of the literature to consider so-called \emph{regularised} Stackelberg problems, where, for a fixed $\eps>0$, the infimum would now be taken over the set of actions of the follower which give him a value $\eps$-close to his optimal one, see {\rm\cite[Section 3]{mallozzi1995weak}} and the references therein. {We point out that our approach allows to tackle both the \emph{optimistic} and the \emph{pessimistic} problems in the same way, the difference being in the resulting Hamiltonians of the {\rm HJB} equations associated to each one of the two problems. More details will be given below.} 
	\end{enumerate}
\end{remark}

\section{Reduction to a target control problem}\label{sec:stochastic_target}

In this section, we characterise the solutions $(\P^\star,\beta^\star)\in \Rc^\star(\alpha)$ to the continuous-time stochastic control problem \eqref{eq:problem-follower} for fixed leader's control $\alpha\in\Ac$. 
Our approach is inspired by the dynamic programming approach to principal--agent problems developed in \cite{cvitanic2018dynamic}. 

\smallskip

As standard in the control literature, we introduce the following Hamiltonian functions $H^{\rm F}:[0,T]\times\Omega\times\R^d\times\S^d \times A \longrightarrow\R$ and $h^{\rm F}:[0,T]\times\Omega\times\R^d\times\S^d \times A\times B \longrightarrow\R$
\begin{align}\label{eq.hamiltonian}
	H^{\rm F}_t(x,z,\gamma,a) &\coloneqq  \sup_{b\in B}\big\{ h _t^{\rm F}(x,z,\gamma,a,b)\big\},\\ \text{and} \;
	h^{\rm F}_t(x,z,\gamma,a,b) &\coloneqq  c_t(x,a,b) + \sigma\lambda_t(x,a,b)\cdot z + \frac12 \mathrm{Tr}[\sigma\sigma_t^\t(x,a,b)\gamma]. \nonumber
\end{align}
Define now the set $A_t(x,\Sigma,a)\coloneqq \big\{ b\in B:  \sigma\sigma_t^\t(x,a,b) = \Sigma \big\}$ for $(t,x,\Sigma,a)\in [0,T]\times\Omega\times \S^d_+\times A $. For $(\alpha,\P)\in \Ac\times \Pc^\alpha$, the set of controls for the follower is
\[
\Bc(\alpha,\P)\coloneqq  \{ \beta\in\Bc: \beta_t\in A_t(x,\sigmah_t^2,\alpha_t), \; \mathrm{d}t\otimes\mathbb{P}\ae \},
\]
where $\sigmah_t^2$ has been defined in \Cref{eq:sigmahat}.
With these definitions, we can isolate the partial maximisation with respect to the squared diffusion in $H^{\rm F}$. More precisely, letting $F:[0,T]\times\Omega\times\R^d\times \S^d_+  \times A\longrightarrow\R$, be given by
\[
F_t(x,z,\Sigma,a) \coloneqq  \sup_{b\in A_\smallfont{t}(x,\Sigma,a)} \big\{ c_t(x,a,b)+ \sigma \lambda_t(x,a,b) \cdot z\},
\]
we have that $2H^{\rm F}=(-2F)^\ast$ where the superscript $\ast$ denotes the Legendre transform
\[
H^{\rm F}_t(x,z,\gamma,a) = \sup_{\Sigma\in\S_\smallfont{d}^\smallfont{+}} \bigg\{  F_t(x,z,\Sigma,a) +\frac12 \text{Tr}[\Sigma\gamma]\bigg\}.
\]

Recalling \eqref{eq.girsanov}, we can equivalently write the problem of the follower \eqref{eq:problem-follower} as
\begin{align}\label{eq.prob.follower.separated}
	V_{\rm F}(\alpha) = \sup_{\P\in\Pc^\smallfont{\alpha}} \sup_{\beta \in\Bc(\alpha,\P)}  \bigg\{\E^{\bar\P} \bigg[ \int_0^T c_s(X_{\cdot\wedge s},\alpha_s,\beta_s) \mathrm{d}s + g(X_{\cdot\wedge T}) \bigg]\bigg\},
\end{align}
to which we will associate the second-order BSDE\footnote{We refer the reader to \cite{possamai2018stochastic,soner2012wellposedness} for an introduction and extension of the theory of such equations.}
\begin{equation}\label{eq:2bsde-follower}
	Y_t = g(X_{\cdot\wedge T}) + \int_t^T F_s(X_{\cdot\wedge s},Z_s, \sigmah_s^2,\alpha_s) \mathrm{d}s - \int_t^T Z_s \cdot \mathrm{d}X_s +\int_t^T\mathrm{d}K_s,\;  \Pc^\alpha\qs,
\end{equation}
for $ t\in[0,T]$. {Notice that, similarly to \cite{cvitanic2018dynamic}, we consider an aggregated version of the non-decreasing process $K$}.\footnote{We require the aggregation of the component $K$, as well as the one of the stochastic integral, to define later the forward process $Y^{y,Z,\Gamma,\alpha}$ independently of any probability. There are aggregation results for the stochastic integral in \cite{nutz2012pathwise}, which suit our setting and use the notion of medial limits. By following this route, one would need to assume ZFC plus some other axioms. We refer the reader to \cite[Footnote 7]{possamai2018stochastic} for a further discussion on the weakest set of axioms known to be sufficient for the existence of medial limits.}  We have then the following notion of solution to the 2BSDE, the functional spaces mentioned in the following definition can be found in \Cref{sec:funcspace}. We also use the notation
\[
\Pc^\alpha[\mathbb P,\mathbb F^+,t] \coloneqq  \{  \P^\prime \in \Pc^\alpha: \P[E] = \P^\prime[E],\; \forall E \in\Fc_t^+  \}.
\]

\begin{definition}\label{def:sol2BSDE} We say that the triple {$(Y,Z,K)$} is a solution to the {\rm 2BSDE} \eqref{eq:2bsde-follower} if there exists $p>1$ such that {$(Y,Z,K)\in\mathbb S^p(\mathbb F^{\Pc^\smallfont{\alpha}},\Pc^\alpha) \times  \mathbb H^p(\mathbb F^{\Pc^\smallfont{\alpha}},\Pc^\alpha) \times  \mathbb I^p(\mathbb F^{\Pc^\smallfont{\alpha}},\Pc^\alpha) $}  satisfies \eqref{eq:2bsde-follower} and $K$ satisfies the minimality condition
	{\begin{equation}\label{eq:minimality condition}
			K_t = \essinf_{\mathbb P'\in \Pc^\smallfont{\alpha}[\mathbb P,\mathbb F^\smallfont{+},t]}\big\{\E^{\mathbb P^\smallfont{\prime}} \big[K_T \big| \mathcal F_t^{\mathbb P,+} \big]\big\},\; t\in [0,T],\; \Pc^\alpha\qs
	\end{equation} }
\end{definition}

The next result connects the problem of the follower with the 2BSDE \eqref{eq:2bsde-follower}. 

\begin{proposition}\label{prop:main-follower}
	There exists a unique solution {$(Y,Z,K)$} to the {\rm 2BSDE} \eqref{eq:2bsde-follower}, for which the value of the follower satisfies $V_{\rm F}(\alpha)=\sup_{\P\in\Pc^\smallfont{\alpha}} \{\E^{\bar \P}[Y_0]\}$. Moreover, $(\P^\star,\beta^\star)\in \Rc^\star(\alpha)$ if and only if {$K_T=0$}, $\P^\star\as$ and
	\begin{equation}\label{eq:optimal-effort-follower}
		\beta^\star \;\text{\rm is a maximiser in the definition of } F_\cdot(X_\cdot,Z_\cdot,\sigmah^2_\cdot,\alpha_\cdot), \; \mathrm{d}t \otimes \mathrm{d}\P^\star\ae
	\end{equation}
\end{proposition}

\begin{proof}
	Notice that the follower's problem can be seen as the particular problem of an agent who is offered by the principal a terminal remuneration of the form $\xi=g( X_{\cdot \wedge T})$. Since the function $g$ is assumed to be bounded, the result is a direct application of \cite[Propositions 4.5 and 4.6]{cvitanic2018dynamic}. 
\end{proof}

For $p>1$, $(y,\alpha,Z,K)\in \Ac\times\R \times  \mathbb H^p(\mathbb F^{\Pc^\smallfont{\alpha}}, \Pc^\alpha) \times  \mathbb I^p (\mathbb F^{\Pc^\smallfont{\alpha}}, \Pc^\alpha)$, $K$ satisfying \eqref{eq:minimality condition}, the process $Y^{y,\alpha,Z,K}$, given by
\begin{equation*}\label{eq:Y-ZK}
	Y^{y,\alpha,Z,K}_t\coloneqq  y-\int_0^t F_s(X_{\cdot\wedge s},Z_s,\sigmah_s^2,\alpha_s) \mathrm{d}s+\int_0^tZ_s\cdot \mathrm{d}X_s-\int_0^t \mathrm{d}K_s, \; t\in [0,T],
\end{equation*}
is well-defined, independent of the probability $\mathbb{P}$, because the stochastic integrals can be defined pathwise (see \cite[Definition 3.2]{cvitanic2018dynamic} and the paragraph thereafter). 
The idea is to look at the tuples $(y,\alpha,Z,K)$ for which it holds that $Y^{y,\alpha,Z,K}_T= g(X_{\cdot\wedge T})$. 
Furthermore, as argued in \cite[Theorem 3.6]{cvitanic2018dynamic}, the processes $K$ can be approximated by those of the form 
\begin{equation*}
	\int_0^t \bigg(    H^{\rm F}_s(X_{\cdot\wedge s},Z_s,\Gamma_s,\alpha_s) -F_s(X_{\cdot\wedge s},Z_s,\sigmah^2_s,\alpha_s) -\frac12\text{Tr} [ \sigmah^2_s \Gamma_s\big] \bigg) \mathrm{d}s,
\end{equation*}
for some appropriate control $\Gamma$. However, as recently highlighted by \cite{chiusolo2024new}, the above representation may not always be possible if the following assumption is not satisfied.
\begin{assumption}\label{ass:duality}
	For all $(t,x,z,\Sigma,a) \in [0,T] \times \Omega \times \R^d \times \S^d_+ \times A$, we have
	\begin{align*}
		F_t(x,z,\Sigma,a) = \min_{\gamma \in \S^d} \bigg\{H^{\rm F}_t(x,z,\gamma,a) - \dfrac12 {\rm Tr}  [\Sigma \gamma] \bigg\}.
	\end{align*}
\end{assumption}
Under the previous assumption, \cite[Theorem 3.6]{cvitanic2018dynamic} is verified and it is then appropriate to define the following class of processes that would be seen as controls from the point of view of the leader.
\begin{definition}\label{def.parametrization}
	For any $\alpha\in\Ac$, let ${\Cc^\alpha}$ be the class of $\F^{\Pc^\smallfont{\alpha}}$-predictable processes $(Z,\Gamma):[0,T]\times\Omega\longrightarrow \R^d\times\S^\xdim$ such that 
	\[ 
	\|Y^{y,\alpha,Z,\Gamma}\|_{\S^\smallfont{p}(\mathbb F^{\smallfont{\Pc}^\tinyfont{\alpha}}, \Pc^\smallfont{\alpha})}^p+\|Z\|_{\H^\smallfont{p}(\mathbb F^{\smallfont{\Pc}^\tinyfont{\alpha}}, \Pc^\smallfont{\alpha})}^p <+\infty ,\]
	for some $p>1$, where for $y\in \R$ we define, $\P\as$ for all $\P\in\Pc^\alpha$, the process
	\begin{equation}\label{eq:Y-ZGamma}
		Y^{y,\alpha,Z,\Gamma}_t\coloneqq  y-\int_0^t H^{\rm F}_s(X_{\cdot\wedge s},Z_s,\Gamma_s,\alpha_s) \mathrm{d}s+\int_0^tZ_s\cdot \mathrm{d}X_s+ \frac12\int_0^t \mathrm{Tr} [ \sigmah^2_s \Gamma_s]  \mathrm{d}s,\; t\in [0,T].
	\end{equation}
\end{definition}

The next proposition provides an optimality condition for a pair $(\P,\beta)$, when the process $Y^{y,\alpha,Z,\Gamma}$ \emph{hits} the correct terminal condition, \emph{i.e.} $Y_T^{y,\alpha,Z,\Gamma}=g(X_{\cdot\wedge T})$, $\P\as$ In such a case, the follower's value coincides with $y$, and his optimal actions correspond to maximisers of the Hamiltonian $H^{\rm F}$. We will use this characterisation in the next section to obtain a reformulation of the problem of the leader.

\begin{proposition}\label{prop-response-follower-niceform}
	Let $\alpha\in\Ac$ and {\color{black}$(y,Z,\Gamma)\in\R\times{\Cc^\alpha}$} be such that $Y_T^{y,\alpha,Z,\Gamma}=g(X_{\cdot\wedge T})$, $\P\as$, for some $(\P,\beta)\in\Rc(\alpha)$. Then, the following are equivalent
	\begin{enumerate}[label=$(\roman*)$, ref=.$(\roman*)$,wide,  labelindent=0pt]
		\item $(\P ,\beta )\in\Rc^\star(\alpha)$ and $V_{\rm F}(\alpha)=y;$
		
		\item $\beta$ maximises $h^{\rm F}$ on the support of $\P $
		\begin{equation}\label{eq:optimal-effort-reformulated}
			H^{\rm F}_t(X_{\cdot\wedge t},Z_t,\Gamma_t,\alpha_t)= h^{\rm F}_t(X_{\cdot\wedge t},Z_t,\Gamma_t,\alpha_t,\beta_t), \; \mathrm{d}t \otimes \mathrm{d}\P\ae
		\end{equation}
	\end{enumerate}
\end{proposition}

\begin{proof}[Proof of {\Cref{prop-response-follower-niceform}}]
	Let $(\P,\beta)\in\Rc(\alpha)$ such that $Y_T^{y,\alpha,Z,\Gamma}=g(X_{\cdot\wedge T})$, $\P\as$ Assume $(i)$ holds. Then, the value and utility of the follower satisfy
	\begin{align*}
		V_{\rm F}(\alpha) &=U_{\rm F}(\P ,\beta )
		= \E^{\bar \P} \bigg[ \int_0^T c_s(X_{\cdot\wedge s},\alpha_s,\beta_s) \mathrm{d}s + Y_T^{y,\alpha,Z,\Gamma}  \bigg] .
	\end{align*}
	By writing the dynamics of $Y^{y,\alpha,Z,\Gamma} $ and the fact that $\P$ is a weak solution to \eqref{eq:X-dynamics-with-drift} with $(\alpha,\beta)$, we obtain
	\begin{align*}
		U_{\rm F}(\P,\beta) & = \E^{\bar \P} \bigg[ \int_0^T c_s(X_{\cdot\wedge s},\alpha_s,\beta_s) \mathrm{d}s + y-\int_0^T H^{\rm F}_s(X_{\cdot\wedge s},Z_s,\Gamma_s,\alpha_s) \mathrm{d}s +\int_0^T Z_s\cdot \mathrm{d}X_s\\
		&\qquad \quad+ \frac12\int_0^T \mathrm{Tr} [ \widehat\sigma^2_s \Gamma_s]  \mathrm{d}s \bigg] \\
		& = y + \E^{\bar \P} \bigg[  \int_0^T \big( h^{\rm F}_s(X_{\cdot\wedge s},Z_s,\Gamma_s,\alpha_s,\beta_s) - H^{\rm F}_s(X_{\cdot\wedge s},Z_s,\Gamma_s,\alpha_s) \big) \mathrm{d}s \\
		&\qquad \quad +\int_0^T Z_s\cdot \sigma_s(X,\alpha_s,\beta_s^\star) \mathrm{d}W^{\bar \P^\star}_s \bigg]  \\
		& = y + \E^{\bar \P} \bigg[  \int_0^T \big( h^{\rm F}_s(X_{\cdot\wedge s},Z_s,\Gamma_s,\alpha_s,\beta_s) - H^{\rm F}_s(X_{\cdot\wedge s},Z_s,\Gamma_s,\alpha_s) \big) \mathrm{d}s \bigg],
	\end{align*}
	since the stochastic integral is a martingale due to the integrability conditions specified in the definition of ${\Cc^\alpha}$. Now, by definition of $H^{\rm F}$, see \eqref{eq.hamiltonian}, we see that $V_{\rm F}(\alpha)\leq y$. 
	Since $V_{\rm F}(\alpha)= y$, we deduce $(ii)$ holds. Let us now assume $(ii)$. 
	Since $(\P ,\beta)\in \Rc(\alpha)$, it follows from \eqref{eq.prob.follower.separated} that
	\begin{align}\label{eq.supfixedp}
		V_{\rm F}(\alpha) \geq \sup_{\beta \in\Bc(\alpha,\P )} \bigg\{\E^{\bar\P} \bigg[ \int_0^T c_s(X_{\cdot\wedge s},\alpha_s,\beta_s) \mathrm{d}s + g(X_{\cdot\wedge T}) \bigg]\bigg\}.
	\end{align}
	The value on the right \textcolor{black}{corresponds to $\E^{\bar \P}[Y_0]$, where $(Y,Z,K)$ is the unique solution to the BSDE} 
	\[
	Y_t = g(X_{\cdot\wedge T}) + \int_t^T F_s(X_{\cdot\wedge s},Z_s, \sigmah_s^2,\alpha_s) \mathrm{d}s - \int_t^T Z_s \cdot \mathrm{d}X_s+\int_t^T \d K_s , \; t\in[0,T],\; \P  \as,
	\]
	and equality in \eqref{eq.supfixedp} holds if $K_T=0$, $\P\as$ Since $2H^{\rm F}=(-2F)^\ast$ and \eqref{eq:optimal-effort-reformulated} hold, together with the condition $Y_T^{y,\alpha,Z,\Gamma}=g(X_{\cdot\wedge T})$, $\P\as$, we see that $Y^{y,\alpha,Z,\Gamma}$ satisfies
	\[
	Y_t^{y,\alpha,Z,\Gamma} = g(X_{\cdot\wedge T}) + \int_t^T F_s(X_{\cdot\wedge s},Z_s, \sigmah_s^2,\alpha_s) \mathrm{d}s - \int_t^T Z_s \cdot \mathrm{d}X_s +\int_t^T \d K_s^{Z,\Gamma,\alpha}, \; \P \as, 
	\]
	for $t\in[0,T]$, where
	\[
	K_t^{Z,\Gamma,\alpha} \coloneqq \int_0^t \bigg(    H^{\rm F}_s(X_{\cdot\wedge s},Z_s,\Gamma_s,\alpha_s) -h^{\rm F}_s(X_{\cdot\wedge s},Z_s,\Gamma_s,\alpha_s,\beta_s)  \bigg) \mathrm{d}s,
	\]
	which by assumption satisfies $K_T^{Z,\Gamma,\alpha}=0,$ $\P\as$
	Hence $(\P,\beta)\in \Rc^\star(\alpha)$ from the previous discussion. 
	Finally, since by uniqueness of the solution we have that $y=\E^{\bar \P}[Y_0]$, the fact that $V_{\rm F}(\alpha)= y$ is argued as in $(i)$.
\end{proof}

\subsection{A stochastic target reformulation of the problem of the leader}

In light of the results from the previous section, we are drawn to reformulate the problem faced by the leader as a stochastic control problem with stochastic target constraints. 
Indeed, \Cref{prop-response-follower-niceform} tells us that the value of the follower (given the control $\alpha$ by the leader) is equal to $V_{\rm F}(\alpha)=y$, and any pair $(\P^\star,\beta^\star)$ that satisfies \eqref{eq:optimal-effort-reformulated} is a solution to the problem of the follower, as long as $Y^{y,\alpha,Z,\Gamma}$ hits the correct terminal value. 

\smallskip
For $(Z,\Gamma,\alpha)\in {\Cc^\alpha}\times\Ac $ and \textcolor{black}{deterministic $y\in\R$}, which represents the value of the follower, 
let us define the set 
\begin{align*}
	\Rc^\star(y,\alpha,Z,\Gamma)&\coloneqq \{ (\P,\beta) \in  \Rc(\alpha) : Y_T^{y,\alpha,Z,\Gamma} = g(X_{\cdot \wedge T}),\; \text{and }\eqref{eq:optimal-effort-reformulated} \text{ hold, } \P\as\}.
\end{align*}

We propose then the following reformulation of the problem of the leader 
\begin{equation}\label{eq:leader-problem-reformulated}
	\hat V_{\rm L} \coloneqq  \sup_{y \in\R} \sup_{(\alpha,Z,\Gamma)\in {\Ac\times \Cc^\smallfont{\alpha}} } \sup_{(\P,\beta)\in\Rc^\smallfont{\star}(y,\alpha,Z,\Gamma)}
	\bigg\{\E^{\bar\P} \bigg[   \int_0^T C_s(X_{\cdot\wedge s},\alpha_s,\beta_s) \d s + G(X_{\cdot \wedge T}) \bigg]\bigg\}.
\end{equation}
\begin{remark}
	Let us briefly digress on the nature of \eqref{eq:leader-problem-reformulated}.
	\begin{enumerate}[label=$(\roman*)$, ref=.$(\roman*)$,wide,  labelindent=0pt]
		\item	A distinctive feature of \eqref{eq:leader-problem-reformulated} is that, as described in {\rm \Cref{sec.stateprocessform}}, the dynamics of the controlled process $X$ are given in weak formulation whereas those of $Y$ are given in strong formulation as in {\rm \eqref{eq:Y-ZGamma}}.
		Though the reader might find this atypical, we recall that this feature is common in the dynamic programming approach in contract theory.
		Since up until this point, our approach has borrowed ideas from this literature, it is not surprising to find this feature in \eqref{eq:leader-problem-reformulated}.
		\item {\color{black} Let us also digress on our choice to reformulate \eqref{eq:problem-leader} as an optimal control problem with target constraints.
			This is certainly not the only possible reformulation available.
			Alternatively, thanks to {\rm \Cref{prop:main-follower}}, \eqref{eq:problem-leader} also admits a reformulation as an optimal control problem of {\rm FBSDEs}.
			Yet we think that there are some shortcomings in following this route.
			Though there exists some literature on this class of control problems, because there is no general comparison principle for {\rm FBSDEs}, results tend to leverage the stochastic maximum principle to derive both necessary and sufficient conditions for optimality.
			Consequently, most of these works consider continuously differentiable state-dependent data in order to derive necessary conditions.
			Additional concavity/convexity assumptions are needed to derive sufficient conditions in terms of a system of {\rm FBSDEs} with twice as many variables as in the initial system, see for instance {\rm \cite[Chapter 10]{cvitanic2012contract}}.
			Be it as it may, we believe that the sufficient condition obtained through our approach (see {\rm \Cref{thm.verification}} below),  which leads naturally to a PDE-based numerical scheme for solving \eqref{eq.fullHJB},}  is more amenable to the analysis and numerical implementations than those in the literature on the control of {\rm FBSDEs}.
	\end{enumerate}
\end{remark}

Recall that $\Rc^\star(y,\alpha,Z,\Gamma) $ is non-empty thanks to \Cref{prop:main-follower} and the discussion thereafter. {Since we agreed that the supremum over an empty set is $-\infty$}, the supremum in the $y$-variable could be taken instead over the set 
\[
\mathfrak{T}\coloneqq \{y\in\R: \Rc^\star(y,\alpha,Z,\Gamma)   \neq\emptyset \text{ for some }(Z,\Gamma,\alpha)\in \Cc^\alpha \times \Ac  \},
\] 
which corresponds to the so-called target reachability set in the language of stochastic target problems as studied for instance in \cite{soner2003stochastic}.
By \Cref{eq:problem-leader}, the reward of the leader is only computed under optimal responses $(\P ,\beta )\in \Rc^\star(\alpha)$, and $\Rc^\star(y,\alpha,Z,\Gamma)$ provides the optimal responses of the follower. \smallskip

The interpretation of $\hat V_{\rm L}$ is as follows. The leader decides $y\in\R$ and optimal controls $(Z^\star,\Gamma^\star,\alpha^\star)\in{\Cc^{\alpha^\smallfont{\star}}}\times \Ac$. She then announces her control $\alpha^\star\in\Ac$ for which she knows that the value of the follower is $y$, \emph{i.e.} $V_{\rm F}(\alpha^\star)=y$, and that his optimal controls belong to $\Rc^\star(y, Z^\star,\Gamma^\star,\alpha^\star)$. The leader can make a recommendation to the follower for his optimal response and the corresponding value, which the latter will follow since he has no better alternative. 
This holds true for every {$y\in \mathfrak{T}$} and the optimal choice of this value is the one that maximises the objective function of the leader. This new problem is a reformulation of the problem of the leader as the following result shows. 

\begin{theorem}\label{prop.equivformulationleader}
	The reformulated and the original problem of the leader have the same value, that is, $\hat V_{\rm L}= V_{\rm L}$.
\end{theorem}

\begin{proof} 
	$(i)$ Let $y\in\R$ and assume that $y\in \mathfrak{T}$ since the supremum in the $y$-variable in \eqref{eq:leader-problem-reformulated} can be reduced to this set. 
	Take next $( \alpha,Z,\Gamma)\in\Ac\times{\Cc^\alpha}$, $(\P,\beta)\in\Rc^\star( y,\alpha,Z,\Gamma)$, and let $Y^{y,\alpha,Z,\Gamma}$ be the process given by \eqref{eq:Y-ZGamma}. By \Cref{prop-response-follower-niceform}, $y=Y_0^{y,\alpha,Z,\Gamma}=V_{\rm F}(\alpha)$ and $(\P,\beta)\in\Rc^\star(\alpha)$. This means that the optimal response of the follower to the action $\alpha$ is given by $(\P,\beta)$. Therefore, the objective function in problem $\hat{V}_{\rm L}$ at $(y,\alpha,Z,\Gamma,\P,\beta)$ is matched by the objective function in $V_{\rm L}$ at $(\alpha,\P,\beta)$. This implies $\hat V_{\rm L} \leq  V_{\rm L}$.\smallskip
	
	$(ii)$ We show that the leader's objective function in $V_{\rm L}$ can be approximated by elements in ${\Cc^\alpha}$. Let $\alpha\in\Ac$ and $(\P^\star,\beta^\star)\in \Rc^\star(\alpha)$. By \Cref{prop:main-follower}, there is $(Y,Z,K)$ solution to the 2BSDE \eqref{eq:2bsde-follower}. We argue in 2 steps.\smallskip
	
	\emph{Step $1$.} We construct an approximate solution to \eqref{eq:2bsde-follower}. Let $\eps>0$, $y\coloneqq \E^{\P^\smallfont{\star}}[Y_0]$ and define
	\[
	K_t^\eps\coloneqq \frac{1}\eps \int_{(t-\eps)^\smallfont{+}}^tK_s \d s,\; Y^{\eps}_t\coloneqq  y-\int_0^t F_s(X_{\cdot\wedge s},Z_s,\sigmah_s^2,\alpha_s) \mathrm{d}s+\int_0^tZ_s\cdot \mathrm{d}X_s+\int_0^t \mathrm{d}K_s^\eps.
	\]
	Note that $K^\eps$ is absolutely continuous, $\F^{\Pc^\smallfont{\alpha}}$-predictable, non-decreasing $\Pc^\alpha\qs$, and $K_T^\eps=0$, $\P^\star\as$ Since $K_T^\eps \leq K_T$, $K^\eps\in \I^p(\F^{\Pc^\smallfont{\alpha}},\Pc^\alpha)$ satisfies \eqref{eq:minimality condition} and $Y_T^\eps$ satisfies the required integrability. 
	{\color{black} That is, $(Y^\eps,Z,K^\eps)$ satisfies \eqref{eq:2bsde-follower} with terminal condition $Y_T^\eps$.
		By standard \emph{a priori} estimates, see \cite[Theorem 4.4]{possamai2018stochastic}, we have that $\|Y^{\eps}\|_{\S^\smallfont{p}(\mathbb F^{\smallfont{\Pc}^\tinyfont{\alpha}}, \Pc^\smallfont{\alpha})}<\infty$.} 
	All in all, we deduce that $(Y^\eps,Z,K^\eps)$ is a solution to 2BSDE \eqref{eq:2bsde-follower} with terminal condition $Y_T^\eps$.\smallskip
	
	\emph{Step $2$.} We show the approximation can be given in terms of elements in ${\Cc^\alpha}$. Let $\dot K^\eps$ be the density, with respect to Lebesgue measure, of $K^\eps$. We claim that there is an $\F$-predictable process $\Gamma^\eps$ such that 
	\begin{align}\label{eq:Keps}
		\dot K_t^\eps = H^{\rm F}_t(X_{\cdot\wedge t},Z_t,\Gamma_t^\eps,\alpha_t) -F_t(X_{\cdot\wedge t},Z_t,\sigmah^2_t,\alpha_t) -\frac12\text{Tr} [\sigmah^2_t \Gamma^\eps_t].
	\end{align}
	Indeed, by definition of $H^{\rm F}$, the map $\gamma \in \S^{d} \longmapsto  H^{\rm F}_t(x,z,\gamma,a) -F_t(x,z,\Sigma,a) - \frac12\text{Tr} [\Sigma \gamma]$ is positive and continuous, for all $(t,x,z,\Sigma,a) \in [0,T] \times \Omega \times \R^d \times \S^d_+ \times A$. Then, by the boundedness of $c$, $\lambda$, and $\sigma$, the supremum over $\gamma$ of the previous function is $+\infty$. Finally, by \Cref{ass:duality}, its infimum over $\gamma$ is $0$ and is attained. Therefore, the previous map is surjective onto $[0, \infty)$, and since the density $\dot K^\eps$ is a non-negative process, we thus deduce the existence of $\Gamma^\varepsilon$ such that \eqref{eq:Keps} holds. Since 
	\[
	Y^{\eps}_T=y-\int_0^T H^{\rm F}_s(X_{\cdot\wedge s},Z_s,\Gamma_s^\eps,\alpha_s) \mathrm{d}s+\int_0^T Z_s\cdot \mathrm{d}X_s+\frac12 \int_0^T \text{Tr} [\sigmah^2_s \Gamma^\eps_s]\d s=Y_T^{y,Z,\Gamma^\eps,\alpha},
	\]
	we then find that $(Z,\Gamma^\eps)\in {\Cc^\alpha}$ and, recalling that $K=K^\eps=0$, $\P^\star\as$, we see that $\Gamma^\eps=\Gamma$, $\d t \otimes\d \P^\star\ae$, and deduce that $Y=Y^\eps,$ $\P^\star\as$ In particular, $Y^\eps_T=g(X_{\cdot\wedge T}),\;\P^\star\as$ Thus, from \Cref{prop-response-follower-niceform}, we deduce that $(\P^\star,\beta^\star)$ satisfies \eqref{eq:optimal-effort-reformulated} and thus $(\P^\star,\beta^\star)\in \Rc^\star(y,Z,\Gamma^\eps,\alpha)$. Similarly to the conclusion in part $(i)$, this implies that $ \hat V_{\rm L}\geq V_{\rm L} $.
\end{proof}

\begin{remark}\label{rmk.reformulation.pes}
	{Following on {\rm \Cref{rem:nonuniq}\ref{rem:nonuniq.3}} and inspecting the proof of {\rm\Cref{prop.equivformulationleader}}, we have that $V_{\rm L}^{\text{\scalefont{0.8}\textup{Pes}}}=\hat V_{\rm L}^{\text{\scalefont{0.8}\textup{Pes}}}$ where
		\[
		\hat V_{\rm L}^{\text{\scalefont{0.8}\textup{Pes}}}\coloneqq  \sup_{y \in\R} \sup_{(\alpha,Z,\Gamma)\in \Ac\times \Cc^{\smallfont{\alpha}} } \newinf_{(\P,\beta)\in\Rc^\smallfont{\star}(y,\alpha,Z,\Gamma)}\bigg\{
		\E^{\bar\P} \bigg[   \int_0^T C_s(X_{\cdot\wedge s},\alpha_s,\beta_s) \d s + G(X_{\cdot \wedge T}) \bigg]\bigg\}.
		\]}
\end{remark}

\section{Solving the problem of the leader: strong formulation}\label{sec:solve_leader}

In this section, we use the techniques developed in \cite{bouchard2010optimal,bouchard2009stochastic} based on the geometric dynamic programming principle \cite{soner2002dynamic,soner2002stochastic}, to study \emph{Markovian} stochastic target control problems. To take full advantage of the standard tools from stochastic target problems, we bring ourselves to a Markovian setting and study the strong formulation of \eqref{eq:leader-problem-reformulated}. More precisely, we assume here that the coefficients of the Stackelberg game, namely $C$, $c$, $G$, $g$ as well as $\lambda$ and $\sigma$, only depend on $X$ at time $t$ through its value $X_t$. However, as noticed when solving the illustrative example, considering a Markovian framework does not prevent the leader's strategy $\alpha$ to be a function of the path of $X$. 

\smallskip
Indeed, in such a Markovian setting, and as already mentioned in \Cref{rk:markovian}, the leader's reformulated problem naturally becomes a Markovian control problem, thanks to the consideration of the additional state variable $Y$. In particular, although ultimately the optimal strategy at time $t$ for the leader can be viewed as a `feedback' function, depending on time $t$, $X_t$ and $Y_t$, one should keep in mind that, in general, $Y_t$ depends on the trajectory of $X$ up to time $t$, ensuring that we are still considering general closed-loop strategies for the leader, and not a `Markovian' or `feedback' restriction. Regarding the strong formulation, we expect it to be equivalent to $\hat V_{\rm L}$, and refer to \Cref{rmk.strongf}.

\subsection{Characterising the reachability set}
In this setting, $(\Omega,\Fc_T,\F,\P)$ denotes an abstract complete probability space supporting a $\P$--Brownian motion, which we still denote $W$, and $\F$ denotes the filtration generated by $W$, augmented under $\P$ so that it satisfies the usual conditions. In addition, the dependence of the data of the problem on $(t,x)\in[0,T]\times \Cc([0,T]; \R^d)$ is only through $(t,x(t))\in [0,T]\times \R^\xdim$. With a slight abuse of notation, we now write $c(t,x(t),a,b)$ instead of $c_t(x,a,b)$---and similarly for all the other mappings introduced in the previous sections---and thus without any risk of misunderstanding, consider now the maps as defined on $\R^d$ instead of $\Cc([0,T]; \R^d)$.

\smallskip
In light of \Cref{prop-response-follower-niceform}, by a classical measurable selection argument, we introduce $\Bc^\star_o$ as the set of Borel-measurable maps $b^\star: [0,T]\times\R^\xdim\times \R^\xdim\times\S^\xdim\times A \longrightarrow B$ such that 
\begin{align*}
	H^{\rm F}(t,x,z,\gamma,a)  =  h^{\rm F}(t,x,z,\gamma,a,b^\star(t,x,z,\gamma,a)),
\end{align*}
for all $(t,x,z,\gamma,a)\in[0,T]\times\R^\xdim\times \R^\xdim\times\S^\xdim\times A$.
We now topologise $\Bc^\star_o$.
	Consider the measurable space $(O,\textup{O}, \lambda)$, where $O\coloneqq [0,T]\times\R^\xdim\times \R^\xdim\times\S^\xdim\times A$, and $\textup{O}$ and $\lambda$ denote the Borel $\sigma$-algebra and Lebesgue measure on $O$, respectively. 
	We see $\Bc^\star_o$ as a subspace of $\L^1(O , \nu)$, the space of $\textup{O}$-measurable mappings on $O$ integrable with respect to $\d \nu \coloneqq  C \e ^{-\|\cdot \|}\d \lambda $,  where $C>0$ is a normalising constant. 
	In this way, as a subspace of a separable metric space, $\Bc^\star_o$ is separable. 
Lastly, for any $b^\star\in \Bc^\star_o$ and $\varphi\in\{C,c,\lambda,\sigma, \lambda\sigma, \sigma\sigma^\t\}$ we define 
\begin{align}\label{eq.convention.b} 
	\varphi^{b^\smallfont{\star}\!}(t,x,a,z,\gamma)\coloneqq  \varphi(t,x,a,b^\star(t,x,z,\gamma,a)).
\end{align}
We introduce the following set of assumptions, which we comment on in \Cref{rmk.strongf}.

\begin{assumption}\label{assumption.strongtarget} In addition to {\rm \Cref{Assumption.data}}, we assume that
	\begin{enumerate}[label=$(\roman*)$, ref=.$(\roman*)$,wide,  labelindent=0pt]
		\item \label{assumption.strongtarget.2}  
		$c$, $\sigma$ and $\sigma \lambda$ are Lipschitz-continuous in $(x,b)$, uniformly in $(t,a)$;
		\item \label{assumption.strongtarget.3} 
		$\Bc^\star\neq \emptyset$, where
	\end{enumerate}
			\[ \Bc^\star \coloneqq \big\{ b^\star \in \Bc^\star_o: b^\star(t,x,z,\gamma,a) \text{ is Lipschitz-continuous in }(x,z),\text{uniformly in }(t,\gamma,a)\big\}.\]
\end{assumption} 

We let $\mathfrak{C}$ be the family of tuples $(\alpha, Z,\Gamma,b^\star)$ consisting of $\F$-predictable processes $(\alpha, Z,\Gamma):[0,T]\times\Omega\longrightarrow A\times \R^\xdim\times\S^\xdim$ and $b^\star\in \Bc^\star$ such that, for some $p>1$,
\[ 
\|Z\|_{\H^{\raisebox{1pt}{\smallfont{p}}}}^p  +\| \Gamma \|_{\G^{\raisebox{1pt}{\smallfont{p}}}}^p+\|b^\star\|_{\L^{\raisebox{1pt}{\smallfont{p}}}}^p <+\infty.\footnotemark
\] \footnotetext{Here and henceforth, see \Cref{sec:funcspace}, $\|\cdot\|_{\H^{\raisebox{1pt}{\smallfont{p}}}}$ and $\|\cdot\|_{\G^{\raisebox{1pt}{\smallfont{p}}}}$ denote $\|\cdot\|_{\H^{\raisebox{1pt}{\smallfont{p}}}(\F,\P)}$ and $\|\cdot\|_{\G^{\raisebox{1pt}{\smallfont{p}}}(\F,\P)}$ with $\hat \sigma=I_{d}$, respectively. See additional comments in \Cref{rmk.strongf}.}To alleviate the notation, we use $\upsilon$ to denote a generic element of $\mathfrak{C}$ and $\hat \upsilon=(\alpha, Z,\Gamma)$ its first three components. With this, given $t\in[0,T]$, $(x,y)\in \R^{d+1}$ and $\upsilon \in \mathfrak{C}$, the controlled state processes are given by
\begin{align}\label{eq.state.strongformulation}
	\begin{split}
		X_u^{t,x,\upsilon} & = x +  \int_t^u  (\sigma \lambda)^{b^\smallfont{\star}\!} (s,X_s^{t,x,\upsilon},\hat\upsilon_s  ) \d s + \int_t^u \sigma^{b^\smallfont{\star}\!} (s,X_s^{t,x,\upsilon},\hat\upsilon_s ) \mathrm{d}W_s,\; u\in[t,T],\\
		Y_u^{t,x,y,\upsilon} & = y  - \int_t^u c^{b^\smallfont{\star}\!} (s,X_s^{t,x,\upsilon},\hat\upsilon_s ) \mathrm{d}s + \int_t^u Z_s \cdot\sigma ^{b^\smallfont{\star}\!}(s,X_s^{t,x,\upsilon},\hat\upsilon_s  ) \mathrm{d}W_s,\; u\in[t,T].
	\end{split}
\end{align}

With this, we define the problem
\begin{align}\label{eq:leader-strong-problem-reformulated-2}
	\begin{split}
	\tilde V_{\rm L} &\coloneqq  \sup_{y\in\R} V(0,x_0,y), \\
	\text{with} \; V(t,x,y) &\coloneqq  \sup_{ \upsilon \in\Cf(t,x,y)} \bigg\{\E^{\P}  \bigg[   \int_t^T C^{b^\smallfont{\star}\!} (s,X_s^{t,x,\upsilon},\hat\upsilon _s ) \d s + G(X_T^{t,x,\upsilon})  \bigg]\bigg\},
	\end{split}
\end{align}
and where, for $(t,x,y)\in[0,T]\times\R^{d+1}$,
\[
\mathfrak{C}(t,x,y)\coloneqq  \big\{ \upsilon\in \mathfrak{C}: \hat\upsilon \text{ is independent of } \Fc_t \text{, and } Y_T^{t,y,x,\upsilon} = g(X_T^{t,x,\upsilon}), \ \P\as \big\}.
\]
\begin{remark}\label{rmk.strongf}
	Let us comment on the previous formulation. 
	\begin{enumerate}[label=$(\roman*)$, ref=.$(\roman*)$,wide,  labelindent=0pt]
		\item {\color{black} We remind the reader that in the strong formulation, the background probability measure $\P$ is fixed. \textcolor{black}{Consequently, the norms in the definition of $\Cf$ not only coincide with those in the standard literature but also, under the assumptions on $\sigma$ and $\sigma\sigma^\t$, are equivalent to those used in the previous section.} 
			In particular, contrary to the weak formulation, the family $\mathfrak{C}$ does not depend on the choice of $\alpha\in \Ac$.
			We also remark that $\mathfrak{C}$ is a separable topological space. 
			This guarantees the geometric dynamic programming principle of {\rm \cite{bouchard2010optimal}}, based on {\rm \cite{soner2002dynamic}}, holds.}
			
		\item Let us now comment on {\rm \Cref{assumption.strongtarget}}. 
			The Lipschitz-continuity of $\sigma$ and $\lambda\sigma$ in $(x,b)$ and of $b^\star\in \Bc^\star$ in $x$ ensures that the process $X^{t,x,\upsilon}$ is well-defined.
			Notice that $Y^{t,x,y,\upsilon}$ is a direct definition.
			The Lipschitz-continuity in $(x,b)$ of $c$ in {\rm \Cref{assumption.strongtarget}} together with the Lipschitz-continuity of $b^\star\in \Bc^\star$ in $z$ will be used to establish a comparison principle for the target boundaries in {\rm \Cref{sec.verification.leader}}.
Regarding \Cref{assumption.strongtarget}\ref{assumption.strongtarget.3}, since we do not assume uniqueness of the maximisers of $h^{\rm F}$ in $b$, we focus on the class $\Bc^\star$ to guarantee uniqueness of strong solutions in the formulation of the problem though this can be certainly verified without it on a case-by-case basis.
			We stress this assumption is not needed for the main result of this section, namely \Cref{thm.verification}, so long as one can verify existence of a strong solution.
			Needless to say, from the point of view of the Leader, optimizing over $\Bc^\star$ is, in general, suboptimal.\footnote{In the case in which $\Bc^\star_o$ is single-valued, our assumption reads as the unique maximizing function $b^\star$ being Lipscthiz continuous. Then $\Bc^\star=\Bc^\star_o$ and the Leader is not solving a suboptimal problem.}
			We highlight that the assumption on the elements of $\Bc^\star$ is ultimately one on the primitives of the model.
			In general, membership in $\Bc^\star$ can be verified via convex or first-order analysis in cases where $h^{\rm F}$ exhibits a separable structure in $(x,b)$ or admits regular optimizers, respectively. This applies, for example, to the setting discussed in Section 2 in which $\Bc^\star$ is equal to $\Bc^\star_o$ and consists of a single element the moment $h^{\rm F}$ is strongly concave. 
		\item Let us also digress on the equivalence of the strong and weak formulations. 
		A potential roadmap to obtain this result uses {\rm \cite{karoui2013capacities2}}. 
		Indeed, to handle the constraint in both formulations, it is natural to embed it in the reward by means of a Lagrange multiplier $k \geq 0$ and the continuous penalty function $\Phi(y,x)\coloneqq |g(x)-y|^2$. 
		In this way, after establishing that strong duality holds, the results in {\rm \cite{karoui2013capacities2}} will allow us to obtain the equivalence of the strong and weak formulations for each element of a family of penalised problems, obtained by fixing $k$ and optimising over the corresponding controls. The only work needed to complete this argument is the strong duality results for the Lagrangian versions of both $\hat V_{\rm L}$ and $\tilde V_{\rm L}$. 
		We have refrained from writing such arguments as this will require, for instance, introducing the so-called relaxed formulation of $\hat V_{\rm L}$, which will unnecessarily encumber our analysis.
	\end{enumerate}
\end{remark}

As usual in stochastic target problems, we define the target reachability set as the set of triplets $(t,x,y)$ such that the set $\mathfrak{C}(t,x,y)$ is non-empty. That is
\[
V_g(t)\coloneqq  \big\{ (x,y)\in\R^{d+1}: \exists \upsilon \in\Cf(t,x,y), \; Y_T^{t,x,y,\upsilon} = g(X_T^{t,x,\upsilon}),\;  \P\as \big\}.
\]
We are interested in characterising $V_g(t)$, through the auxiliary sets
\begin{align*}
	V_g^-(t)\coloneqq  \big\{ (x,y)\in\R^{d+1}: \exists \upsilon \in\Cf(t,x,y) , \;  Y_T^{t,x,y,\upsilon} \geq g(X_T^{t,x,\upsilon}),\; \P\as \big\}, \\
	V_g^+(t)\coloneqq  \big\{ (x,y)\in\R^{d+1}: \exists \upsilon \in\Cf(t,x,y) , \;  Y_T^{t,x,y,\upsilon } \leq g(X_T^{t,x,\upsilon}), \; \P\as  \big\}.
\end{align*}
Notice that the inclusion $V_g(t) \subseteq V_g^-(t)\cap V_g^+(t)$ is immediate. The set $V_g^-(t)$ has been studied by \cite{soner2002stochastic,bouchard2009stochastic} and its boundary can be characterised through the auxiliary value function defined below
\begin{equation} \label{eq:auxiliary-w1}
	w^-(t,x)\coloneqq  \inf \{ y\in\R: (x,y)\in V_g^-(t)\}.
\end{equation}
It is known, see for instance \cite[Corollary 2.1]{bouchard2010optimal}, that the closure of $V_g^-(t)$ is given by
\[
\overline{V_g^-(t)} = \big\{ (x,y) : y\geq w^-(t,x) \big\}.
\]
Moreover, $w^-$ is a discontinuous viscosity solution of the following PDE,
\begin{equation}\label{eq:pde-inf}
	-\partial_t w(t,x) - H^-\big(t,x,\partial_x w(t,x),\partial_{xx}^2 w(t,x)\big) = 0,\; (t,x)\in[0,T)\times\R^d,
\end{equation}
with terminal condition $w(T^-,x)=g(x)$, and
where $H^-:[0,T]\times\R^\xdim\times\R^\xdim\times\S^\xdim  \longrightarrow\R$, $h^{\rm b}:[0,T]\times\R^\xdim\times\R^\xdim\times\S^\xdim \times A \times\R^d \times \S^\xdim\times \Bc^\star \longrightarrow\R$ are given by 
\begin{align}\label{eq:h_b}
	\begin{split}
		H^-(t,x,p,Q) & \coloneqq  \inf_{(\hat u,b^\smallfont{\star})\in N(t,x,p)} \big\{h^{\rm b}(t,x,p,Q,\hat u,b^{\star})\big\},  \\
		h^{\rm b}(t,x,p,Q,\hat u,b^{\star}) & \coloneqq  c^{b^\smallfont{\star}\!}(t,x,\hat u) + (\sigma \lambda)^{b^\smallfont{\star}\!}(t,x,\hat u)\cdot p + \frac12 \text{Tr}[ (\sigma{\sigma}^\t)^{b^\smallfont{\star}\!}(t,x,\hat u)Q],
	\end{split}
\end{align}
where we denoted $\hat u=(a,z,\gamma) \in A\times \R^\xdim\times\S^\xdim$. Since $\sigma\sigma^\t$ is invertible by assumption
\begin{align}\label{eq.optmatchvol}
	N(t,x,p) &\coloneqq \{ (\hat u,b^{\star})\in A\times \R^d\times\S^\xdim  \times \Bc^\star :  (\sigma^\t)^{b^\smallfont{\star}\!}(t,x,\hat u) (z-p)=0\} \nonumber \\ &= A\times \{p\} \times\S^\xdim  \times \Bc^\star. 
\end{align}
Similarly, by doing a change of variables and following the same ideas, the closure of $V_g^+(t)$ can be characterised as follows
\[
\overline{V_g^+(t)} = \big\{ (x,y) : y\leq w^+(t,x) \big\},
\]
where the auxiliary value function $w^+$ is defined by
\begin{equation} \label{eq:auxiliary-w2}
	w^+(t,x)\coloneqq  \sup \{ y\in\R: (x,y)\in V_g^+(t)\},
\end{equation}
and it is a discontinuous viscosity solution of the PDE
\begin{equation}\label{eq:pde-sup}
	-\partial_t w(t,x)-H^+\big(t,x,\partial_xw(t,x),\partial_{xx}^2 w(t,x)\big) = 0, \; (t,x)\in[0,T)\times\R^d,
\end{equation}
with terminal condition $w(T^-,x)=g(x), \; x\in\R^d$, and where $H^+:[0,T]\times\R^\xdim\times\R^\xdim\times\S^\xdim  \longrightarrow\R$ is given by
\[
H^+(t,x,p,Q)\coloneqq  \sup_{(\hat u,b^\smallfont{\star})\in N(t,x,p)} \big\{ h^{\rm b}(t,x,p,Q, \hat u,b^{\star})\big\} .
\]

We propose the two auxiliary value functions as the upper and lower boundaries of $V_g(t)$, and thus define the set 
\[
\hat{V}_g(t)\coloneqq  \{ (x,y): w^-(t,x) \leq y \leq w^+(t,x) \},
\]
which, provided the upper and lower boundaries are sufficiently separated before $T$, corresponds to the closure of the reachability set $V_G(t)$, as we prove next. For this, we introduce
\[
\delta_\eps\coloneqq \inf_{(t,x)\in [0,T-\eps]\times \R^\smallfont{\xdim}} \big\{| w^-(t,x)-w^+(t,x)|\big\}, \;  \eps>0.
\]

\begin{lemma}\label{lemma.boundaries}
	Let $t\in[0,T]$. The following holds
	\begin{enumerate}[label=$(\roman*)$, ref=.$(\roman*)$,wide,  labelindent=0pt]
		\item $V_g(t)\subseteq \hat{V}_g(t).$
		\item If in addition $\delta_\eps>0$ for any $\eps>0$, and $w^-$ and $w^+$ are continuous, then, ${\rm int}\big(\hat{V}_g(t)\big) \subseteq V_g(t)$ and $\overline{V_g(t)} = \hat{V}_g(t)$.
	\end{enumerate}
\end{lemma}

\begin{remark}
	Let us provide a sufficient structural condition for the assumption $\delta_\eps>0$ for any $\eps>0$, before presenting the proof of {\rm \Cref{lemma.boundaries}}. We claim that it holds if {\rm PDE} \eqref{eq:pde-inf} satisfies a comparison principle, as we will establish in {\rm \Cref{sec.appen.comp.boundaries}}, and there is $\eta>0$ such that
	\begin{equation}\label{eq:sufficient-delta-eps}
		H^+(t,x,p,Q) \geq H^-(t,x,p,Q) + \eta, \; \forall (t,x,p,Q)\in [0,T]\times\R^\xdim\times\R^\xdim\times\S^\xdim.
	\end{equation}
	Indeed, under this condition, it is easy to see that the function $\hat{w}^-(t,x)\coloneqq  w^-(t,x) + \eta (T-t)$ is a discontinuous viscosity sub-solution to {\rm PDE} \eqref{eq:pde-inf}. Therefore, from the comparison principle we have $\hat{w}^- \leq w^+$, which implies $\delta_\eps >0$ for any $\eps>0$. A similar argument works if {\rm PDE} \eqref{eq:pde-sup} satisfies a comparison principle instead.
\end{remark}

\begin{proof}[Proof of {\rm \Cref{lemma.boundaries}}]
	Let us argue $(i)$. 
	Let $(x,y)\in V_g(t)$, then there exists $\upsilon \in \mathfrak{C}(t,x,y)$ such that $Y_T^{t,x,y,\upsilon } = g(X_T^{t,x,y,\upsilon }), \; \P\as$ 
	Then it is clear that $(x,y)$ belongs to both auxiliary sets $ V_g^-(t)$ and $V_g^+(t)$, that is, $(x,y)\in V_g^-(t) \cap V_g^+(t)$. 
	Since $\hat{V}_g(t)= \overline{V_g^-(t)} \cap  \overline{V_g^+(t)}$, it follows that $V_g(t)\subseteq \hat{V}_g(t)$.
	\smallskip
	
	As for $(ii)$, we first note that the second part of the statement, \emph{i.e.} $\overline{V_g(t)} = \hat{V}_g(t)$, follows from the inclusions $\text{int}(\hat{V}_g(t)) \subseteq V_g(t) \subseteq \hat{V}_g(t)$ by taking closure. 
	Let us now argue $\text{int}(\hat{V}_g(t)) \subseteq V_g(t)$.
	To increase the readability of the proof, given $(t,x,y)\in [0,T]\times \R^{\xdim+1}$ and $\upsilon \in \Cf(t,x,y)$, we will say that $\upsilon$ satisfies $(U)$ or $(L)$ whenever $Y_T^{t,x,y,\upsilon} \geq  g(X_T^{t,x,\upsilon}),\; \P\as$, or, $Y_T^{t,x,y,\upsilon} \leq  g(X_T^{t,x,\upsilon}),\; \P\as$, respectively.
	Let $t\in [0,T]$ and $(x,y)\in {\rm int}(\hat{V}_G(t))$.
	We argue in 2 steps.\smallskip
	
	{\bf Step 1.} We fix $n\in\N^\star$ and construct an admissible control up to $T_n\coloneqq T- n^{-1}$.
	Since $(x,y)\in {\rm int}(\hat{V}_G(t))$, by continuity, we have that $w^-(t,x) < y < w^+(t,x)$.
	Thus, in particular, there is $\upsilon^{0,n}  \in \Cf(t,x,y)$ satisfying $(U)$.
	Let $X^{0,n}\coloneqq  X^{t,x,\upsilon^{\smallfont{0}\smallfont{,}\smallfont{n}}}, Y^{0,n} \coloneqq Y^{t,x,y,\upsilon^{\smallfont{0}\smallfont{,}\smallfont{n}}}$.
	By \cite[Corollary 2.1]{bouchard2010optimal}, $Y^{{0},n}_s \geq w^-(s,X^{ 0,n}_s)$, $s\in [t,T]$.
	\textcolor{black}{We have two cases.
		If $Y^{0,n}_{s} = w^-(s,X^{0,n}_{s})$ for some $s\in[t,T]$, by definition of $w^-$, we find that $(x,y)\in V_g(t)$ as desired and conclude the proof.
		Otherwise, we have that $Y^{0,n}_s > w^-(s,X^{0,n}_s)$, $s\in [t,T]$.
		Let $\gamma_0\coloneqq \inf_{s\in[t,T]} \{Y_s^{0,n}-w^-(s,X_s^{0,n})\}$, and note that $\gamma_0>0$, $\P\as$, thanks to the $\omega$-by-$\omega$ continuity of $[t,T] \ni s\longmapsto Y_s^{0,n}-w^-(s,X_s^{0,n})$. 
		Thus, there is random variable $N_0$ with values in $\N$ such that $\delta_{n^{-1}}/N_0<\gamma_0,\; \P\as$}	 
	We then define the sequence of $\F$--stopping times $(\tau_k^n)_{ k\in\{0,\dots, k(n)\}}$, with $k(n)\in\N$ to be defined, recursively as follows
	\begin{align*}
		\tau_0^n& \coloneqq \inf \big \{ s\geq t: w^+\big (s,X^{0,n}_s\big)-Y^{0,n}_s \leq  \delta_{n^{\smallfont{-}\smallfont{1}}}/{N_0}\big \}\wedge T_n.
	\end{align*}
	If $\tau_0^n=T_n$, we set $k(n)=0$ and conclude the construction. 
	Otherwise, by continuity, we have that $w^+\big (\tau_0^n, X_{\tau_\smallfont{0}^\smallfont{n}}^{0,n} \big)-Y_{\tau_\smallfont{0}^\smallfont{n}}^{0,n}= \delta_{n^{\smallfont{-}\smallfont{1}}}/N_0 $.
	By definition of $\delta_\eps$, we have that
	\begin{align}\label{eq.lemmasetV4}
		\big(X_{\tau_\smallfont{0}^\smallfont{n}}^{0,n},Y_{\tau_\smallfont{0}^\smallfont{n}}^{0,n}\big)\in {\rm int}\big(\hat{V}_g(\tau_0^n)\big),\; \P\as,\; \text{\emph{i.e.}}\; w^-\big (\tau_0^n, X_{\tau_\smallfont{0}^\smallfont{n}}^{0,n} \big)< Y_{\tau_\smallfont{0}^\smallfont{n}}  < w^+\big (\tau_0^n, X_{\tau_\smallfont{0}^\smallfont{n}}^{0,n} \big),\; \P\as
	\end{align}
	Thus, by \cite[Corollary 2.1]{bouchard2010optimal}, there is $\upsilon^{1,n} \in\Cf(t,x,y)$, satisfying $(L)$ and $\upsilon^{1,n}=\upsilon^{0,n}$ on $[t,\tau_0^n)$. Let now
	\begin{align*}
		\tau_{1}^n& \coloneqq \inf \big \{ s\geq \tau_0^n: Y_s^{1,n} - w^-(s,X_s^{1,n})\leq  \delta_{n^{\smallfont{-}\smallfont{1}}}/N_1 \big \}\wedge T_n,
	\end{align*}
	for $X^{1,n}\coloneqq  X^{t , x , \upsilon^{\smallfont{1}\smallfont{,}\smallfont{n}}}, \; Y^{1,n}\coloneqq Y^{t,x,y, \upsilon^{\smallfont{1}\smallfont{,}\smallfont{n}}}$
	and {$N_1$ defined so that $\delta_{n^{-1}}/N_1<\gamma_1$, with $\gamma_1\coloneqq\inf_{s\in [\tau_0^n,T]}\{ w^+(s,X_s^{1,n})-Y_s^{1,n}\}$}. 
	Arguing as above, by definition of $\tau_1^n$, we find that $(X_{\tau_\smallfont{1}^\smallfont{n}}^{1,n},Y_{\tau_\smallfont{1}^\smallfont{n}}^{1,n})\in {\rm int}\big(\hat{V}_g(\tau_1^n)\big),\; \P\as$
	Thus, again by \cite[Corollary 2.1]{bouchard2010optimal}, there is $\upsilon^{2,n} \in\Cf(t,x,y)$ such that $(U)$ holds and $\upsilon^{2,n}=\upsilon^{1,n}$ on $[\tau_0^n,\tau_1^n)$.
	Recursively, for $k\in\N^\star$ we let $X^{k,n}\coloneqq  X^{t,x , \upsilon^{\smallfont{k}\smallfont{,}\smallfont{n}}},\; Y^{k,n}\coloneqq Y^{t,x,y , \upsilon^{\smallfont{k}\smallfont{,}\smallfont{n}}}$
	\begin{align*}
		\tau_{2k}^n& \coloneqq \inf \big \{ s\geq \tau_{2k-1}^n: w^+ (\tau_{k-1}^n,X_s^{k,n})-Y_s^{k,n}\leq   \delta_{n^{\smallfont{-}\smallfont{1}}}/N_{2k} \big \}\wedge T_n, \\
		\tau_{2k+1}^n& \coloneqq \inf \big \{ s\geq \tau_{2k}^n: Y_s^{k,n }- w^- (\tau_{k-1}^n,X_s^{k,n})\leq   \delta_{n^{\smallfont{-}\smallfont{1}}}/N_{2k+1} \big \}\wedge T_n, 
	\end{align*}
	and find $\upsilon^{k+1,n}\in \mathfrak{C}(t,x,y)$ for which $(X_{\tau_\smallfont{k}^\smallfont{n}}^{k,n},Y_{\tau_\smallfont{k}^\smallfont{n}}^{k,n})\in {\rm int}\big(\hat{V}_g(\tau_k^n)\big),\; \P\as$ We now claim that there is a process $k(n)$ with values in $\N$ such that $\tau^n_{k(n)}=T_n$, $\P\as$ 
	Indeed, by continuity of $w^-$ and $w^+$, the mappings 
	\[
	[t,T_n]\ni s\longmapsto w^+\big (s,X_s^{t,x, \upsilon }\big)-Y_s^{t,x,y,  \upsilon}, \;\text{and}\; [t,T_n]\ni s\longmapsto  Y_s^{t,x ,y ,\upsilon} - w^-\big(s,X_s^{t,x,\upsilon} \big),
	\]
	are, $\omega$-by-$\omega$, uniformly continuous for any $\upsilon\in \Cf(t,x,y)$. Hence, there exists a constant $\bar \gamma_n>0$ and a $[\bar \gamma_n, T_n]$-valued random variable $\gamma_n$ such that, $\|\tau_k^n-\tau_{k-1}^n\|_\infty>\gamma_n,$ $\P\text{\rm --a.s.},\; k\in \N$. This proves the claim. At the end of this construction, we set $\upsilon^n\coloneqq \upsilon^{k(n),n}$, 
	and notice that $\upsilon^{n} \in  \Cf(t,x,y)$ and
	\begin{align}\label{eq.lemmasetV1}
		w^-\big (T_n, X_{T_\smallfont{n}}^n  \big)< Y_{T_\smallfont{n}}^n < w^+\big (T_n, X_{T_\smallfont{n}}^n \big),\; \P\as, \text{ for } X^n\coloneqq  X^{t,x,\upsilon^\smallfont{n}}, Y^n\coloneqq  Y^{t,x,y,\upsilon^\smallfont{n}}.
	\end{align}

	{\bf Step 2.} We iterate the previous construction. 
	From here on, we can repeat {\bf Step 1}, with $(T_n, X_{T_\smallfont{n}}^n )$, control $\upsilon^n$, and $n+1$ playing the role of $(t,x)$, $\upsilon^{0,n}$ and $n$, respectively. 
	With this, we obtain the existence of $\upsilon^{n+1} \in \Cf(t,x,y)$, such that, by uniform continuity, \eqref{eq.lemmasetV1} holds at $(T_{n+1},X^{n+1}_{T_{\smallfont{n}\smallfont{+}\smallfont{1}}})$ and $Y^{n+1}_{T_{\smallfont{n}\smallfont{+}\smallfont{1}}}$.
	Iterating this construction, we find $\upsilon$ which is well-defined $\d t\otimes \d \P\ae$ on $[0,T]\times\Omega$.\footnote{Indeed, the construction allows us to define said process $\d t\otimes\d \P\ae$ on $[t,T)\times \Omega$, and consequently, $\d t\otimes \d \P\ae$ on $[0,T]\times\Omega$.}\smallskip
	
	To conclude $(x,y)\in V_g(t)$, let $n\longrightarrow \infty$ in \eqref{eq.lemmasetV1}, and notice that by continuity of $w^-$ and $w^+$ we have that $Y_{T}^{t,x,y,\upsilon}=g(X_T^{t,x,\upsilon})$ as desired.\qedhere
\end{proof}

\subsection{PDE characterisation for the problem of the leader}\label{sec.verification.leader}

We begin our analysis by establishing a verification theorem for the solutions to PDEs \eqref{eq:pde-inf} and \eqref{eq:pde-sup}.
That is, for the boundaries of the domain of Problem \ref{eq:leader-strong-problem-reformulated-2}.
The proof is deferred to \Cref{sec.appen.comp.boundaries} and relies on establishing a comparison theorem for the solutions to PDEs \eqref{eq:pde-inf} and \eqref{eq:pde-sup} by classical arguments.

\begin{theorem}\label{thm.aux.func}
	Let $u$ and  $v$ be continuous viscosity solutions to \eqref{eq:pde-inf} and \eqref{eq:pde-sup}, respectively, with linear growth. Then $u=w^-$ and $v=w^+$. 
\end{theorem}

Having conducted the analysis of the auxiliary boundary functions $w^-$ and $w^+$, we are in a position to provide a verification theorem for Problem \ref{eq:leader-strong-problem-reformulated-2}.
\Cref{thm.verification} below provides a PDE characterisation for the intermediate problem of the leader under the CL information structure.
Let us remark that once $V(t,x,y)$ is found it only remains to optimise over $y\in \R$.
\smallskip

To ease the notation, we will use ${\rm x}\in \R^{\xdim+1}$ and $u\in A\times \R^\xdim\times\S^\xdim \times \Bc^\star\eqqcolon U$, to denote the values of the state and control processes associated with Problem \ref{eq:leader-strong-problem-reformulated-2}, that is, we make the convention ${\rm x}=(x,y)$ and $u=(a,z,\gamma,b^\star)$. In this way, recall \eqref{eq.convention.b}, we let $C(t,{\rm x},u)\coloneqq C^{b^\smallfont{\star}\!}(t,x,a,z,\gamma)$ and similarly for the other functions in the analysis. Moreover, we denote the drift and volatility coefficients, $(\mu,\vartheta):[0,T]\times  \R^{\xdim+1}\times  U \longrightarrow \R^{\xdim+1} \times \R^{(\xdim+1)\times \bmdim} $ associated with the state process ${\rm X}\coloneqq (X,Y)$ by
\begin{align*}
	\mu(t,{\rm x}, u)\coloneqq \begin{pmatrix}
		\sigma \lambda  (t,x,u)\\
		-c  (t,x,u)
	\end{pmatrix},\;
	\vartheta(t,{\rm x},u)\coloneqq \begin{pmatrix}
		\sigma  (t,x,u) \\
		z\cdot \sigma (t,x,u)
	\end{pmatrix}.
\end{align*}
Given $\textup{w}\in C^{1,2}([0,T)\times\R^\xdim)$, we define the sets
\begin{align*}
	U^{-}(t, x,\textup{w})&\coloneqq \big\{ u\in U:  \sigma^\t  (t,x,u) ( z- \partial_x \textup{w}(t,x))=0,\; -\partial_t \textup{w} - h^{\rm b}(t,x,\partial_x \textup{w}, \partial_{xx}^2 \textup{w}, u)  \geq 0\big\},\\
	U^+(t,x,\textup{w})&\coloneqq \big\{ u\in U:  \sigma^\t  (t,x,u) (z-  \partial_x \textup{w}(t,x))=0, \; -\partial_t \textup{w} - h^{\rm b}(t,x,\partial_x\textup{w}, \partial_{xx}^2 \textup{w}, u) \leq 0\big\},
\end{align*}
and, for $i\in\{-,+\}$, introduce the Hamiltonians $(H^{\rm L},H^{i,w}):[0,T]\times  \R^{\xdim+1}\times \R^{\xdim+1} \times \S^{\xdim+1} \longrightarrow \R$, given by
\begin{align}\label{eq.HamiltonianLeader}
	\begin{split}
		{H}^{\rm L}(t,{\rm x} ,{\rm p},{\rm Q}) &\coloneqq \sup_{u\in U} \big\{  {h}^{\rm L} (t,{\rm x},{\rm p}, {\rm Q}, u)\big\},\\ {\rm H}^{i,w}(t,{\rm x} ,{\rm p},{\rm Q}) &\coloneqq \sup_{u\in U^\smallfont{i}(t, x,w)} \big\{ {h}^{\rm L}(t,{\rm x},{\rm p}, {\rm Q}, u)\big\},
	\end{split}
\end{align}
where
\[ {h}^{\rm L}(t,{\rm x},{\rm p}, {\rm Q},u)\coloneqq C(t,x,u)+ \mu(t,{\rm x}, u)\cdot {\rm p}+\frac{1}2 {\rm Tr}[\vartheta\vartheta^\t(t, {\rm x}, u) {\rm Q}].\]

Below, $\Tc_T$ denotes the family of $\F$--stopping times with values on $[0,T]$.
With this, we have all the elements necessary to state our main result, which is the following verification theorem. 
\begin{theorem}\label{thm.verification}
	\begin{enumerate}[label=$(\roman*)$, ref=.$(\roman*)$,wide,  labelindent=0pt]
		\item 
		Let $\textup{w}^i \in C^{1,2}([0,T)\times\R^\xdim)\cap C^{0}([0,T]\times\R^\xdim)$, $i\in\{-,+\}$, be solutions to \eqref{eq:pde-inf} and \eqref{eq:pde-sup}, respectively with linear growth. Then, $\textup{w}^i(t,x) =w^i(t,x)$, for $w^i(t,x)$ given by \eqref{eq:auxiliary-w1} or \eqref{eq:auxiliary-w2}, respectively.
		\item Moreover, if $v\in C^{1,2}([0,T)\times\R^{\xdim}\times  [\textup{w}^-,\textup{w}^+])\cap C^{0}([0,T]\times\R^{\xdim}\times  [\textup{w}^-,\textup{w}^+])$\footnote{Here $[0,T]\times \R^\xdim\times [\textup{w}^-,\textup{w}^+]\coloneqq \{ (t,x,y) \in [0,T]\times\R^\xdim\times \R:  \textup{w}^-(t,x)\leq y\leq \textup{w}^+(t,x)  \}$.} satisfies
		\begin{align}\label{eq.fullHJB}
			\begin{cases}
				- \partial_t v-{\rm H}^{\rm L}(t,{\rm x},\partial_{\rm x} v, \partial_{\rm xx}^2 v)=0, \;(t,x,y)\in  [0,T)\times \R^\xdim\times (\textup{w}^-(t,x),\textup{w}^+(t,x)),\\
				-\partial_t v-{\rm H}^{i,w^i}(t,{\rm x},\partial_{\rm x} v, \partial_{\rm xx}^2 v)=0, \;(t,x,y)\in  [0,T)\times \R^\xdim\times \{ \textup{w}^i(t,x)\},\;  i\in\{-,+\},\\
				v(T^-,{\rm x})=G(x),\; (x,y)\in \R^{\xdim}\times\{g(x)\}.
			\end{cases}
		\end{align}
		Moreover, suppose that
		\begin{itemize}
			\item the family $\{v(\tau,{X}_\tau^\upsilon,Y_\tau^\upsilon)\}_{\tau\in \Tc_T}$ is uniformly integrable for all controls $\upsilon \in \Cf;$
			\item there exists $\upsilon^\star:[0,T]\times \R^d\times  [\textup{w}^-,\textup{w}^+] \longrightarrow A\times \R^d\times \S^\xdim\times \Bc^\star$ attaining the maximisers in $H^{\rm L}$, $H^{i,\textup{w}^i}$, $i\in \{+,-\};$
			\item there is a unique strong solution to the system \eqref{eq.state.strongformulation} with control $(\alpha^\star_\cdot, Z_\cdot^\star,\Gamma_\cdot^\star,b^\star) \coloneqq \upsilon^\star(\cdot,X_\cdot,Y_\cdot);$
			\item $(\alpha^\star, Z^\star,\Gamma^\star,b^\star) \in \Cf$.
		\end{itemize}
		Then, $V(t,x,y)=v(t,x,y)$, and $(\alpha^\star, Z^\star,\Gamma^\star,b^\star)$ is an optimal control for the problem $V(t,x,y)$.
	\end{enumerate}
\end{theorem}

\begin{remark}\label{rmk.verification}
	\begin{enumerate}[label=$(\roman*)$, ref=.$(\roman*)$,wide,  labelindent=0pt]
		\item 
		We remark that we could build upon one of the main results of {\rm \cite{bouchard2010optimal}} to characterise the functions $V$, $w^+$ and $w^-$ given by \eqref{eq:leader-strong-problem-reformulated-2}, \eqref{eq:auxiliary-w1}, and \eqref{eq:auxiliary-w2}, respectively, as viscosity solutions to---a relaxed version of---\eqref{eq:pde-inf}, \eqref{eq:pde-sup} and \eqref{eq.fullHJB}, respectively.
		In particular, if one can show that $V$, $w^+$, and $w^-$ are smooth and the associated Hamiltonians are continuous, the relaxation reduces to the above system.
		We refer to {\rm \cite{bouchard2010optimal}} for details.
		We have refrained from doing so as the above verification theorem gives the result most useful in solving any example in practice.
		In {\rm \Cref{ss:approach-example}}, we use the above result and search for solution to the above system directly.
		\item We also note that in the pessimistic case, see {\rm Remarks \ref{rem:nonuniq}} and {\rm \ref{rmk.reformulation.pes}}, we expect to arrive at
			\begin{align*}
			V^{\text{\scalefont{0.8}\textup{Pes}}}(t,x,y)
			\coloneqq  \sup_{ (\alpha,Z,\Gamma)\in\H^{\smallfont{p}}\times\H^\smallfont{p}\times\G^\smallfont{p}}\newinf_{b^\smallfont{\star}\in \Bc^\smallfont{\star}}  \bigg\{ & \E^{\P}    \bigg[   \int_t^T C^{b^\smallfont{\star}\!} (s,X_s^{t,x,\upsilon},\hat\upsilon _s ) \d s + G(X_T^{t,x,\upsilon})  \bigg]: \\ &\ (\alpha,Z,\Gamma,b^\star)\in \mathfrak{C}(t,x,y)\bigg\}.
			\end{align*}
			$V^{\text{\scalefont{0.8}\textup{Pes}}}(t,x,y)$ is the lower value of a zero-sum game with stochastic target constraints. 
			In this game, the leader seeks to maximise the criterion over the controls $(\alpha,Z,\Gamma)$ and the adversarial player seeks to minimise it over the follower's best responses $b^\star$.
			Crucially, the controls must lead to a state process $(X,Y)$ satisfying a stochastic target constraint i.e. $(\alpha,Z,\Gamma,b^\star)\in \mathfrak{C}(t,x,y)$.
			Then, if a geometric {\rm DPP} for zero-sum games is available, informally one would introduce the lower version of the Hamiltonians in \eqref{eq.HamiltonianLeader} leading to a lower version of the system in \eqref{eq.fullHJB}, or {\rm HJBI} system.
		\item It may happen that our Stackelberg equilibria end up being written as Markovian feedback functions of $t$ and $X_t$. More precisely, imagine that one manages to obtain, say, a smooth solution to the HJB equation for the leader's reformulated problem. This can then, thanks to the verification theorem above, lead to finding appropriate optimal controls 
\[
\alpha^\star_t\coloneqq a(t,X_t,Y_t),\; Z^\star_t\coloneqq z(t,X_t,Y_t),\; \Gamma^\star_t\coloneqq \gamma(t,X_t,Y_t).
\]
The question then boils down to whether $Y$ itself can be obtained as a Markovian functional of $X$. Though this does not have to be true in general, one way forward would be to consider that the triplet $(X^\star,Y,Z^\star)$ now solves the Markovian fully coupled FBSDE
\begin{align*}
	\begin{split}
		X_t^{\star} & = x +  \int_0^t  (\sigma \lambda)^{b^{\star}\!} \big(s,X_s^{\star}, a(s,X_s^\star,Y_s),Z^\star_s,\gamma(s,X_s^\star,Y_s) \big) \d s \\
		&\quad+ \int_0^t \sigma^{b^{\star}\!} (s,X_s^{\star},a(s,X_s^\star,Y_s),Z^\star_s,\gamma(s,X_s^\star,Y_s) ) \mathrm{d}W_s,\; t\in[0,T],\\
		Y_t & = g(X^\star_T)  +\int_t^T c^{b^{\star}\!} \big(s,X_s^{\star},a(s,X_s^\star,Y_s),Z^\star_s,\gamma(s,X_s^\star,Y_s) \big) \mathrm{d}s\\
		&\quad - \int_t^T Z^\star_s \cdot\sigma ^{b^{\star}\!}\big(s,X_s^{\star},a(s,X_s^\star,Y_s),Z^\star_s,\gamma(s,X_s^\star,Y_s) \big) \mathrm{d}W_s,\; t\in[0,T].
	\end{split}
\end{align*}
It is allegedly an extremely complex object in general, but if it were to happen that it is well-posed, then we could identify $Y$ as its unique solution, and if a decoupling field were to also exist (see \cite{ma2015well}), then we would automatically have that $Y_t=v(t,X^\star_t)$ for some map $v$. In this case, our Stackelberg equilibrium would actually be feedback. It is however unclear to us whether the above strategy can always be implemented.
	\end{enumerate}
\end{remark}

\begin{proof}
	We begin noticing that $(i)$ follows from \Cref{thm.aux.func}. From now on, we use $w$ and $\textup{w}$ interchangeably.
	We now argue $(ii)$. Let $t\in [0,T]$, $(x,y)\in V_g(t)$, $\upsilon  \in \mathfrak{C}(t,x,y)$, and $(X,Y)\coloneqq (X^{t,x,\upsilon},Y^{t,x,y,\upsilon})$ be given by \eqref{eq.state.strongformulation}. 
	We set ${\rm X} \coloneqq (X,Y)$.
	Thanks to \Cref{lemma.boundaries}, we have that $w^-(t,x)\leq y\leq w^+(t,x)$. 
	Moreover, the feasibility of $\upsilon$ gives that $w^-(s,X_s)\leq Y_s\leq w^+(s,X_s)$, $s\in [t,T]$, $\P\as$
	Otherwise, $w^-(s,X_s)> Y_s$ or $w^+(s,X_s)< Y_s$ for some $s\in [t,T]$ contradicts the feasibility of $\upsilon$ by definition of $w^-$ and $w^+$, see \eqref{eq:auxiliary-w1} and \eqref{eq:auxiliary-w2}. 
	Let 
	\begin{align*}
		\theta_1 &\coloneqq\inf \{ s> t :  Y_s=w^-(s,X_s), \text{ or, }Y_s= w^+(s,X_s)\},\\
		\theta_2 &\coloneqq\inf \{ s> \theta_1  :  w^-(s,X_s)<Y_s< w^+(s,X_s)\}\wedge T.
	\end{align*}
	and notice that $t\leq \theta_1 \leq T$ since by the feasibility of $\upsilon$, $w^-(T,X_T)=w^+(T,X_T)=g(X_T)=Y_T$.
	We now consider the process $v(t,{\rm X}_t)\coloneqq v(t,X_t,Y_t)$ and compute $v(t,{\rm X}_t)-v(\theta_2,{\rm X}_{\theta_\smallfont{2}})=v(t,{\rm X}_t)-v(\theta_1 ,{\rm X}_{\theta_\smallfont{1}})+v(\theta_1 ,{\rm X}_{\theta_\smallfont{1}})-v(\theta_2 ,{\rm X}_{\theta_\smallfont{2}})\eqqcolon I_1+I_2$.
	It follows from Itô's formula that
	\begin{align}\label{eq.verificationi1}
		I_1&= -\int_t^{\theta_\smallfont{1}} \Big( \partial_t  v (s,{\rm X}_s)\d s+\frac12{\rm Tr}[\partial_{\rm xx}^2v(s, {\rm X} _s)\d \langle {\rm X}  \rangle_s]\Big) -\int_t^{\theta_\smallfont{1}} \partial_{\rm x}   v(s,{\rm X}_s )\cdot  \d {\rm X}_s\nonumber\\
		&=  \int_t^{\theta_\smallfont{1}}   \Big(  {\rm H}^{\rm L} \big(s, {\rm X}_s    , \partial_{\rm x}  v(s, {\rm X}_s  ), \partial_{\rm xx}^2   v(s, {\rm X} _s ) \big) - {\rm h}^{\rm L} \big(s, {\rm X}_s  , \partial_{\rm x}  v(s, {\rm X}_s ), \partial_{\rm xx}^2   v(s, {\rm X}_s ) , \upsilon_s\big) \Big) \d s  \nonumber \\
		&\quad - \int_t^{\theta_\smallfont{1} } \big( \partial_{x}   v(s,{\rm X}_s ) ,  \partial_{y}   v(s,{\rm X}_s ) \big)^\top \cdot \big( \sigma(s,X_s,\upsilon_s) \d W_s,  Z_s\cdot \sigma(s,X_s,\upsilon_s) \d W_s\big)^\top\nonumber \\
		&\quad+ \int_t^{\theta_\smallfont{1}}    C(s,X_s,\upsilon_s)\d s\nonumber\\
		&\geq  \int_t^{\theta_\smallfont{1}} C(s,X_s,\upsilon_s) \d s-\int_t^{\theta_\smallfont{1} } \Big( \partial_x  v(s,{\rm X}_s) + \partial_y v(s,{\rm X}_s)Z_s  \Big)\cdot    \sigma(s,X_s,\upsilon_s) \d W_s,
	\end{align}
	where we used the fact, on $[t,{\theta_1 })$, $v$ satisfies the first equation in \eqref{eq.fullHJB}, computed the dynamics of ${\rm X}$ and added and subtracted $C$ to complete the term $h^{\rm L}$. The inequality follows from the definition of $H^{\rm L}$.
	\smallskip
	
	We now consider $I_2$. 
	Without loss of generality, we assume that $Y_{\theta_\smallfont{1}}=w^-(\theta_1,X_{\theta_\smallfont{1}})$, and note that $Y_s=w^-(s,X_s)$ for $s\in[\theta_1,\theta_2]$, $\P\as$
	By the uniqueness of their It\^o decomposition, we deduce that $Z_t=\partial_x w^-(t, X_t)$, and $\upsilon_t \in N(t,X_t, \partial_x w^-(t,X_t)), \d t\otimes \d\P\ae$ on $[\theta_1,\theta_2]\times \Omega$.
	With this, applying Itô's formula to $w^-(t,X_t)-Y_t$, $t\in [\theta_1,\theta_2]$, we find that
	\begin{align*}
		\int_{\theta_\smallfont{1}}^t \big( h^b(s,X_s,\partial_x  w^-(s, X_s),\partial^2_{xx}  w^-(s, X_s),\upsilon_s)-H^-(s,X_s,\partial_x  w^-(s, X_s),\partial^2_{xx}  w^-(s, X_s))\big)  \d s
	\end{align*}
	vanishes, which by the previous discussion implies that $\upsilon$ attains the infimum in \eqref{eq:h_b}; in particular, $\upsilon_t \in U^-(t,X_t),\d t\otimes \d \P\ae$
	Let $\bar v(t,x)\coloneqq v(t,x, w^-(t,x))$, so that
	\begin{align*}
		I_2 &=  -\int_{\theta_\smallfont{1}}^{\theta_\smallfont{2}} \bigg( \partial_t \bar v (s,X_s )\d s+\frac12{\rm Tr}[\partial_{xx}^2\bar v(s, X_s )\d \langle X\rangle_s]\bigg) -\int_{\theta_\smallfont{1}}^{\theta_\smallfont{2}} \partial_x \bar v(s,X_s )\cdot  \d X_s \\
		& =   -\int_{\theta_\smallfont{1}}^{\theta_\smallfont{2}}\Big( \partial_t v(s,X_s,w^-(s,X_s)) +  \sigma\lambda(s,X_s,\upsilon_s)\cdot \partial_x v(s,X_s,w^-(s,X_s))\Big)\d s\\
		&\quad -  \frac12 \int_{\theta_\smallfont{1}}^{\theta_\smallfont{2}} {\rm Tr}[\partial_{xx}^2\bar v(s, X_s)\d \langle X\rangle_s]   \\
		&\quad-\int_{\theta_\smallfont{1}}^{\theta_\smallfont{2}} \partial_y v(s,X_s,w^-(s,X_s))\big(  \partial_t w^-(s,X_s)   +\sigma\lambda(s,X_s,\upsilon_s) \cdot \partial_x w^- (s,X_s)   \big)\d s \\
		&\quad-\int_{\theta_\smallfont{1}}^{\theta_\smallfont{2}} \big( \partial_x  v(s,X_s,w^-(s,X_s)) + \partial_y v(s,X_s,w^-(s,X_s))  \partial_x w^-(s,X_s)\big)\cdot    \sigma(s,X_s,\upsilon_s) \d W_s \\
		& =  -\int_{\theta_\smallfont{1}}^{\theta_\smallfont{2}} \Big( \partial_t v(s,X_s,w^-(s,X_s)) +\sigma\lambda(s,X_s,\upsilon_s)\cdot \partial_x v(s,X_s,w^-(s,X_s))   \Big) \d s \\
		&\quad - \int_{\theta_\smallfont{1}}^{\theta_\smallfont{2}}  \partial_y v(s,X_s,w^-(s,X_s))  \Big(   \partial_t w^-(s,X_s)  - c(s,X_s,\upsilon_s) \Big) \d s \\
		&\quad- \frac12 \int_{\theta_\smallfont{1}}^{\theta_\smallfont{2}}  \bigg(  {\rm Tr}[\partial_{xx}^2\bar v(s, X_s)\d \langle X\rangle_s]-  \partial_y v(s,X_s,w^-(s,X_s)) {\rm Tr} [ \partial^2_{x x} w^-(s,X_s)\d \langle X\rangle_s ] \bigg)  \\
		&\quad-\int_{\theta_\smallfont{1}}^{\theta_\smallfont{2}} \partial_y v(s,X_s,w^-(s,X_s)) h^{\rm b} (s,X_s,   \partial_{x } w^-(s,X_s)  ,  \partial^2_{x x} w^-(s,X_s)  , \upsilon_s) \d s \\
		&\quad-\int_{\theta_\smallfont{1}}^{\theta_\smallfont{2}} \big( \partial_x  v(s,X_s,w^-(s,X_s)) + \partial_y v(s,X_s,w^-(s,X_s))  \partial_x w^-(s,X_s)\big)\cdot    \sigma(s,X_s,\upsilon_s) \d W_s,
	\end{align*}
	where in the first equality, we computed the time and space derivatives of $\bar v$, the dynamics of $X$, and in the second equality, we added and subtracted $\partial_y v   \big( c + \frac{1}2 {\rm Tr}[\sigma\sigma^\t  \partial^2_{xx}w^- ]\big)$ and use the fact that $Z_\cdot=\partial_x w^-(\cdot, X_\cdot)$ to complete the term $h^{\rm b}$ in the third line.
	\smallskip
	
	Recalling that $\upsilon$ attains the infimum in \eqref{eq:h_b}, the term $  \partial_t w^- +h^{\rm b}$ equals $0$.
	Moreover, since $Z_\cdot=\partial_x w^-(\cdot, X_\cdot)$, we have
	\begin{align*}
		{\rm Tr} \big[ \partial^2_{\rm x x} v(t,{\rm X}_t)\d \langle {\rm X}\rangle_t \big]={\rm Tr} \big[ \partial^2_{ x x} \bar  v(t,{X}_t)\d \langle {X}\rangle_t \big]-\partial_y  v(t,X_t,w^-(t,{ X}_t)) {\rm Tr}\big[ \partial^2_{ x x} w^-(t,{ X}_t)\d \langle {X}\rangle_t \big], \d t\otimes \d\P\ae
	\end{align*}
	Consequently,
	\begin{align}\label{eq.verificationi2}
		I_2&=- \int_{\theta_\smallfont{1}}^{\theta_\smallfont{2}}\Big( \partial_t v(s,{\rm X}_s)  + {\rm h}^{\rm L} \big(s, {\rm X}_s, \partial_{\rm x}  v(s, {\rm X}_s  ), \partial_{\rm xx}^2   v(s, {\rm X}_s) , \upsilon_s\big)  \Big) \d s \nonumber\\
		&\quad +\int_{\theta_\smallfont{1}}^{\theta_\smallfont{2}} C(s,X_s,\upsilon_s) \d s-\int_{\theta_1}^{\theta_2} \Big( \partial_x  v(s,{\rm X}_s) + \partial_y v(s,{\rm X}_s)  \partial_x w^-(s,X_s)\Big)\cdot    \sigma(s,X_s,\upsilon_s) \d W_s \nonumber  \\
		&= \int_{\theta_\smallfont{1}}^{\theta_\smallfont{2}}  \Big({\rm H}^{-,w^-}\! \big(s, {\rm X}_s    , \partial_{\rm x}  v(s, {\rm X}_s  ), \partial_{\rm xx}^2   v(s, {\rm X} _s ) \big)- {\rm h}^{\rm L} \big(s, {\rm X}_s , \partial_{\rm x}  v(s, {\rm X}_s), \partial_{\rm xx}^2   v(s, {\rm X}_s) , \upsilon_s\big) \Big)\d s \nonumber  \\
		&\quad +\int_{\theta_\smallfont{1}}^{\theta_\smallfont{2}} C(s,X_s,\upsilon_s) \d s-\int_{\theta_1}^{\theta_2} \Big( \partial_x  v(s,{\rm X}_s) + \partial_y v(s,{\rm X}_s)  \partial_x w^-(s,X_s)\Big)\cdot    \sigma(s,X_s,\upsilon_s) \d W_s\nonumber  \\
		& \geq \int_{\theta_\smallfont{1}}^{\theta_\smallfont{2}} C(s,X_s,\upsilon_s) \d s-\int_{\theta_1}^{\theta_2} \Big( \partial_x  v(s,{\rm X}_s) + \partial_y v(s,{\rm X}_s)  \partial_x w^-(s,X_s)\Big)\cdot    \sigma(s,X_s,\upsilon_s) \d W_s,
	\end{align}
	where in the first equality we added and subtracted $C$ to complete the expression for $h^{\rm L}$, and in the second equality, we used the fact that $v$ satisfies the second equation in \eqref{eq.fullHJB} for ${\rm X}_\cdot=(X_\cdot,w^-(\cdot,X_\cdot))$. The inequality follows from the definition of $H^{-,w^-}$ and the fact that $\upsilon_\cdot \in N(\cdot,X_\cdot, \partial_x w^-(\cdot,X_\cdot)),\d t\otimes \d \P\ae$
	\smallskip	
	
	We now notice that for arbitrary feasible $\upsilon$, in general, there exists a sequence $(\theta_n)_{n\in\N^\smallfont{\star}}\subseteq \Tc_T$, $t\leq \theta_n\leq \theta_{n+1}$, $n\geq1$, $\theta_n\longrightarrow T$, $\P\as$ as $n$ goes to infinity, of the form discussed above.
	Without loss of generality, we can assume $(X,Y)$ is bounded on $[t,\theta_n]$, so that by continuity, the terms $v, w^+$, and their derivates are bounded on $[t,\theta_n]$.
	Thus, since $\sigma$ is bounded and $\|Z\|_{\H^{\raisebox{1pt}{\smallfont{p}}}(\F,\P)}^p < \infty$, the stochastic integrals in \eqref{eq.verificationi1} and \eqref{eq.verificationi2} are martingales. 
	It then follows from \eqref{eq.verificationi1} and \eqref{eq.verificationi2}
	\begin{align*}
		v(t,x,y)\geq \E^\P\bigg[ v(\theta_n,X_{\theta_n},Y_{\theta_n}) + \int_t^{\theta_n} C(s,X_s , \upsilon_s )\d s\bigg| \Fc_t \bigg].
	\end{align*}
	
	Thus, the uniform integrability of the family $\{ v(\theta_n,X_{\theta_n},Y_{\theta_n})\}_{n\in\N^\smallfont{\star}}$, the boundedness of $C$, together with an application of dominated convergence, gives
	\begin{align}\label{eq.ineq.veri}
		v(t,x,y)\geq \E^\P\bigg[ G(X_{T}) + \int_t^{T} C(s,X_s , \upsilon_s )\d s\bigg| \Fc_t \bigg],
	\end{align}
	where we used the boundary condition in time in \eqref{eq.fullHJB} and that $w^-(T^-,x)=g(x)$, see \eqref{eq:pde-sup}. The arbitrariness of $\upsilon$ gives $v(t,x,y)\geq V(t,x,y)$. To conclude, note that for $(Z^\star,\Gamma^\star,\alpha^\star,b^\star)$ as in the statement, the inequalities in \eqref{eq.verificationi1} and \eqref{eq.verificationi2} are tight. 
\end{proof}

\begin{funding}
	The second author is supported by ANID FONDECYT Iniciaci\'on 11240944.	The third author is supported in part by NSF grant DMS-2307736. The fourth author is supported in part by SNSF MINT project 205121\_219818 and a 2022 Consolidation Grant from the Leading House for the Latin American Region.
\end{funding}

\bibliographystyle{imsart-nameyear} 
\bibliography{biblioAAP2254}       

\begin{thebibliography}{87}

\bibitem[\protect\citeauthoryear{A{\"\i}d, Basei and
  Pham}{2020}]{aid2020mckean}
\begin{barticle}[author]
\bauthor{\bsnm{A{\"\i}d},~\bfnm{R.}\binits{R.}},
  \bauthor{\bsnm{Basei},~\bfnm{M.}\binits{M.}} \AND
  \bauthor{\bsnm{Pham},~\bfnm{H.}\binits{H.}}
(\byear{2020}).
\btitle{A {M}c{K}ean--{V}lasov approach to distributed electricity generation
  development}.
\bjournal{Mathematical Methods of Operations Research}
\bvolume{91}
\bpages{269--310}.
\end{barticle}
\endbibitem

\bibitem[\protect\citeauthoryear{Ba{\c s}ar}{1979}]{basar1979stochastic}
\begin{binproceedings}[author]
\bauthor{\bsnm{Ba{\c s}ar},~\bfnm{T.}\binits{T.}}
(\byear{1979}).
\btitle{Stochastic stagewise {S}tackelberg strategies for linear quadratic
  systems}.
In \bbooktitle{Stochastic control theory and stochastic differential systems.
  Proceedings of a workshop of the {„}Sonderforschungsbereich 72 der
  Deutschen Forschungsgemeinschaft an der Universit{\"a}t Bonn{''} which took
  place in January 1979 at Bad Honnef}
(\beditor{\bfnm{M.}\binits{M.}~\bsnm{Kohlmann}} \AND
  \beditor{\bfnm{W.}\binits{W.}~\bsnm{Vogel}}, eds.).
\bseries{Lecture notes in control and information sciences}
\bvolume{16}
\bpages{264--276}.
\bpublisher{Springer Berlin, Heidelberg}.
\end{binproceedings}
\endbibitem

\bibitem[\protect\citeauthoryear{Ba{\c s}ar}{1981}]{basar1981new}
\begin{barticle}[author]
\bauthor{\bsnm{Ba{\c s}ar},~\bfnm{T.}\binits{T.}}
(\byear{1981}).
\btitle{A new method for the {S}tackelberg solution of differential games with
  sampled-data state information}.
\bjournal{IFAC Proceedings Volumes}
\bvolume{14}
\bpages{1365--1370}.
\end{barticle}
\endbibitem

\bibitem[\protect\citeauthoryear{Ba{\c s}ar and
  Haurie}{1984}]{basar1984feedback}
\begin{bincollection}[author]
\bauthor{\bsnm{Ba{\c s}ar},~\bfnm{T.}\binits{T.}} \AND
  \bauthor{\bsnm{Haurie},~\bfnm{A.}\binits{A.}}
(\byear{1984}).
\btitle{Feedback equilibria in differential games with structural and modal
  uncertainties}.
In \bbooktitle{Advances in large scale systems},
(\beditor{\bfnm{J.~B.}\binits{J.~B.}~\bsnm{Cruz~Jr.}}, ed.)
\bvolume{1}
\bpages{163--201}.
\end{bincollection}
\endbibitem

\bibitem[\protect\citeauthoryear{Ba{\c s}ar and
  Olsder}{1999}]{basar1999dynamic}
\begin{bbook}[author]
\bauthor{\bsnm{Ba{\c s}ar},~\bfnm{T.}\binits{T.}} \AND
  \bauthor{\bsnm{Olsder},~\bfnm{G.~J.}\binits{G.~J.}}
(\byear{1999}).
\btitle{Dynamic noncooperative game theory},
\bedition{2nd revised} ed.
\bpublisher{SIAM}.
\end{bbook}
\endbibitem

\bibitem[\protect\citeauthoryear{Ba{\c s}ar and Selbuz}{1978}]{basar1978new}
\begin{binproceedings}[author]
\bauthor{\bsnm{Ba{\c s}ar},~\bfnm{T.}\binits{T.}} \AND
  \bauthor{\bsnm{Selbuz},~\bfnm{H.}\binits{H.}}
(\byear{1978}).
\btitle{A new approach for derivation of closed-loop {S}tackelberg strategies}.
In \bbooktitle{1978 IEEE conference on decision and control including the 17th
  symposium on adaptive processes}
(\beditor{\bfnm{R.~E.}\binits{R.~E.}~\bsnm{Larson}} \AND
  \beditor{\bfnm{A.~S.}\binits{A.~S.}~\bsnm{Willsky}}, eds.)
\bpages{1113--1118}.
\end{binproceedings}
\endbibitem

\bibitem[\protect\citeauthoryear{Ba{\c s}ar and
  Selbuz}{1979}]{basar1979closedloop}
\begin{barticle}[author]
\bauthor{\bsnm{Ba{\c s}ar},~\bfnm{T.}\binits{T.}} \AND
  \bauthor{\bsnm{Selbuz},~\bfnm{H.}\binits{H.}}
(\byear{1979}).
\btitle{Closed-loop {S}tackelberg strategies with applications in the optimal
  control of multilevel systems}.
\bjournal{IEEE Transactions on Automatic Control}
\bvolume{24}
\bpages{166--179}.
\end{barticle}
\endbibitem

\bibitem[\protect\citeauthoryear{Ba{\c{s}}ar and
  Olsder}{1980}]{basar1980teamoptimal}
\begin{barticle}[author]
\bauthor{\bsnm{Ba{\c{s}}ar},~\bfnm{T.}\binits{T.}} \AND
  \bauthor{\bsnm{Olsder},~\bfnm{G.~J.}\binits{G.~J.}}
(\byear{1980}).
\btitle{Team-optimal closed-loop {S}tackelberg strategies in hierarchical
  control problems}.
\bjournal{Automatica}
\bvolume{16}
\bpages{409--414}.
\end{barticle}
\endbibitem

\bibitem[\protect\citeauthoryear{Bagchi}{1984}]{bagchi1984stackelberg}
\begin{bbook}[author]
\bauthor{\bsnm{Bagchi},~\bfnm{A.}\binits{A.}}
(\byear{1984}).
\btitle{Stackelberg differential games in economic models}.
\bseries{Lecture notes in control and information sciences}
\bvolume{64}.
\bpublisher{Springer Berlin, Heidelberg}.
\end{bbook}
\endbibitem

\bibitem[\protect\citeauthoryear{Bensoussan, Chen and
  Sethi}{2014}]{bensoussan2014feedback}
\begin{bincollection}[author]
\bauthor{\bsnm{Bensoussan},~\bfnm{A.}\binits{A.}},
  \bauthor{\bsnm{Chen},~\bfnm{S.}\binits{S.}} \AND
  \bauthor{\bsnm{Sethi},~\bfnm{S.~P.}\binits{S.~P.}}
(\byear{2014}).
\btitle{Feedback {S}tackelberg solutions of infinite-horizon stochastic
  differential ames}.
In \bbooktitle{Models and methods in economics and management science: essays
  in honor of Charles S. Tapiero},
(\beditor{\bfnm{F.}\binits{F.}~\bsnm{El~Ouardighi}} \AND
  \beditor{\bfnm{K.}\binits{K.}~\bsnm{Kogan}}, eds.).
\bseries{International series in operations research {\&} management science}
\bvolume{198}
\bpages{3--15}.
\bpublisher{Springer Cham}.
\end{bincollection}
\endbibitem

\bibitem[\protect\citeauthoryear{Bensoussan, Chen and
  Sethi}{2015}]{bensoussan2015maximum}
\begin{barticle}[author]
\bauthor{\bsnm{Bensoussan},~\bfnm{A.}\binits{A.}},
  \bauthor{\bsnm{Chen},~\bfnm{S.}\binits{S.}} \AND
  \bauthor{\bsnm{Sethi},~\bfnm{S.~P.}\binits{S.~P.}}
(\byear{2015}).
\btitle{The maximum principle for global solutions of stochastic {S}tackelberg
  differential games}.
\bjournal{SIAM Journal on Control and Optimization}
\bvolume{53}
\bpages{1956--1981}.
\end{barticle}
\endbibitem

\bibitem[\protect\citeauthoryear{Bensoussan
  et~al.}{2019}]{bensoussan2019feedback}
\begin{barticle}[author]
\bauthor{\bsnm{Bensoussan},~\bfnm{A.}\binits{A.}},
  \bauthor{\bsnm{Chen},~\bfnm{S.}\binits{S.}},
  \bauthor{\bsnm{Chutani},~\bfnm{A.}\binits{A.}},
  \bauthor{\bsnm{Sethi},~\bfnm{S.~P.}\binits{S.~P.}},
  \bauthor{\bsnm{Siu},~\bfnm{C.~C.}\binits{C.~C.}} \AND
  \bauthor{\bsnm{Yam},~\bfnm{S.~C.~P.}\binits{S.~C.~P.}}
(\byear{2019}).
\btitle{Feedback {S}tackelberg--{N}ash equilibria in mixed leadership games
  with an application to cooperative advertising}.
\bjournal{SIAM Journal on Control and Optimization}
\bvolume{57}
\bpages{3413--3444}.
\end{barticle}
\endbibitem

\bibitem[\protect\citeauthoryear{Bouchard, \'Elie and
  Touzi}{2009}]{bouchard2009stochastic}
\begin{barticle}[author]
\bauthor{\bsnm{Bouchard},~\bfnm{B.}\binits{B.}},
  \bauthor{\bsnm{\'Elie},~\bfnm{R.}\binits{R.}} \AND
  \bauthor{\bsnm{Touzi},~\bfnm{N.}\binits{N.}}
(\byear{2009}).
\btitle{Stochastic target problems with controlled loss}.
\bjournal{SIAM Journal on Control and Optimization}
\bvolume{48}
\bpages{3123--3150}.
\end{barticle}
\endbibitem

\bibitem[\protect\citeauthoryear{Bouchard, \'Elie and
  Imbert}{2010}]{bouchard2010optimal}
\begin{barticle}[author]
\bauthor{\bsnm{Bouchard},~\bfnm{B.}\binits{B.}},
  \bauthor{\bsnm{\'Elie},~\bfnm{R.}\binits{R.}} \AND
  \bauthor{\bsnm{Imbert},~\bfnm{C.}\binits{C.}}
(\byear{2010}).
\btitle{Optimal control under stochastic target constraints}.
\bjournal{SIAM Journal on Control and Optimization}
\bvolume{48}
\bpages{3501--3531}.
\end{barticle}
\endbibitem

\bibitem[\protect\citeauthoryear{Bressan}{2011}]{bressan2011noncooperative}
\begin{barticle}[author]
\bauthor{\bsnm{Bressan},~\bfnm{A.}\binits{A.}}
(\byear{2011}).
\btitle{Noncooperative differential games}.
\bjournal{Milan Journal of Mathematics}
\bvolume{79}
\bpages{357--427}.
\end{barticle}
\endbibitem

\bibitem[\protect\citeauthoryear{Carmona}{2016}]{carmona2016lectures}
\begin{bbook}[author]
\bauthor{\bsnm{Carmona},~\bfnm{R.}\binits{R.}}
(\byear{2016}).
\btitle{Lectures on {BSDEs}, stochastic control, and stochastic differential
  games with financial applications}.
\bseries{Financial mathematics}
\bvolume{1}.
\bpublisher{SIAM}.
\end{bbook}
\endbibitem

\bibitem[\protect\citeauthoryear{Castanon}{1976}]{castanon1976equilibria}
\begin{bphdthesis}[author]
\bauthor{\bsnm{Castanon},~\bfnm{D.~A.}\binits{D.~A.}}
(\byear{1976}).
\btitle{Equilibria in stochastic dynamic games of {S}tackelberg type},
\btype{PhD thesis},
\bpublisher{Massachusetts Institute of Technology}.
\end{bphdthesis}
\endbibitem

\bibitem[\protect\citeauthoryear{Chen and Cruz~Jr.}{1972}]{chen1972stackelberg}
\begin{barticle}[author]
\bauthor{\bsnm{Chen},~\bfnm{C.~I.}\binits{C.~I.}} \AND
  \bauthor{\bsnm{Cruz~Jr.},~\bfnm{J.~B.}\binits{J.~B.}}
(\byear{1972}).
\btitle{Stackelberg solution for two-person games with biased information
  patterns}.
\bjournal{IEEE Transactions on Automatic Control}
\bvolume{17}
\bpages{791--798}.
\end{barticle}
\endbibitem

\bibitem[\protect\citeauthoryear{Chiusolo and Hubert}{2024}]{chiusolo2024new}
\begin{barticle}[author]
\bauthor{\bsnm{Chiusolo},~\bfnm{A.}\binits{A.}} \AND
  \bauthor{\bsnm{Hubert},~\bfnm{E.}\binits{E.}}
(\byear{2024}).
\btitle{A new approach to principal--agent problems with volatility control}.
\bjournal{ArXiv preprint arXiv:2407.09471}.
\end{barticle}
\endbibitem

\bibitem[\protect\citeauthoryear{Chutani and Sethi}{2014}]{chutani2014feedback}
\begin{bincollection}[author]
\bauthor{\bsnm{Chutani},~\bfnm{A.}\binits{A.}} \AND
  \bauthor{\bsnm{Sethi},~\bfnm{S.~P.}\binits{S.~P.}}
(\byear{2014}).
\btitle{A feedback {S}tackelberg game of cooperative advertising in a durable
  goods oligopoly}.
In \bbooktitle{Dynamic games in economics},
(\beditor{\bfnm{J.}\binits{J.}~\bsnm{Haunschmied}},
  \beditor{\bfnm{V.~M.}\binits{V.~M.}~\bsnm{Veliov}} \AND
  \beditor{\bfnm{S.}\binits{S.}~\bsnm{Wrzaczek}}, eds.).
\bseries{Dynamic modeling and econometrics in economics and finance}
\bvolume{16}
\bpages{89--114}.
\bpublisher{Springer}.
\end{bincollection}
\endbibitem

\bibitem[\protect\citeauthoryear{Cong and Shi}{2024}]{cong2024direct}
\begin{barticle}[author]
\bauthor{\bsnm{Cong},~\bfnm{W.}\binits{W.}} \AND
  \bauthor{\bsnm{Shi},~\bfnm{J.}\binits{J.}}
(\byear{2024}).
\btitle{Direct approach of linear--quadratic {S}tackelberg mean field games of
  backward--forward stochastic systems}.
\bjournal{ArXiv preprint arXiv:2401.15835}.
\end{barticle}
\endbibitem

\bibitem[\protect\citeauthoryear{Crandall, Ishii and
  Lions}{1992}]{crandall1992user}
\begin{barticle}[author]
\bauthor{\bsnm{Crandall},~\bfnm{M.~G.}\binits{M.~G.}},
  \bauthor{\bsnm{Ishii},~\bfnm{H.}\binits{H.}} \AND
  \bauthor{\bsnm{Lions},~\bfnm{P.~L.}\binits{P.~L.}}
(\byear{1992}).
\btitle{User's guide to viscosity solutions of second order partial
  differential equations}.
\bjournal{Bulletin of the American Mathematical Society}
\bvolume{27}
\bpages{1--67}.
\end{barticle}
\endbibitem

\bibitem[\protect\citeauthoryear{Cruz~Jr.}{1975}]{cruz1975survey}
\begin{bincollection}[author]
\bauthor{\bsnm{Cruz~Jr.},~\bfnm{J.~B.}\binits{J.~B.}}
(\byear{1975}).
\btitle{Survey of {N}ash and {S}tackelberg equilibrium strategies in dynamic
  games}.
In \bbooktitle{Annals of economic and social measurement},
\bvolume{4}
\bpages{339--344}.
\bpublisher{National Bureau of Economic Research}.
\end{bincollection}
\endbibitem

\bibitem[\protect\citeauthoryear{Cruz~Jr.}{1976}]{cruz1976stackelberg}
\begin{bincollection}[author]
\bauthor{\bsnm{Cruz~Jr.},~\bfnm{J.~B.}\binits{J.~B.}}
(\byear{1976}).
\btitle{Stackelberg strategies for multilevel systems}.
In \bbooktitle{Directions in large-scale systems}
(\beditor{\bfnm{Y.~C.}\binits{Y.~C.}~\bsnm{Ho}} \AND
  \beditor{\bfnm{S.~K.}\binits{S.~K.}~\bsnm{Mitter}}, eds.)
\bpages{139--147}.
\bpublisher{Springer New York, NY}.
\end{bincollection}
\endbibitem

\bibitem[\protect\citeauthoryear{Cvitani{\'c}, Possama{\"\i} and
  Touzi}{2018}]{cvitanic2018dynamic}
\begin{barticle}[author]
\bauthor{\bsnm{Cvitani{\'c}},~\bfnm{J.}\binits{J.}},
  \bauthor{\bsnm{Possama{\"\i}},~\bfnm{D.}\binits{D.}} \AND
  \bauthor{\bsnm{Touzi},~\bfnm{N.}\binits{N.}}
(\byear{2018}).
\btitle{Dynamic programming approach to principal--agent problems}.
\bjournal{Finance and Stochastics}
\bvolume{22}
\bpages{1--37}.
\end{barticle}
\endbibitem

\bibitem[\protect\citeauthoryear{Cvitani{\'c} and
  Zhang}{2012}]{cvitanic2012contract}
\begin{bbook}[author]
\bauthor{\bsnm{Cvitani{\'c}},~\bfnm{J.}\binits{J.}} \AND
  \bauthor{\bsnm{Zhang},~\bfnm{J.}\binits{J.}}
(\byear{2012}).
\btitle{Contract theory in continuous-time models}.
\bpublisher{Springer}.
\end{bbook}
\endbibitem

\bibitem[\protect\citeauthoryear{Dayanıklı and
  Lauri{\`e}re}{2023}]{dayanikli2023machine}
\begin{barticle}[author]
\bauthor{\bsnm{Dayanıklı},~\bfnm{G.}\binits{G.}} \AND
  \bauthor{\bsnm{Lauri{\`e}re},~\bfnm{M.}\binits{M.}}
(\byear{2023}).
\btitle{A machine learning method for {S}tackelberg mean field games}.
\bjournal{ArXiv preprint arXiv:2302.10440}.
\end{barticle}
\endbibitem

\bibitem[\protect\citeauthoryear{Dockner
  et~al.}{2000}]{dockner2000differential}
\begin{bbook}[author]
\bauthor{\bsnm{Dockner},~\bfnm{E.~J.}\binits{E.~J.}},
  \bauthor{\bsnm{Jorgensen},~\bfnm{S.}\binits{S.}},
  \bauthor{\bsnm{Van~Long},~\bfnm{N.}\binits{N.}} \AND
  \bauthor{\bsnm{Sorger},~\bfnm{G.}\binits{G.}}
(\byear{2000}).
\btitle{Differential games in economics and management science}.
\bpublisher{Cambridge University Press}.
\end{bbook}
\endbibitem

\bibitem[\protect\citeauthoryear{El~Karoui and
  Tan}{2013}]{karoui2013capacities2}
\begin{btechreport}[author]
\bauthor{\bsnm{El~Karoui},~\bfnm{N.}\binits{N.}} \AND
  \bauthor{\bsnm{Tan},~\bfnm{X.}\binits{X.}}
(\byear{2013}).
\btitle{Capacities, measurable selection and dynamic programming part {II}:
  application in stochastic control problems}
\btype{Technical Report},
\bpublisher{\'Ecole Polytechnique and universit{\'e} Paris-Dauphine}.
\end{btechreport}
\endbibitem

\bibitem[\protect\citeauthoryear{Feng, Hu and Huang}{2022}]{feng2022backward}
\begin{barticle}[author]
\bauthor{\bsnm{Feng},~\bfnm{X.}\binits{X.}},
  \bauthor{\bsnm{Hu},~\bfnm{Y.}\binits{Y.}} \AND
  \bauthor{\bsnm{Huang},~\bfnm{J.}\binits{J.}}
(\byear{2022}).
\btitle{Backward {S}tackelberg differential game with constraints: a mixed
  terminal-perturbation and linear--quadratic approach}.
\bjournal{SIAM Journal on Control and Optimization}
\bvolume{60}
\bpages{1488--1518}.
\end{barticle}
\endbibitem

\bibitem[\protect\citeauthoryear{Fu and Horst}{2020}]{fu2020mean}
\begin{barticle}[author]
\bauthor{\bsnm{Fu},~\bfnm{G.}\binits{G.}} \AND
  \bauthor{\bsnm{Horst},~\bfnm{U.}\binits{U.}}
(\byear{2020}).
\btitle{Mean-field leader--follower games with terminal state constraint}.
\bjournal{SIAM Journal on Control and Optimization}
\bvolume{58}
\bpages{2078--2113}.
\end{barticle}
\endbibitem

\bibitem[\protect\citeauthoryear{Gardner and
  Cruz~Jr.}{1977}]{gardner1977feedback}
\begin{barticle}[author]
\bauthor{\bsnm{Gardner},~\bfnm{B}\binits{B.}} \AND
  \bauthor{\bsnm{Cruz~Jr.},~\bfnm{J.~B.}\binits{J.~B.}}
(\byear{1977}).
\btitle{Feedback {S}tackelberg strategy for a two player game}.
\bjournal{IEEE Transactions on Automatic Control}
\bvolume{22}
\bpages{270--271}.
\end{barticle}
\endbibitem

\bibitem[\protect\citeauthoryear{Gou, Huang and Wang}{2023}]{gou2023linear}
\begin{barticle}[author]
\bauthor{\bsnm{Gou},~\bfnm{Z.}\binits{Z.}},
  \bauthor{\bsnm{Huang},~\bfnm{N.~J.}\binits{N.~J.}} \AND
  \bauthor{\bsnm{Wang},~\bfnm{M.~H.}\binits{M.~H.}}
(\byear{2023}).
\btitle{A linear--quadratic mean-field stochastic {S}tackelberg differential
  game with random exit time}.
\bjournal{International Journal of Control}
\bvolume{96}
\bpages{731--745}.
\end{barticle}
\endbibitem

\bibitem[\protect\citeauthoryear{Guan, Liang and
  Song}{2024}]{guan2024stackelberg}
\begin{barticle}[author]
\bauthor{\bsnm{Guan},~\bfnm{G.}\binits{G.}},
  \bauthor{\bsnm{Liang},~\bfnm{Z.}\binits{Z.}} \AND
  \bauthor{\bsnm{Song},~\bfnm{Y.}\binits{Y.}}
(\byear{2024}).
\btitle{A {S}tackelberg reinsurance--investment game under $\alpha$-maxmin
  mean--variance criterion and stochastic volatility}.
\bjournal{Scandinavian Actuarial Journal}
\bvolume{2024}
\bpages{28--63}.
\end{barticle}
\endbibitem

\bibitem[\protect\citeauthoryear{Han, Landriault and Li}{2024}]{han2024optimal}
\begin{barticle}[author]
\bauthor{\bsnm{Han},~\bfnm{X.}\binits{X.}},
  \bauthor{\bsnm{Landriault},~\bfnm{D.}\binits{D.}} \AND
  \bauthor{\bsnm{Li},~\bfnm{D.}\binits{D.}}
(\byear{2024}).
\btitle{Optimal reinsurance contract in a {S}tackelberg game framework: a view
  of social planner}.
\bjournal{Scandinavian Actuarial Journal}
\bvolume{2024}
\bpages{124--148}.
\end{barticle}
\endbibitem

\bibitem[\protect\citeauthoryear{Havrylenko, Hinken and
  Zagst}{2022}]{havrylenko2022risk}
\begin{barticle}[author]
\bauthor{\bsnm{Havrylenko},~\bfnm{Y.}\binits{Y.}},
  \bauthor{\bsnm{Hinken},~\bfnm{M.}\binits{M.}} \AND
  \bauthor{\bsnm{Zagst},~\bfnm{R.}\binits{R.}}
(\byear{2022}).
\btitle{Risk sharing in equity-linked insurance products: {S}tackelberg
  equilibrium between an insurer and a reinsurer}.
\bjournal{ArXiv preprint arXiv:2203.04053}.
\end{barticle}
\endbibitem

\bibitem[\protect\citeauthoryear{He, Prasad and
  Sethi}{2008}]{he2008cooperative}
\begin{binproceedings}[author]
\bauthor{\bsnm{He},~\bfnm{X.}\binits{X.}},
  \bauthor{\bsnm{Prasad},~\bfnm{A.}\binits{A.}} \AND
  \bauthor{\bsnm{Sethi},~\bfnm{S.~P.}\binits{S.~P.}}
(\byear{2008}).
\btitle{Cooperative advertising and pricing in a dynamic stochastic supply
  chain: feedback {S}tackelberg strategies}.
In \bbooktitle{PICMET '08, Portland international conference on management of
  engineering \& technology}
(\beditor{\bfnm{D.~F.}\binits{D.~F.}~\bsnm{Kocaoglu}},
  \beditor{\bfnm{T.~R.}\binits{T.~R.}~\bsnm{Anderson}} \AND
  \beditor{\bfnm{T.~U.}\binits{T.~U.}~\bsnm{Daim}}, eds.)
\bpages{1634--1649}.
\end{binproceedings}
\endbibitem

\bibitem[\protect\citeauthoryear{He et~al.}{2007}]{he2007survey}
\begin{barticle}[author]
\bauthor{\bsnm{He},~\bfnm{X.}\binits{X.}},
  \bauthor{\bsnm{Prasad},~\bfnm{A.}\binits{A.}},
  \bauthor{\bsnm{Sethi},~\bfnm{S.~P.}\binits{S.~P.}} \AND
  \bauthor{\bsnm{Gutierrez},~\bfnm{G.~J.}\binits{G.~J.}}
(\byear{2007}).
\btitle{A survey of {S}tackelberg differential game models in supply and
  marketing channels}.
\bjournal{Journal of Systems Science and Systems Engineering}
\bvolume{16}
\bpages{385--413}.
\end{barticle}
\endbibitem

\bibitem[\protect\citeauthoryear{Huang and Shi}{2021}]{huang2021verification}
\begin{barticle}[author]
\bauthor{\bsnm{Huang},~\bfnm{Q.}\binits{Q.}} \AND
  \bauthor{\bsnm{Shi},~\bfnm{J.}\binits{J.}}
(\byear{2021}).
\btitle{A verification theorem for {S}tackelberg stochastic differential games
  in feedback information pattern}.
\bjournal{ArXiv preprint arXiv:2108.06498}.
\end{barticle}
\endbibitem

\bibitem[\protect\citeauthoryear{Kang and Shi}{2022}]{kang2022three}
\begin{barticle}[author]
\bauthor{\bsnm{Kang},~\bfnm{K.}\binits{K.}} \AND
  \bauthor{\bsnm{Shi},~\bfnm{J.}\binits{J.}}
(\byear{2022}).
\btitle{A three-level stochastic linear--quadratic {S}tackelberg differential
  game with asymmetric information}.
\bjournal{ArXiv preprint arXiv:2210.11808}.
\end{barticle}
\endbibitem

\bibitem[\protect\citeauthoryear{Karandikar}{1995}]{karandikar1995pathwise}
\begin{barticle}[author]
\bauthor{\bsnm{Karandikar},~\bfnm{R.~L.}\binits{R.~L.}}
(\byear{1995}).
\btitle{On pathwise stochastic integration}.
\bjournal{Stochastic Processes and their Applications}
\bvolume{57}
\bpages{11--18}.
\end{barticle}
\endbibitem

\bibitem[\protect\citeauthoryear{Leitmann}{1978}]{leitmann1978generalized}
\begin{barticle}[author]
\bauthor{\bsnm{Leitmann},~\bfnm{G.}\binits{G.}}
(\byear{1978}).
\btitle{On generalized {S}tackelberg strategies}.
\bjournal{Journal of Optimization Theory and Applications}
\bvolume{26}
\bpages{637--643}.
\end{barticle}
\endbibitem

\bibitem[\protect\citeauthoryear{Li and Han}{2023}]{li2023solving}
\begin{barticle}[author]
\bauthor{\bsnm{Li},~\bfnm{Y.}\binits{Y.}} \AND
  \bauthor{\bsnm{Han},~\bfnm{S.}\binits{S.}}
(\byear{2023}).
\btitle{Solving strongly convex and smooth {S}tackelberg games without modeling
  the follower}.
\bjournal{ArXiv preprint arXiv:2303.06192}.
\end{barticle}
\endbibitem

\bibitem[\protect\citeauthoryear{Li and Sethi}{2017}]{li2017review}
\begin{barticle}[author]
\bauthor{\bsnm{Li},~\bfnm{T.}\binits{T.}} \AND
  \bauthor{\bsnm{Sethi},~\bfnm{S.~P.}\binits{S.~P.}}
(\byear{2017}).
\btitle{A review of dynamic {S}tackelberg game models}.
\bjournal{Discrete \& Continuous Dynamical Systems--B}
\bvolume{22}
\bpages{125--129}.
\end{barticle}
\endbibitem

\bibitem[\protect\citeauthoryear{Li and Shi}{2023a}]{li2023closed}
\begin{barticle}[author]
\bauthor{\bsnm{Li},~\bfnm{Z.}\binits{Z.}} \AND
  \bauthor{\bsnm{Shi},~\bfnm{J.}\binits{J.}}
(\byear{2023}a).
\btitle{Closed-loop solvability of linear quadratic mean-field type
  {S}tackelberg stochastic differential games}.
\bjournal{ArXiv preprint arXiv:2303.07544}.
\end{barticle}
\endbibitem

\bibitem[\protect\citeauthoryear{Li and Shi}{2023b}]{li2023linear}
\begin{barticle}[author]
\bauthor{\bsnm{Li},~\bfnm{Z.}\binits{Z.}} \AND
  \bauthor{\bsnm{Shi},~\bfnm{J.}\binits{J.}}
(\byear{2023}b).
\btitle{Linear quadratic leader--follower stochastic differential games:
  closed-loop solvability}.
\bjournal{Journal of Systems Science and Complexity}
\bvolume{36}
\bpages{1373--1406}.
\end{barticle}
\endbibitem

\bibitem[\protect\citeauthoryear{Li, Xu and Zhang}{2022}]{li2022closed}
\begin{barticle}[author]
\bauthor{\bsnm{Li},~\bfnm{H.}\binits{H.}},
  \bauthor{\bsnm{Xu},~\bfnm{J.}\binits{J.}} \AND
  \bauthor{\bsnm{Zhang},~\bfnm{H.}\binits{H.}}
(\byear{2022}).
\btitle{Closed--loop {S}tackelberg strategy for linear--quadratic
  leader--follower game}.
\bjournal{ArXiv preprint arXiv:2212.08977}.
\end{barticle}
\endbibitem

\bibitem[\protect\citeauthoryear{Li and Yu}{2018}]{li2018forward}
\begin{barticle}[author]
\bauthor{\bsnm{Li},~\bfnm{N.}\binits{N.}} \AND
  \bauthor{\bsnm{Yu},~\bfnm{Z.}\binits{Z.}}
(\byear{2018}).
\btitle{Forward--backward stochastic differential equations and
  linear--quadratic generalized {S}tackelberg games}.
\bjournal{SIAM Journal on Control and Optimization}
\bvolume{56}
\bpages{4148--4180}.
\end{barticle}
\endbibitem

\bibitem[\protect\citeauthoryear{Liu et~al.}{2018}]{liu2018pessimistic}
\begin{barticle}[author]
\bauthor{\bsnm{Liu},~\bfnm{J.}\binits{J.}},
  \bauthor{\bsnm{Fan},~\bfnm{Y.}\binits{Y.}},
  \bauthor{\bsnm{Chen},~\bfnm{Z.}\binits{Z.}} \AND
  \bauthor{\bsnm{Zheng},~\bfnm{Y.}\binits{Y.}}
(\byear{2018}).
\btitle{Pessimistic bilevel optimization: a survey}.
\bjournal{International Journal of Computational Intelligence Systems}
\bvolume{11}
\bpages{725--736}.
\end{barticle}
\endbibitem

\bibitem[\protect\citeauthoryear{Lv, Xiong and Zhang}{2023}]{lv2023linear}
\begin{barticle}[author]
\bauthor{\bsnm{Lv},~\bfnm{S.}\binits{S.}},
  \bauthor{\bsnm{Xiong},~\bfnm{J.}\binits{J.}} \AND
  \bauthor{\bsnm{Zhang},~\bfnm{X.}\binits{X.}}
(\byear{2023}).
\btitle{Linear quadratic leader--follower stochastic differential games for
  mean-field switching diffusions}.
\bjournal{Automatica}
\bvolume{154}
\bpages{1--9}.
\end{barticle}
\endbibitem

\bibitem[\protect\citeauthoryear{Ma et~al.}{2015}]{ma2015well}
\begin{barticle}[author]
\bauthor{\bsnm{Ma},~\bfnm{J.}\binits{J.}},
  \bauthor{\bsnm{Wu},~\bfnm{Z.}\binits{Z.}},
  \bauthor{\bsnm{Zhang},~\bfnm{D.}\binits{D.}} \AND
  \bauthor{\bsnm{Zhang},~\bfnm{J.}\binits{J.}}
(\byear{2015}).
\btitle{On well-posedness of forward--backward {SDE}s---a unified approach}.
\bjournal{The Annals of Applied Probability}
\bvolume{25}
\bpages{2168--2214}.
\end{barticle}
\endbibitem

\bibitem[\protect\citeauthoryear{Mallozzi and Morgan}{1995}]{mallozzi1995weak}
\begin{barticle}[author]
\bauthor{\bsnm{Mallozzi},~\bfnm{L.}\binits{L.}} \AND
  \bauthor{\bsnm{Morgan},~\bfnm{J.}\binits{J.}}
(\byear{1995}).
\btitle{Weak {S}tackelberg problem and mixed solutions under data
  perturbations}.
\bjournal{Optimization}
\bvolume{32}
\bpages{269--290}.
\end{barticle}
\endbibitem

\bibitem[\protect\citeauthoryear{Moon}{2021}]{moon2021linear}
\begin{barticle}[author]
\bauthor{\bsnm{Moon},~\bfnm{J.}\binits{J.}}
(\byear{2021}).
\btitle{Linear--quadratic stochastic {S}tackelberg differential games for
  jump--diffusion systems}.
\bjournal{SIAM Journal on Control and Optimization}
\bvolume{59}
\bpages{954--976}.
\end{barticle}
\endbibitem

\bibitem[\protect\citeauthoryear{Ni, Liu and Zhang}{2023}]{ni2023deterministic}
\begin{barticle}[author]
\bauthor{\bsnm{Ni},~\bfnm{Y.~H.}\binits{Y.~H.}},
  \bauthor{\bsnm{Liu},~\bfnm{L.}\binits{L.}} \AND
  \bauthor{\bsnm{Zhang},~\bfnm{X.}\binits{X.}}
(\byear{2023}).
\btitle{Deterministic dynamic {S}tackelberg games: time-consistent open-loop
  solution}.
\bjournal{Automatica}
\bvolume{148}
\bpages{1--9}.
\end{barticle}
\endbibitem

\bibitem[\protect\citeauthoryear{Nutz}{2012}]{nutz2012pathwise}
\begin{barticle}[author]
\bauthor{\bsnm{Nutz},~\bfnm{M.}\binits{M.}}
(\byear{2012}).
\btitle{Pathwise construction of stochastic integrals}.
\bjournal{Electronic Communications in Probability}
\bvolume{17}
\bpages{1--7}.
\end{barticle}
\endbibitem

\bibitem[\protect\citeauthoryear{{\O}ksendal, Sandal and
  Ub{\o}e}{2013}]{oksendal2013stochastic}
\begin{barticle}[author]
\bauthor{\bsnm{{\O}ksendal},~\bfnm{B.}\binits{B.}},
  \bauthor{\bsnm{Sandal},~\bfnm{L.}\binits{L.}} \AND
  \bauthor{\bsnm{Ub{\o}e},~\bfnm{J.}\binits{J.}}
(\byear{2013}).
\btitle{Stochastic {S}tackelberg equilibria with applications to time-dependent
  newsvendor models}.
\bjournal{Journal of Economic Dynamics and Control}
\bvolume{37}
\bpages{1284--1299}.
\end{barticle}
\endbibitem

\bibitem[\protect\citeauthoryear{Papavassilopoulos}{1979}]{papavassilopoulos1979leader}
\begin{bphdthesis}[author]
\bauthor{\bsnm{Papavassilopoulos},~\bfnm{G.~P.}\binits{G.~P.}}
(\byear{1979}).
\btitle{Leader--follower and {N}ash strategies with state information},
\btype{PhD thesis},
\bpublisher{University of Illinois at Urbana-Champaign}.
\end{bphdthesis}
\endbibitem

\bibitem[\protect\citeauthoryear{Papavassilopoulos and
  Cruz~Jr.}{1979}]{papavassilopoulos1979nonclassical}
\begin{barticle}[author]
\bauthor{\bsnm{Papavassilopoulos},~\bfnm{G.~P.}\binits{G.~P.}} \AND
  \bauthor{\bsnm{Cruz~Jr.},~\bfnm{J.~B.}\binits{J.~B.}}
(\byear{1979}).
\btitle{Nonclassical control problems and {S}tackelberg games}.
\bjournal{IEEE Transactions on Automatic Control}
\bvolume{24}
\bpages{155--166}.
\end{barticle}
\endbibitem

\bibitem[\protect\citeauthoryear{Papavassilopoulos and
  Cruz~Jr.}{1980}]{papavassilopoulos1980sufficient}
\begin{barticle}[author]
\bauthor{\bsnm{Papavassilopoulos},~\bfnm{G.~P.}\binits{G.~P.}} \AND
  \bauthor{\bsnm{Cruz~Jr.},~\bfnm{J.~B.}\binits{J.~B.}}
(\byear{1980}).
\btitle{Sufficient conditions for {S}tackelberg and {N}ash strategies with
  memory}.
\bjournal{Journal of Optimization Theory and Applications}
\bvolume{31}
\bpages{233--260}.
\end{barticle}
\endbibitem

\bibitem[\protect\citeauthoryear{Possama{\"\i}, Tan and
  Zhou}{2018}]{possamai2018stochastic}
\begin{barticle}[author]
\bauthor{\bsnm{Possama{\"\i}},~\bfnm{D.}\binits{D.}},
  \bauthor{\bsnm{Tan},~\bfnm{X.}\binits{X.}} \AND
  \bauthor{\bsnm{Zhou},~\bfnm{C.}\binits{C.}}
(\byear{2018}).
\btitle{Stochastic control for a class of nonlinear kernels and applications}.
\bjournal{The Annals of Probability}
\bvolume{46}
\bpages{551--603}.
\end{barticle}
\endbibitem

\bibitem[\protect\citeauthoryear{Possama{\"\i}, Touzi and
  Zhang}{2020}]{possamai2020zero}
\begin{barticle}[author]
\bauthor{\bsnm{Possama{\"\i}},~\bfnm{D.}\binits{D.}},
  \bauthor{\bsnm{Touzi},~\bfnm{N.}\binits{N.}} \AND
  \bauthor{\bsnm{Zhang},~\bfnm{J.}\binits{J.}}
(\byear{2020}).
\btitle{Zero-sum path-dependent stochastic differential games in weak
  formulation}.
\bjournal{The Annals of Applied Probability}
\bvolume{30}
\bpages{1415--1457}.
\end{barticle}
\endbibitem

\bibitem[\protect\citeauthoryear{Shi, Wang and Xiong}{2016}]{shi2016leader}
\begin{barticle}[author]
\bauthor{\bsnm{Shi},~\bfnm{J.}\binits{J.}},
  \bauthor{\bsnm{Wang},~\bfnm{G.}\binits{G.}} \AND
  \bauthor{\bsnm{Xiong},~\bfnm{J.}\binits{J.}}
(\byear{2016}).
\btitle{Leader--follower stochastic differential game with asymmetric
  information and applications}.
\bjournal{Automatica}
\bvolume{63}
\bpages{60--73}.
\end{barticle}
\endbibitem

\bibitem[\protect\citeauthoryear{Si and Wu}{2021}]{si2021backward}
\begin{barticle}[author]
\bauthor{\bsnm{Si},~\bfnm{K.}\binits{K.}} \AND
  \bauthor{\bsnm{Wu},~\bfnm{Z.}\binits{Z.}}
(\byear{2021}).
\btitle{Backward--forward linear--quadratic mean-field {S}tackelberg games}.
\bjournal{Advances in Difference Equations}
\bvolume{2021}
\bpages{1--23}.
\end{barticle}
\endbibitem

\bibitem[\protect\citeauthoryear{Simaan and
  Cruz~Jr.}{1973a}]{simaan1973additional}
\begin{barticle}[author]
\bauthor{\bsnm{Simaan},~\bfnm{M.}\binits{M.}} \AND
  \bauthor{\bsnm{Cruz~Jr.},~\bfnm{J.~B.}\binits{J.~B.}}
(\byear{1973}a).
\btitle{Additional aspects of the {S}tackelberg strategy in nonzero-sum games}.
\bjournal{Journal of Optimization Theory and Applications}
\bvolume{11}
\bpages{613--626}.
\end{barticle}
\endbibitem

\bibitem[\protect\citeauthoryear{Simaan and
  Cruz~Jr.}{1973b}]{simaan1973stackelberg}
\begin{barticle}[author]
\bauthor{\bsnm{Simaan},~\bfnm{M.}\binits{M.}} \AND
  \bauthor{\bsnm{Cruz~Jr.},~\bfnm{J.~B.}\binits{J.~B.}}
(\byear{1973}b).
\btitle{On the {S}tackelberg strategy in nonzero-sum games}.
\bjournal{Journal of Optimization Theory and Applications}
\bvolume{11}
\bpages{533--555}.
\end{barticle}
\endbibitem

\bibitem[\protect\citeauthoryear{Simaan and
  Cruz~Jr.}{1976}]{simaan1976stackelberg}
\begin{bincollection}[author]
\bauthor{\bsnm{Simaan},~\bfnm{M.}\binits{M.}} \AND
  \bauthor{\bsnm{Cruz~Jr.},~\bfnm{J.~B.}\binits{J.~B.}}
(\byear{1976}).
\btitle{On the {S}tackelberg strategy in nonzero-sum games}.
In \bbooktitle{Multicriteria decision making and differential games},
(\beditor{\bfnm{G.}\binits{G.}~\bsnm{Leitmann}}, ed.).
\bseries{Mathematical concepts and methods in science and engineering}
\bpages{173--195}.
\bpublisher{Springer New York, NY}.
\end{bincollection}
\endbibitem

\bibitem[\protect\citeauthoryear{Soner and Touzi}{2002a}]{soner2002dynamic}
\begin{barticle}[author]
\bauthor{\bsnm{Soner},~\bfnm{H.~M.}\binits{H.~M.}} \AND
  \bauthor{\bsnm{Touzi},~\bfnm{N.}\binits{N.}}
(\byear{2002}a).
\btitle{Dynamic programming for stochastic target problems and geometric
  flows}.
\bjournal{Journal of the European Mathematical Society}
\bvolume{4}
\bpages{201--236}.
\end{barticle}
\endbibitem

\bibitem[\protect\citeauthoryear{Soner and Touzi}{2002b}]{soner2002stochastic}
\begin{barticle}[author]
\bauthor{\bsnm{Soner},~\bfnm{H.~M.}\binits{H.~M.}} \AND
  \bauthor{\bsnm{Touzi},~\bfnm{N.}\binits{N.}}
(\byear{2002}b).
\btitle{Stochastic target problems, dynamic programming, and viscosity
  solutions}.
\bjournal{SIAM Journal on Control and Optimization}
\bvolume{41}
\bpages{404--424}.
\end{barticle}
\endbibitem

\bibitem[\protect\citeauthoryear{Soner and Touzi}{2003}]{soner2003stochastic}
\begin{barticle}[author]
\bauthor{\bsnm{Soner},~\bfnm{H.~M.}\binits{H.~M.}} \AND
  \bauthor{\bsnm{Touzi},~\bfnm{N.}\binits{N.}}
(\byear{2003}).
\btitle{A stochastic representation for mean curvature type geometric flows}.
\bjournal{The Annals of Probability}
\bvolume{31}
\bpages{1145--1165}.
\end{barticle}
\endbibitem

\bibitem[\protect\citeauthoryear{Soner, Touzi and
  Zhang}{2012}]{soner2012wellposedness}
\begin{barticle}[author]
\bauthor{\bsnm{Soner},~\bfnm{H.~M.}\binits{H.~M.}},
  \bauthor{\bsnm{Touzi},~\bfnm{N.}\binits{N.}} \AND
  \bauthor{\bsnm{Zhang},~\bfnm{J.}\binits{J.}}
(\byear{2012}).
\btitle{Wellposedness of second order backward {SDE}s}.
\bjournal{Probability Theory and Related Fields}
\bvolume{153}
\bpages{149--190}.
\end{barticle}
\endbibitem

\bibitem[\protect\citeauthoryear{Stroock and
  Varadhan}{1997}]{stroock1997multidimensional}
\begin{bbook}[author]
\bauthor{\bsnm{Stroock},~\bfnm{D.~W.}\binits{D.~W.}} \AND
  \bauthor{\bsnm{Varadhan},~\bfnm{S.~R.~S.}\binits{S.~R.~S.}}
(\byear{1997}).
\btitle{Multidimensional diffusion processes}.
\bseries{Grundlehren der mathematischen Wissenschaften}
\bvolume{233}.
\bpublisher{Springer-Verlag Berlin Heidelberg}.
\end{bbook}
\endbibitem

\bibitem[\protect\citeauthoryear{Sun, Wang and Wen}{2023}]{sun2023zero}
\begin{barticle}[author]
\bauthor{\bsnm{Sun},~\bfnm{J.}\binits{J.}},
  \bauthor{\bsnm{Wang},~\bfnm{H.}\binits{H.}} \AND
  \bauthor{\bsnm{Wen},~\bfnm{J.}\binits{J.}}
(\byear{2023}).
\btitle{Zero-sum {S}tackelberg stochastic linear--quadratic differential
  games}.
\bjournal{SIAM Journal on Control and Optimization}
\bvolume{61}
\bpages{252--284}.
\end{barticle}
\endbibitem

\bibitem[\protect\citeauthoryear{Van~Long}{2010}]{van2010survey}
\begin{bbook}[author]
\bauthor{\bsnm{Van~Long},~\bfnm{N.}\binits{N.}}
(\byear{2010}).
\btitle{A survey of dynamic games in economics}.
\bseries{Surveys on theories in economics and business administration}
\bvolume{1}.
\bpublisher{World Scientific}.
\end{bbook}
\endbibitem

\bibitem[\protect\citeauthoryear{Vasal}{2022a}]{vasal2022sequential}
\begin{binproceedings}[author]
\bauthor{\bsnm{Vasal},~\bfnm{D.}\binits{D.}}
(\byear{2022}a).
\btitle{Sequential decomposition of stochastic {S}tackelberg games}.
In \bbooktitle{2022 American control conference}
(\beditor{\bfnm{B.}\binits{B.}~\bsnm{Ferri}} \AND
  \beditor{\bfnm{F.}\binits{F.}~\bsnm{Zhang}}, eds.)
\bpages{1266--1271}.
\bpublisher{IEEE}.
\end{binproceedings}
\endbibitem

\bibitem[\protect\citeauthoryear{Vasal}{2022b}]{vasal2022master}
\begin{barticle}[author]
\bauthor{\bsnm{Vasal},~\bfnm{D.}\binits{D.}}
(\byear{2022}b).
\btitle{Master equation of discrete-time {S}tackelberg mean field games with
  multiple leaders}.
\bjournal{ArXiv preprint arXiv:2209.03186}.
\end{barticle}
\endbibitem

\bibitem[\protect\citeauthoryear{von
  Stackelberg}{1934}]{stackelberg1934marktform}
\begin{bbook}[author]
\bauthor{\bparticle{von} \bsnm{Stackelberg},~\bfnm{H.}\binits{H.}}
(\byear{1934}).
\btitle{Marktform und Gleichgewicht}.
\bpublisher{Springer-Verlag Wien New York}.
\end{bbook}
\endbibitem

\bibitem[\protect\citeauthoryear{Wang, Wang and
  Zhang}{2020}]{wang2020asymmetric}
\begin{barticle}[author]
\bauthor{\bsnm{Wang},~\bfnm{G.}\binits{G.}},
  \bauthor{\bsnm{Wang},~\bfnm{Y.}\binits{Y.}} \AND
  \bauthor{\bsnm{Zhang},~\bfnm{S.}\binits{S.}}
(\byear{2020}).
\btitle{An asymmetric information mean-field type linear--quadratic stochastic
  {S}tackelberg differential game with one leader and two followers}.
\bjournal{Optimal Control Applications and Methods}
\bvolume{41}
\bpages{1034--1051}.
\end{barticle}
\endbibitem

\bibitem[\protect\citeauthoryear{Wiesemann
  et~al.}{2013}]{wiesemann2013pessimistic}
\begin{barticle}[author]
\bauthor{\bsnm{Wiesemann},~\bfnm{W.}\binits{W.}},
  \bauthor{\bsnm{Tsoukalas},~\bfnm{A.}\binits{A.}},
  \bauthor{\bsnm{Kleniati},~\bfnm{P.~M.}\binits{P.~M.}} \AND
  \bauthor{\bsnm{Rustem},~\bfnm{B.}\binits{B.}}
(\byear{2013}).
\btitle{Pessimistic bilevel optimization}.
\bjournal{SIAM Journal on Optimization}
\bvolume{23}
\bpages{353--380}.
\end{barticle}
\endbibitem

\bibitem[\protect\citeauthoryear{Wu}{2013}]{wu2013general}
\begin{barticle}[author]
\bauthor{\bsnm{Wu},~\bfnm{Z.}\binits{Z.}}
(\byear{2013}).
\btitle{A general maximum principle for optimal control of forward--backward
  stochastic systems}.
\bjournal{Automatica}
\bvolume{49}
\bpages{1473--1480}.
\end{barticle}
\endbibitem

\bibitem[\protect\citeauthoryear{Yong}{2002}]{yong2002leader}
\begin{barticle}[author]
\bauthor{\bsnm{Yong},~\bfnm{J.}\binits{J.}}
(\byear{2002}).
\btitle{A leader--follower stochastic linear quadratic differential game}.
\bjournal{SIAM Journal on Control and Optimization}
\bvolume{41}
\bpages{1015--1041}.
\end{barticle}
\endbibitem

\bibitem[\protect\citeauthoryear{Yong}{2010}]{yong2010optimality}
\begin{barticle}[author]
\bauthor{\bsnm{Yong},~\bfnm{J.}\binits{J.}}
(\byear{2010}).
\btitle{Optimality variational principle for controlled forward--backward
  stochastic differential equations with mixed initial--terminal conditions}.
\bjournal{SIAM Journal on Control and Optimization}
\bvolume{48}
\bpages{3675--4179}.
\end{barticle}
\endbibitem

\bibitem[\protect\citeauthoryear{Zemkoho}{2016}]{zemkoho2016solving}
\begin{barticle}[author]
\bauthor{\bsnm{Zemkoho},~\bfnm{A.~B.}\binits{A.~B.}}
(\byear{2016}).
\btitle{Solving ill-posed bilevel programs}.
\bjournal{Set-Valued and Variational Analysis}
\bvolume{24}
\bpages{423--448}.
\end{barticle}
\endbibitem

\bibitem[\protect\citeauthoryear{Zhang}{2017}]{zhang2017backward}
\begin{bbook}[author]
\bauthor{\bsnm{Zhang},~\bfnm{J.}\binits{J.}}
(\byear{2017}).
\btitle{Backward stochastic differential equations---from linear to fully
  nonlinear theory}.
\bseries{Probability theory and stochastic modelling}
\bvolume{86}.
\bpublisher{Springer-Verlag New York}.
\end{bbook}
\endbibitem

\bibitem[\protect\citeauthoryear{Zheng and Shi}{2020}]{zheng2020stackelberg}
\begin{barticle}[author]
\bauthor{\bsnm{Zheng},~\bfnm{Y.}\binits{Y.}} \AND
  \bauthor{\bsnm{Shi},~\bfnm{J.}\binits{J.}}
(\byear{2020}).
\btitle{A {S}tackelberg game of backward stochastic differential equations with
  applications}.
\bjournal{Dynamic Games and Applications}
\bvolume{10}
\bpages{968--992}.
\end{barticle}
\endbibitem

\bibitem[\protect\citeauthoryear{Zheng and Shi}{2021}]{zheng2021stackelberg}
\begin{barticle}[author]
\bauthor{\bsnm{Zheng},~\bfnm{Y.}\binits{Y.}} \AND
  \bauthor{\bsnm{Shi},~\bfnm{J.}\binits{J.}}
(\byear{2021}).
\btitle{A {S}tackelberg game of backward stochastic differential equations with
  partial information}.
\bjournal{Mathematical Control and Related Fields}
\bvolume{11}
\bpages{797--828}.
\end{barticle}
\endbibitem

\bibitem[\protect\citeauthoryear{Zheng and Shi}{2022a}]{zheng2022linear}
\begin{barticle}[author]
\bauthor{\bsnm{Zheng},~\bfnm{Y.}\binits{Y.}} \AND
  \bauthor{\bsnm{Shi},~\bfnm{J.}\binits{J.}}
(\byear{2022}a).
\btitle{A linear--quadratic partially observed {S}tackelberg stochastic
  differential game with application}.
\bjournal{Applied Mathematics and Computation}
\bvolume{420}
\bpages{1--22}.
\end{barticle}
\endbibitem

\bibitem[\protect\citeauthoryear{Zheng and Shi}{2022b}]{zheng2022stackelberg}
\begin{barticle}[author]
\bauthor{\bsnm{Zheng},~\bfnm{Y.}\binits{Y.}} \AND
  \bauthor{\bsnm{Shi},~\bfnm{J.}\binits{J.}}
(\byear{2022}b).
\btitle{Stackelberg stochastic differential game with asymmetric noisy
  observations}.
\bjournal{International Journal of Control}
\bvolume{95}
\bpages{2510--2530}.
\end{barticle}
\endbibitem

\end{thebibliography}

\begin{appendix}
\section{ACLM solution: additional results and proofs}\label{app:example}

\begin{lemma} \label{lemma-ACLM}
	For $k>0$, consider the closed-loop memoryless strategy $a_k \in \Ac$ defined for all $t \in [0,T]$ by
	\begin{align*}
		a_k(t,X_t) \coloneqq   \frac{1}{c_{\rm L}} + k (X_t - X_t^\star), \; \text{\rm where} \;
		X_t^\star = x_0 + \dfrac{t}{c_{\rm L}} +  \int_0^t   \Pi_{B} \bigg( \frac{\mathrm{e}^{k(T-s)}}{c_{\rm F}} \bigg)  \d s + \sigma W_t.
	\end{align*}
	{Assume that $a_\circ \geq \frac{1}{c_{\rm L}} + \frac{(b_\circ c_{\rm F} -1)^2}{2c_{\rm F}}$ and define $\bar{K}\coloneqq \frac{1}{T} \log\big( \frac{1}{b_\circ} \big( a_\circ + \frac{1}{c_{\rm L}} +\frac{1}{2c_{\rm F}} \big) + \frac{b_\circ c_{\rm F}}2  \big) $}. Then, for a fixed  $k \in (0, \bar{K}]$ and the associated strategy $a_k$, we have that {$a_k$ is the solution to the {\rm ACLM-$k$} problem} and the leader obtains the following reward, which is higher than his value in the {\rm AOL} information case
	\[
	f(k) \coloneq  x_0 + \dfrac{T}{2 c_{\rm L}} +  b_\circ  t_\circ^k + \dfrac{1}{k c_{\rm F}} ( \mathrm{e}^{k(T-t_\circ^k)} -1),
	\]
	where $t_\circ^k \coloneqq( T - \frac{1}{k} \log(b_\circ  c_{\rm F}) )^+$.
\end{lemma}

\begin{proof}[Proof of Lemma \ref{lemma-ACLM}]
	
	(i) To provide the main intuition, suppose first that the leader's actions are unrestricted, that is $A=\R$. This is the usual setting for the ACLM problems that are solved explicitly in the literature. Then, the leader announces her strategy $\alpha_k \in \Ac$ defined by
	\[
	a_k(t,X_t) =  \frac{1}{c_{\rm L}} + k (X_t - X_t^\star), \; t \in [0,T].
	\]
	
	Then, the follower's optimisation problem originally defined in \eqref{eq:pb-follower-example} is the following
	\begin{align}\label{eq:lemma-follower-aclm}
		V_{\rm F} (\alpha_k) &\coloneqq  \sup_{\beta \in \Bc}  \E^\P \bigg[ X_T - \dfrac{c_{\rm F}}2  \int_0^T \beta_t^2 \d t \bigg], \\
		\text{subject to} \; 
		\d X_t &= \bigg( \frac{1}{c_{\rm L}} + k(X_t - X_t^\star) +  \beta_t \bigg) \d t + \sigma \d W_t, \; t \in [0,T]. \nonumber
	\end{align}
	As described in \Cref{sss:ACLM}, one can use the stochastic maximum principle to obtain, after solving the appropriate FBSDE system, that the optimal response of the follower is given by
	\begin{equation}\label{eq:lemma-optimal-beta-aclm}
		\beta^\star_t = \Pi_{[0,b_\circ]} \bigg( \frac{\erm^{k(T-t)} }{c_{\rm F}} \bigg), \; t\in[0,T].
	\end{equation}
	Alternatively, one can solve this stochastic control problem in a more straightforward way, by noticing that the follower's problem defined above by \eqref{eq:lemma-follower-aclm} can be rewritten as follows
	\begin{align*}
		V_{\rm F} (\alpha_k) & = \sup_{\beta \in \Bc}  \bigg\{\E^\P \bigg[ X^\star_T +  \widetilde X_T  - \dfrac{c_{\rm F}}2  \int_0^T \beta_t^2 \d t \bigg] \bigg\}\\
		& = x_0 + \dfrac{T}{c_{\rm L}} + b_\circ  t_\circ^k + \dfrac{1}{k c_{\rm F}} ( \mathrm{e}^{k(T-t_\circ^k)} -1) + \sup_{\beta \in \Bc} \bigg\{ \E^\P \bigg[ \sigma W_T + \widetilde X_T  - \dfrac{c_{\rm F}}2  \int_0^T \beta_t^2 \d t \bigg]\bigg\},
	\end{align*}
	where the process $\widetilde X \coloneqq  X-X^\star$, corresponding to the only state variable of the previous control problem, satisfies the following controlled ODE
	\begin{align}\label{eq:lemma-dynamics-XandXstar}
		\drm \widetilde X_t = 
		\bigg( k \widetilde X_t +  \beta_t - \Pi_{[0,b_\circ]} \bigg( \frac{\erm^{k(T-t)} }{c_{\rm F}} \bigg)
		\bigg) \drm t, \; t \in [0,T], \; \widetilde X_0 =0,
	\end{align}
	whose solution is given by
	\begin{equation} \label{eq:explicit-x-tilde}
		\widetilde X_t\coloneqq  \erm^{kt} \int_0^t \erm^{-ks} \bigg( \beta_s -  \Pi_{[0,b_\circ]} \bigg( \frac{\erm^{k(T-s)} }{c_{\rm F}} \bigg) \bigg) \d s
		= \erm^{kt} \int_0^t \erm^{-ks} \beta_s \d s
		- L_t^k, \; \forall t\in[0,T],
	\end{equation}
	with the process
	\begin{align*}
		L_t^k &\coloneqq  \int_0^t \erm^{k(t-s)} \Pi_{[0,b_\circ]} \bigg( \frac{\erm^{k(T-s)} }{c_{\rm F}} \bigg) \d s \\
		&= \begin{cases} 
			\displaystyle \frac{b_\circ }{k} (\mathrm{e}^{kt} -1 ), \; t \in[0,t_\circ^k]\\[.3cm]
			\displaystyle \frac{b_\circ  \mathrm{e}^{kt}}{k} (1 - \mathrm{e}^{-kt_\circ^k}) + \frac{1}{2k c_{\rm F}}(\mathrm{e}^{k(T-t_\circ^k)} \mathrm{e}^{k(t-t_\circ^k)} - \mathrm{e}^{k(T-t)} ) , \; t\in[t_\circ^k,T].
		\end{cases}
	\end{align*}
	The follower's optimisation problem thus becomes
	\begin{align*}
		V_{\rm F} (\alpha_k)
		&= x_0 + \dfrac{T}{c_{\rm L}} + b_\circ  t_\circ^k + \dfrac{\mathrm{e}^{k(T-t_\circ^k)} -1}{k c_{\rm F}}  + \sup_{\beta \in \Bc} \bigg\{ \erm^{kT} \int_0^T \erm^{-kt} \beta_t \d t - L_T^k
		- \dfrac{c_{\rm F}}2  \int_0^T \beta_t^2 \d t \bigg\}\\
		&= x_0 + \dfrac{T}{c_{\rm L}} +  b_\circ  t_\circ^k + \dfrac{ \mathrm{e}^{k(T-t_\circ^k)} -1}{k c_{\rm F}}  - L_T^k
		+ \sup_{\beta \in \Bc}\bigg\{ \int_0^T \bigg( \erm^{k(T-t)} \beta_t - \dfrac{c_{\rm F}}2 \beta_t^2 \bigg)\d t\bigg\}.
	\end{align*}
	The optimal effort $\beta^\star$ introduced above in \eqref{eq:lemma-optimal-beta-aclm} is deduced by pointwise optimisation. The value of the follower is then
		\begin{align*}
			V_{\rm F} (\alpha_k) &= x_0 + \dfrac{T}{c_{\rm L}} +  b_\circ  t_\circ^k + \dfrac{1}{k c_{\rm F}} ( \mathrm{e}^{k(T-t_\circ^k)} -1)  
			- \dfrac{c_{\rm F}}2    \int_0^T \Pi_{[0,b_\circ]} \bigg( \frac{\erm^{k(T-t)} }{c_{\rm F}} \bigg)^2\d t \\
			& = x_0 + \dfrac{T}{c_{\rm L}} +  b_\circ  t_\circ^k + \dfrac{1}{k c_{\rm F}} ( \mathrm{e}^{k(T-t_\circ^k)} -1)   - \dfrac{c_{\rm F}}2 b_\circ^2 t_\circ^k - \frac{1}{4kc_{\rm F}} (\mathrm{e}^{2k(T-t_\circ^k)} - 1). 
		\end{align*}
	
	\smallskip
	
	Remark that for the optimal control $\beta^\star$, the controlled ODE \eqref{eq:lemma-dynamics-XandXstar} simplifies, and gives the trivial solution $\widetilde X_t = 0$, \textit{i.e.} $X_t = X_t^\star$, for all $t \in [0,T]$.
	In other words, the best choice for the follower is to choose $\beta$ so that the process $X$ coincides with the process $X^\star$. Given the follower's optimal response, the objective value of the leader for the strategy $\alpha_k$ simplifies to
	\begin{align*}
		\E^\P \bigg[ X_T - \dfrac{c_{\rm L}}2  \int_0^T \big( a_k(t,X_t) \big)^2 \d t \bigg]
		&= \E^\P \bigg[X_T^\star - \dfrac{c_{\rm L}}2  \int_0^T \bigg( \frac{1}{c_{\rm L}}  \bigg)^2 \d t \bigg] \\
		&= x_0 + \dfrac{T}{2c_{\rm L}} +  b_\circ  t_\circ^k + \dfrac{1}{k c_{\rm F}} ( \mathrm{e}^{k(T-t_\circ^k)} -1) = f(k).
	\end{align*}

	\smallskip
	$(ii)$  We now show that $a_k(t,X_t)$ provides an admissible strategy for problem ACLM-$k$, for $k\in (0,\bar{K}]$. {Notice that we can write $a_k(t,X_t)=\hat a^1_t + \hat a^2_t X_t$, with $\hat a^2 \equiv k$ and $\hat a^1 = \frac{1}{c_{\rm L}} - k X^\star$. We then have to show that $(\hat a^1,\hat a^2)\in\Ac_k^2$, which means that $a_k(t,X_t)$ takes values in $A=[-a_\circ,a_\circ]$}. From \eqref{eq:explicit-x-tilde}, for any strategy $\beta\in\Bc$ of the follower we have that
	\[
	-k L_T^k \leq -k L_t^k \leq  k \tilde X_t \leq \int_0^T \mathrm{e}^{k(t-s)}  \bigg( b_\circ  -  \Pi_{[0,b_\circ]} \bigg( \frac{\erm^{k(T-s)} }{c_{\rm F}} \bigg) \bigg)  \leq   b_\circ (\mathrm{e}^{kT} - 1) - k L_T^k.
	\]
	By replacing the value of $L_T^k$, we obtain that the term in the right side is increasing in $k$ and 
	\[
	\frac{1}{c_{\rm L}} + k \tilde X_t \leq a_\circ , \; \forall t\in[0,T] \iff \frac{1}{c_{\rm L}}  -b_\circ  + b_\circ^2 c_{\rm F}  -\frac{1}{2c_{\rm F}}(b_\circ^2 c_{\rm F}^2 - 1) \leq a_\circ.
	\]
	This condition is equivalent to the assumption in this lemma. Similarly, the term in the left side is decreasing in $k$ and we have
	\[
	\frac{1}{c_{\rm L}} + k \tilde X_t \geq - a_\circ ,\; \forall t\in[0,T] \iff  - \frac{1}{c_{\rm L}}  + b_\circ  \mathrm{e}^{kT} - b_\circ^2 c_{\rm F} +\frac{1}{2c_{\rm F}}(b_\circ^2 c_{\rm F}^2 - 1) \leq a_\circ,
	\]
	which holds because $k \leq \bar{K}$. We conclude that $(\hat a^1,\hat a^2) \in\Ac^2_k$.

	\smallskip
	$(iii)$ We now show that the value of the ACLM-$k$ is equal to $f(k)$ and therefore $(\hat a^1,\hat a^2) \in\Ac^2_k$ is optimal. In \eqref{eq:ACLM-states-before-linearizing}, the solution to the linear BSDE is given by $Z^{\rm F} = 0$ and
	\[
	Y_t^{\rm F} = \E^\P \Big[ \mathrm{e}^{\int_t^T a_s^2 \d s}   \Big|  \Fc_t^W \Big] \leq \mathrm{e}^{k(T-t)}.
	\]
	Then, we can replace the dynamics of $X_t$ in the objective function and rewrite $\widetilde V_{\rm L}^k$ in order to find an upper bound
	\begin{align*}
		\widetilde V_{\rm L}^k & = x_0 + \sup_{( a^{\smallfont1}, a^{\smallfont2})\in\Ac^2_k}\bigg\{ \E^\P \bigg[ \int_0^T \bigg( a^{2}_t X_t +a^{1}_t + \Pi_B \bigg( \dfrac{Y_t^{\rm F}}{c_{\rm F}} \bigg)  - \dfrac{c_{\rm L}}2   \big( a^{2}_t X_t + a^{1}_t \big)^2 \bigg) \d t \bigg]\bigg\} \\
		& \leq x_0 + \sup_{( a^{\smallfont1}, a^{\smallfont2})\in\Ac^2_k} \bigg\{\E^\P \bigg[ \int_0^T \bigg( a^{2}_t X_t +a^{1}_t + \Pi_B \bigg( \dfrac{  \mathrm{e}^{k(T-t)} }{c_{\rm F}} \bigg)  - \dfrac{c_{\rm L}}2   \big( a^{2}_t X_t + a^{1}_t \big)^2 \bigg) \d t \bigg] \bigg\}\\
		& = x_0 + \int_0^T \bigg( \frac{1}{2 c_{\rm L}} + \Pi_B \bigg( \dfrac{  \mathrm{e}^{k(T-t)} }{c_{\rm F}} \bigg)  \bigg) \d t  =  x_0 + \frac{T}{2 c_{\rm L}} + b_\circ t_\circ^k  + \frac{1}{k c_{\rm F} } (\mathrm{e}^{k(T-t_\circ^k)}-1) = f(k).
	\end{align*}
	Since $(\hat a^1,\hat a^2)$ attains the upper bound, it is optimal.
	
	\smallskip
	$(iv)$ Finally, notice that $f$ is an increasing function of $k$, and that its limit when $k$ goes to $0$ is given by
	\begin{align*}
		\lim_{k \rightarrow 0} f(k) = x_0 + \dfrac{T}{2c_{\rm L}} + \lim_{k \rightarrow 0} \dfrac{\erm^{kT} -1}{k c_{\rm F}}
		= x_0 + \bigg( \dfrac{1}{2c_{\rm L}} + \dfrac{1}{c_{\rm F}} \bigg) T.
	\end{align*}
	As this value corresponds to the leader's value function in the AOL cases, we \textcolor{black}{conclude that the value of the ACLM-$k$ problem is higher than the AOL for $k>0$}. Similarly, we have
	\begin{align*}
		\lim_{k \rightarrow 0} V_{\rm F} (\alpha_k)
		&= x_0 + \bigg( \dfrac{1}{c_{\rm L}} + \dfrac{1}{2c_{\rm F}} \bigg) T. 
	\end{align*}

	\qedhere
\end{proof}

\section{Functional spaces} \label{sec:funcspace}

We introduce the spaces used in this paper, by following \cite{possamai2018stochastic}. Let $(t,x)\in [0,T]\times \Omega$, $(\mathcal P(t,x))_{t\in [0,T]\times x\in \Omega}$ be a family of sets of probability measures on $(\Omega,\mathcal F_T)$. In this section, we denote by $\mathbb X\coloneqq  (\mathcal X_s)_{s\in [0,T]}$ a general filtration on $(\Omega,\mathcal F_T)$. Let $p\geq 1$, $\mathbb P\in \mathcal P(t,x)$ and $\mathbb X_\mathbb P$ the usual $\mathbb P$-augmented filtration associated with $\mathbb X$.
\begin{itemize}
	\item $\mathbb H^p_{t,x}(\mathbb X,\mathbb P)$ (resp. $\mathbb H^p_{t,x}(\mathbb X,\mathcal P)$) denotes the spaces of $\mathbb X$-predictable $\mathbb R^d$-valued processes $Z$ such that
	\[
	\| Z\|_{\mathbb H^\smallfont{p}_{\smallfont{t}\smallfont{,}\smallfont{x}}(\mathbb X, \mathbb P)}^p\coloneqq \mathbb E^\mathbb P\bigg[\bigg( \int_t^T \| \widehat{\sigma}_s^\t Z_s\|^2 \mathrm{d}s\bigg)^\frac p2 \bigg] <+\infty,
	\;\; \bigg(\text{resp. }
	\| Z\|_{\mathbb H^\smallfont{p}_{\smallfont{t}\smallfont{,}\smallfont{x}} (\mathbb X,\mathcal P)   }^p\coloneqq \sup_{\mathbb P\in \mathcal P(t,x)} \| Z\|_{\mathbb H^\smallfont{p}_{\smallfont{t}\smallfont{,}\smallfont{x}}(\mathbb X,\mathbb P)}^p <+\infty 
	\bigg). 
	\]
	\item $\mathbb S^p_{t,x}(\mathbb X,\mathbb P)$ (resp. $\mathbb S^p_{t,x}(\mathbb X,\mathcal P)$) denotes the spaces of $\mathbb X$--progressively measurable $\mathbb R$-valued processes $Y$ such that
	\[
	\| Y\|_{\mathbb S^\smallfont{p}_{\smallfont{t}\smallfont{,}\smallfont{x}}(\mathbb X, \mathbb P)}^p\coloneqq \mathbb E^\mathbb P\bigg[\sup_{s\in [t,T]} |Y_s|^p \bigg] <+\infty, 
	\;\; \bigg(\text{resp. }
	\| Y\|_{\mathbb S^\smallfont{p}_{\smallfont{t}\smallfont{,}\smallfont{x}} (\mathbb X,\mathcal P)}^p\coloneqq \sup_{\mathbb P\in \mathcal P(t,x)} \| Y\|_{\mathbb S^p_{t,x}(\mathbb X,\mathbb P)}^p <+\infty
	\bigg). 
	\]
	\item $\mathbb I^p_{t,x}(\mathbb X,\mathbb P)$ (resp. $\mathbb I^p_{t,x}(\mathbb X,\mathcal P)$) denotes the spaces of $\mathbb X$-optional $\mathbb R$-valued processes $K$ with $\P\as$ c\`adl\`ag and non-decreasing paths on $[t,T]$ with $K_t=0,\; \mathbb P\text{\rm--a.s.}$ and
	\[
	\| K\|_{\mathbb I^\smallfont{p}_{\smallfont{t}\smallfont{,}\smallfont{x}}(\mathbb X, \mathbb P)}^p\coloneqq \mathbb E^\mathbb P[K_T^p ] <+\infty,
	\;\; \bigg(\text{resp. }
	\| K\|_{\mathbb I^\smallfont{p}_{\smallfont{t}\smallfont{,}\smallfont{x}}(\mathbb X, \mathcal P)}^p\coloneqq  \sup_{\mathbb P\in \mathcal P(t,x)} \|K \|^p_{\mathbb I^\smallfont{p}_{\smallfont{t}\smallfont{,}\smallfont{x}}(\mathbb X, \mathbb P)}<+\infty 
	\bigg). 
	\]
	
	\item $\mathbb G^p_{t,x}(\mathbb X,\mathbb P)$ denotes the spaces of $\mathbb X$-predictable $\S^\xdim$-valued processes $\Gamma$ such that
	\[
	\| \Gamma \|_{\mathbb G^\smallfont{p}_{\smallfont{t}\smallfont{,}\smallfont{x}}(\mathbb X, \mathbb P)}^p \coloneqq \E^\P\bigg[ \bigg(\int_t^T \big \| \sigmah^2_s\Gamma_s \big\|^{2}\d s\bigg)^{\frac p2}\bigg]<+\infty. 
	\]

\end{itemize}

When $t=0$, we simplify the previous notations by omitting the dependence on both $t$ and $x$.

\section{Boundaries PDEs: comparison and verification}\label{sec.appen.comp.boundaries}

We conduct the analysis for $w^-$, the argument for $w^+$ being analogous. 
We start by establishing a comparison result for viscosity solutions to \eqref{eq:pde-inf}. Let us recall that $w^-$ is a discontinuous viscosity solution of such an equation. Moreover, we remind the reader that the \Cref{assumption.strongtarget} is in place.

\begin{lemma}\label{lemma-comparison}
	Let $u$ and $v$ be respectively an upper--semi-continuous viscosity sub-solution and a lower--semi-continuous viscosity super-solution of \eqref{eq:pde-inf}, such that for $\varphi\in \{u, v\}$ and some $C>0$, $\varphi (y)\leq C(1+ \|y\|)$, $y\in [0,T]\times\R^d$. 
	If, $u(T,x)\leq v(T,x)$, $x \in \R^\xdim$, then $u\leq v$ on $\Oc\coloneqq (0,T)\times\R^d$.
	\begin{proof}
		{\bf Step 1.} Fix postive constants $\alpha$, $\beta$, $\eta$, and $\eps$, and define $\phi (t,x,y)\coloneqq u^\eta(t,x)-v(t,y)$, where $u^\eta(t,x)\coloneqq u(t,x)-\frac\eta{t}, \; (t,x)\in \Oc$. Note that since $\frac{\partial}{\partial t}  ( -\eta t^{-1} )=\eta t^{-2}>0$, $u^\eta$ is a viscosity sub-solution of \eqref{eq:pde-inf} in $\Oc$. Define
		\[
		\psi_{\alpha,\beta,\eps}(t,x,y)\coloneqq \alpha |x-y|^2/2+\eps|x|^2+\eps|y|^2-\beta(t-T).
		\]
		Let $M_{\alpha,\beta,\eps}\coloneqq \sup_{(t,x,y)\in (0,T]\times \R^\smallfont{\xdim}\times \R^\smallfont{\xdim}} \big\{\phi-\psi_{\alpha,\beta,\eps}\big\}(t,x,y)=(\phi- \psi_{\alpha,\beta,\eps})(t_{\alpha,\beta,\eps},x_{\alpha,\beta,\eps},y_{\alpha,\beta,\eps}),$ for $(t_{\alpha,\beta,\eps},x_{\alpha,\beta,\eps},y_{\alpha,\beta,\eps})\in (0,T]\times \R^\xdim\times \R^\xdim$ thanks to the upper--semi-continuity of $u^\eta-v$, the growth assumptions on $u$ and $v$ and that of $\beta(t-T)-\eta t^{-1}$. 
		Moreover, we have that $-\infty<\lim_{\alpha\rightarrow \infty} M_{\alpha,\beta,\eps} <\infty$, meaning that the supremum is attained on a compact subset of $(0,T]\times \R^\xdim\times \R^\xdim$. 
		Consequently, there is a subsequence $(t_{n}^{\beta,\eps},x_{n}^{\beta,\eps},y_{n}^{\beta,\eps})\coloneqq (t_{\alpha_\smallfont{n},\beta,\eps},x_{\alpha_\smallfont{n},\beta,\eps},y_{\alpha_\smallfont{n},\beta,\eps})$ that converges to some $(\hat t^{\beta,\eps},\hat x^{\beta,\eps},\hat y^{\beta,\eps})$. 
		It then follows from \cite[Proposition 3.7]{crandall1992user} that 
		\begin{align}\label{eq.ci}
			\begin{split}
			\hat x^{\beta,\eps} &=\hat y^{\beta,\eps}, \; \lim_{n\to\infty} \alpha_n|x^{\beta,\eps}_n-y^{\beta,\eps}_n|^2=0, \\ M_{\beta,\eps} &\coloneqq \lim_{n\longrightarrow\infty} M_{\alpha_\smallfont{n},\beta,\eps}=\sup_{(t,x) \in\overline \Oc} (u^\eta-v)(t,x)-2\eps |\hat x^\eps|+\beta (\hat t^{\beta,\eps}-T).
			\end{split}
		\end{align}
		{\bf Step 2.} To prove the statement, as it is standard in the literature, let us assume by contradiction that there is $(t_o,x_o)\in \Oc$ such that $\gamma_o\coloneqq (u-v)(t_o,x_o)>0$. 
		We claim that there are positive $\beta_o$, $\eta_o$, and $\eps_o$ such that for any $\beta_o\geq \beta>0$, $\eta_o\geq \eta>0$, $\eps_o\geq \eps>0$, $(t_{n}^{\beta,\eps},x_{n}^{\beta,\eps},y_{n}^{\beta,\eps})$ is a local maximiser of $\phi(t,x,y) - \psi_{\alpha_\smallfont{n},\beta,\eps}(t,x,y)$ on $(0,T)\times\Kc^2$ for some $\Kc\subseteq\R^\xdim$ compact. 
		We first note that the existence of $\Kc$ 
		is clear since the supremum is attained on a compact set. It remains to show that $t_n^{\beta,\eps}<T$ for all $n\in \N$.
		
		\smallskip
		
		Suppose by contradiction that $t_n^{\beta,\eps}=T$ for some $n$. Thanks to the first step, for any positive $\beta$, $\eps$, and $\eta$ we have that
		\begin{align*}
			\gamma_o-\frac{\eta}{t_o} +\beta(t_o-T)-2\eps|x_o|^2\leq M_{\alpha_\smallfont{n},\beta,\eps} \leq -\frac{\eta}T ,
		\end{align*}
		since 
		\begin{align*}
			M_{\alpha_\smallfont{n},\beta,\eps} =  \sup_{(x,y)\in\R^\smallfont{\xdim}\times\R^\smallfont{d}} \big\{ u(T,x) -v(T,y) - \alpha_n|x-y|^2/2-\eps|x|^2-\eps|y|^2\big\}-\frac{\eta}T,
		\end{align*}
		where the rightmost inequality follows from the assumption $u(T,x)\leq v(T,x)$, $x \in \R^\xdim$. Consequently
		\begin{align*}
			\gamma_o\leq \frac{\eta}{t_o} -\frac{\eta}T  + \beta(T-t_o)+2\eps|x_o|^2,
		\end{align*}
		so that for $\beta$, $\eps$, and $\eta$ sufficiently small, $\gamma_o$ is arbitrarily small which contradicts $\gamma_o>0$. This proves the claim.
		
		\smallskip
		{\bf Step 3.} In light of the second step, it follows from Crandall--Ishii's lemma for parabolic problems, \cite[Theorem 8.3]{crandall1992user}, applied to $u^\eta$ and $v$ that we can find $(q_n,\hat q_n)$, $q_n-\hat q_n=\partial_t \psi_{\alpha,\beta,\eps}(t,x,y)=-\beta$, and symmetric matrices $(X_n^{\beta,\eps},Y_n^{\beta,\eps})$ such that
		\begin{align*}
			\begin{split}
				\big(q_n, \alpha_n(x_n^{\beta,\eps}-y_n^{\beta,\eps})+\eps x_n^{\beta,\eps} ,X_n^{\beta,\eps}\big) &\in \overline \Pc^{1,2,+}u^\eta(x_n^{\beta,\eps}),\\ \big(\hat q_n, -\alpha_n(x_n^{\beta,\eps}-y_n^{\beta,\eps})+\eps  y_n^{\beta,\eps} ,Y_n^{\beta,\eps}\big) &\in \overline \Pc^{1,2,-} v(y_n^{\beta,\eps}),\; 
			\end{split}
		\end{align*}
		and, for $C_n\coloneqq \alpha_n\begin{pmatrix}
			I_{\xdim}  & -I_{\xdim}\\
			-I_{\xdim} & I_{\xdim}
		\end{pmatrix}+\eps I_{2\xdim} $, we have that
		\[
		-\bigg( \frac1\lambda + \| C_n \|\bigg) I_{2 \xdim }\leq \begin{pmatrix}
			X_n^{\beta,\eps} & 0\\
			0 & -Y_n^{\beta,\eps}
		\end{pmatrix}\leq C_n (I_{2 \xdim } +\lambda C_n ), \text{ for all }\lambda >0.
		\]
		Taking $\lambda= (\alpha_n+\eps)^{-1} $ leads to
		
		\begin{align}\label{eq.CI.lemma}
			-\Big( \alpha_n+\eps + \| C_n \|\Big) I_{2 \xdim }\leq \begin{pmatrix}
				X_n^{\beta,\eps} & 0\\
				0 & -Y_n^{\beta,\eps}
			\end{pmatrix}\leq 3 \alpha_n \begin{pmatrix}
				I_{\xdim}  & -I_{\xdim}\\
				-I_{\xdim} & I_{\xdim}
			\end{pmatrix} +2 \eps I_{2d} .
		\end{align}
		{\bf Step 4.} With the notation $(t_{n} ,x_{n} ,y_{n} )\coloneqq (t_{n}^{\beta,\eps},x_{n}^{\beta,\eps},y_{n}^{\beta,\eps})$, $p_n^x\coloneqq \alpha_n(x_{n} -y_{n} )-\eps x_n $, $p_n^y\coloneqq \alpha_n(x_{n} -y_{n})-\eps y_n$, under the above assumptions we claim that
		there exists a universal constant $C>0$ such that
		\begin{align*}
			\begin{split}
				&\ H^-(t_{n} ,y_{n} ,p_n^y,Q_2)-  H^-(t_{n} ,x_{n} ,p_n^x,Q_1) \\
				\leq &\ C\big(1+\eps^2\|x_n\|+\eps^2\|y_n\|+\eps \big)\big ( \alpha_n\|x_{n} -y_{n} \|^2 + \|x_n-y_n\| + \eps \big)
			\end{split}
		\end{align*}
		for matrices $Q_1$, $Q_2$ satisfying \eqref{eq.CI.lemma}. 
		We consider each term in $h^{\rm b}$ separately, recall \eqref{eq:h_b} and \eqref{eq.optmatchvol}.\footnote{The following estimates hold for arbitrary, but fixed, $(a,\gamma,b^\star)$.} 
		
		\smallskip
		Letting $\Sigma^x\coloneqq \sigma^{b^\smallfont{\star}\!} (t_n,x_n,a,p_n^x,\gamma),\Sigma^y\coloneqq \sigma^{b^\smallfont{\star}\!} (t_n,y_n,a,p_n^y,\gamma)$, note that there is $C>0$ such that
		\begin{align*}
			&\text{Tr}[ (\sigma\sigma^\t) (t_n,y_n,a,b^\star(t,y_n,p_n^y,\gamma,a))Q_2]- \text{Tr}[\sigma\sigma^\t (t_n,x_n,a,b^\star(t,x_n,p_n^x,\gamma,a))Q_1]\\
			&=  \text{Tr}\bigg[  \begin{pmatrix} \Sigma^x{\Sigma^x}^\top & \Sigma^x{\Sigma^y}^\top\\
				\Sigma^y {\Sigma^x}^\top &  \Sigma^y{\Sigma^y}^\top   \end{pmatrix}  \begin{pmatrix} Q_2 & 0\\
				0 & - Q_1  \end{pmatrix} \bigg]\\
			&\leq 3\alpha_n  \text{Tr}\bigg[ \begin{pmatrix} \Sigma^x{\Sigma^x}^\top & \Sigma^x{\Sigma^y}^\top\\
				\Sigma^y {\Sigma^x}^\top &  \Sigma^y{\Sigma^y}^\top   \end{pmatrix}  \begin{pmatrix} I_{ \xdim } & -I_{ \xdim}\\
				- I_{ \xdim } & I_{ \xdim }  \end{pmatrix} \bigg]+2\eps  \text{Tr}\bigg[ \begin{pmatrix} \Sigma^x{\Sigma^x}^\top & \Sigma^x{\Sigma^y}^\top\\
				\Sigma^y {\Sigma^x}^\top &  \Sigma^y{\Sigma^y}^\top   \end{pmatrix}  I_{ 2\xdim }   \bigg]\\
			&=3\alpha_n \text{Tr}\big[ (\Sigma^x-\Sigma^y)(\Sigma^x-\Sigma^y)^\top\big]+2\eps \text{Tr} \big[\Sigma^x {\Sigma^x}^\t+\Sigma^y {\Sigma^y} ^\top \big]\\
			&=3\alpha_n \| \Sigma^x-\Sigma^y\|^2+2\eps \text{Tr} \big[\Sigma^x {\Sigma^x}^\t+\Sigma^y {\Sigma^y} ^\top \big]\\
			& \leq 3\alpha_n \|\sigma^{b^\smallfont{\star}\!} (t_n,x_n,p_n^x,\gamma,a)-\sigma^{b^\smallfont{\star}\!} (t_n,y_n,p_n^y,\gamma,a)\|^2+4\eps C_{\sigma}  \leq   C\big( (1+\eps)\alpha_n\|x_n-y_n\|^2+\eps\big),
		\end{align*}
		where the first inequality follows from the right-hand side of \eqref{eq.CI.lemma}, $C_{\sigma}$ denotes the bound on $\sigma\sigma^\t$ given by  \Cref{Assumption.data}, and the last inequality follows from \Cref{assumption.strongtarget}. Similarly, note that there is a constant $C>0$ such that
		\begin{align*}
			& c^{b^\smallfont{\star}\!}(t_n,y_n ,p_n^y,\gamma,a)-  c^{b^\smallfont{\star}\!} (t_n,x_n ,p_n^x,\gamma,a) \\
			\leq &\ C \big( \| x_n-y_n\| + \| b^\star(t_n,y_n,p_n^y,\gamma,a) - b^\star(t_n,x_n,p_n^x,\gamma,a) \| \big)\\
			\leq &\ C \big( \| x_n-y_n\|+\| p_n^y -p_n^x \|\big)\leq C \big( \| x_n-y_n\|+ \eps \big)  ,
		\end{align*}
		and
		\begin{align*}
			& \sigma   \lambda^{b^\smallfont{\star}\!}  (t_n,x_n,p_n^x,\gamma,a) \cdot p_n^x
			-\sigma   \lambda^{b^\smallfont{\star}\!}  (t_n,y_n,p_n^y,\gamma,a)  \cdot p_n^y\\
			\leq &\ \| \sigma   \lambda^{b^\smallfont{\star}\!}  (t_n,x_n,p_n^x,\gamma,a)\|\|  p_n^x-p_n^y\|+ \| \sigma  \lambda^{b^\smallfont{\star}\!}  (t_n,x_n,p_n^x,\gamma,a)-\sigma  \lambda^{b^\smallfont{\star}\!} (t_n,y_n,p_n^y,\gamma,a)  \| \| p_n^y\| \\
			\leq &\ \eps C \|x_n-y_n\| + C\| p_n^y\| (1+\eps)  \|x_n-y_n\| \\
			\leq  &\ C (1+\eps+\eps^2\|y_n\| ) \big( \|x_n-y_n\|  +  \alpha_n \|x_n-y_n\|^2\big).
		\end{align*}
		The result follows from using these estimates back in the Hamiltonian.\smallskip
		
		{\bf Step 5.} We conclude. By {\bf Step 3} and the viscosity properties of $u^\eta$ and $v$, we have
		\[
		- q_n + H^-\big( t_n^{\beta,\eps},  x_n^{\beta,\eps}, \alpha_n(x_n^{\beta,\eps}-y_n^{\beta,\eps})-\eps x_n^{\beta,\eps} , X_n^{\beta,\eps}\big) \leq 0\leq - \hat q_n  + H^- \big(t_n^{\beta,\eps},   y_n^{\beta,\eps},  \alpha_n(x_n^{\beta,\eps}-y_n^{\beta,\eps})-\eps y_n^{\beta,\eps} , Y_n^{\beta,\eps}\big) .
		\] 
		Subtracting, we find from {\bf Step 4} that
		\begin{align*}
			\beta &=\hat q_n-q_n \\
			& \leq  H^- \big(t_n^{\beta,\eps},   y_n^{\beta,\eps},  \alpha_n(x_n^{\beta,\eps}-y_n^{\beta,\eps})-\eps y_n^{\beta,\eps} , Y_n^{\beta,\eps}\big) 
			-H^-\big( t_n^{\beta,\eps},  x_n^{\beta,\eps}, \alpha_n(x_n^{\beta,\eps}-y_n^{\beta,\eps})-\eps x_n^{\beta,\eps} , X_n^{\beta,\eps}\big)\\
			& \leq C (1+\eps^2 \|x_n^{\beta,\eps}\|+\eps^2\|y_n^{\beta,\eps}\|+\eps )\big ( \alpha_n \|x_{n}^{\beta,\eps} -y_{n}^{\beta,\eps} \|^2 + \|x_n^{\beta,\eps}-y_n^{\beta,\eps}\| + \eps \big).
		\end{align*}
		Passing to the limit $n\longrightarrow \infty$ and $\eps\longrightarrow 0$, thanks to \eqref{eq.ci}, we find that $\beta\leq 0$ which is a contradiction.
	\end{proof}
\end{lemma}

The next lemma proves, in particular, that the auxiliary value function satisfies the hypotheses of \Cref{lemma-comparison}.

\begin{lemma}\label{prop:bounded-aux-value}
	{Suppose the functions $H^+$ and $H^-$ are continuous}. The functions $w^-$ and $w^+$ from $[0,T]\times\R^d$ to $\R$ defined in \eqref{eq:auxiliary-w1} are bounded and continuous.
\end{lemma}

{For completing the last step in the verification result, we have assumed the continuity of the Hamiltonian functions. We remark that this assumption holds, for instance, if the optimisation over $\gamma$ in the definition of $H^+$ and $H^-$ can be reduced to a compact set, continuously with respect to $(t,x,p,Q)$.}

\begin{proof}[Proof of {\rm \Cref{prop:bounded-aux-value}}]
	{We prove the result for $w^-$, the other being analogous}. We first argue $w^-$ is bounded. Let $(t,x)\in[0,T]\times\R^d$ and $y> T \ell_c + \ell_g$. We claim that $(x,y)\in V_g(t)$. Indeed, taking the control $Z = 0$, $\Gamma = 0$ and any $(\alpha,b^\star)\in\Ac\times \Bc^\star$ we have
	\[
	Y_T^{t,x,y,\upsilon }= y - \int_t^T c^{b^\smallfont{\star}\!} \big(s,X_s^{t,x,\upsilon},Z_s,\hat\upsilon\big)\textrm{d}s \geq y - T \ell_c > \ell_g \geq g(X_T^{t,x,\upsilon}).
	\]
	That is $w^-(t,x) \leq T \ell_c + \ell_g$. 
	To obtain a lower bound take again $(t,x)\in[0,T]\times\R^d$ and $y < -T\ell_c -\ell_g $. Then, it is easy to check that for any $M\in\R$ and any $\upsilon \in \Cf$ the following process is an $(\F,\P)$--super-martingale 
	\[
	A_s\coloneqq  Y_s^{t,x,y,\upsilon} - s \ell_c +M, \; s\in[0,T].
	\]
	Thus, choosing $M= T \ell_c + \ell_g$, we have that $\E^\P[Y_T^{t,x,y,\upsilon} - T\ell_c + M] \leq y + M < 0$, which implies
	$\P[Y_T^{t,x,y,\upsilon} + \ell_g < 0 ] > 0$.
	Therefore, for any $\upsilon \in \Cf$
	\[
	\P\big[Y_T^{t,x,y,\upsilon} < g(X_T^{t,x,\upsilon})  \big] \geq \P\big[Y_T^{t,x,y,\upsilon} + \ell_g < 0 \big] > 0,
	\]
	which means that the pair $(x,y)\not\in V_g^1(t)$. Thus, $w^-(t,x) \geq - T \ell_c - \ell_g$.\smallskip
	
	Let us now prove the continuity. By \cite[Theorem 2.1]{bouchard2009stochastic}, $w^-$ is a discontinuous viscosity solution to PDE \eqref{eq:pde-sup} as long as we verify Assumption 2.1 therein. 
	Indeed, the continuity condition on the set $N(t,x,p)$ holds in our case given the explicit form that was obtained in \eqref{eq.optmatchvol}. 
	{Since $H^-$ is continuous}, the lower-- and upper--semi-continuous envelopes $w^-_\star$ and $w^{-,\star}$ are viscosity super-solution and sub-solution, respectively, of {\Cref{eq:pde-inf}}. From \cite[Theorem 2.2]{bouchard2009stochastic}, which in our case is not subject to the gradient constraints (see \cite[Remark 2.1]{bouchard2009stochastic}  and notice that in our setting their set \textbf{N}$^c$ is empty), we conclude that $w^{-,\star}(T,\cdot) \leq g \leq  w^-_\star(T,\cdot)$. Finally, from \Cref{lemma-comparison}, we have therefore that $w^{-,\star} \leq w^-_\star$ on $[0,T)\times\R^d$. Since the reverse inequality holds by definition, we conclude the equality of the semicontinuous envelopes and thus the continuity of $w^-$.
\end{proof}

\begin{proof}[Proof of {\rm \Cref{thm.aux.func}}]	
	The result is an immediate consequence of Lemmata \ref{lemma-comparison} and \ref{prop:bounded-aux-value}.
\end{proof}

\end{appendix}

\end{document}